\documentclass[a4paper,10pt]{article}
\usepackage[utf8]{inputenc}
\usepackage{amsthm}
\usepackage{amsmath}
\usepackage{amsfonts,amssymb,latexsym}
\usepackage{mathtools}
\usepackage{enumerate}
\usepackage{bbm}
\usepackage{relsize}
\usepackage{wasysym}
\usepackage{tikz}
\usepackage{ mathrsfs }
\usepackage[french,english]{babel}
\usepackage{forest}
\usepackage{pdflscape}
\usepackage[hyperindex,breaklinks]{hyperref}

\usepackage{subfiles} 

\newtheorem{thm}{Theorème}
\newtheorem{lem}{Lemme}

\newtheorem{defi}{Définition}
\newtheorem{prop}{Proposition}

\newcommand{\N}{{\mathbb N}}
\newcommand{\Z}{{\mathbb Z}}
\newcommand{\Q}{{\mathbb Q}}
\newcommand{\R}{{\mathbb R}}
\newcommand{\C}{{\mathbb C}}
\newcommand{\F}{{\mathbb F}}

\newcommand{\bo}{{\mathcal{O}}}

\newcommand{\RR}{{\mathcal R}}
\newcommand{\Na}{{\mathcal N_{a,\ell}}}
\newcommand{\Pa}{{\mathcal P_a}}
\newcommand{\Ma}{{\mathcal M_a}}
\newcommand{\ord}{{\rm ord}}

\newcommand{\Card}{{\rm Card}}
\newcommand{\Ker}{{\rm Ker}}
\newcommand{\Gal}{{\rm Gal}}

\newcommand{\id}{{\mathbbm{1}}}
\newcommand{\qroot}{{\mathfrak{R}}}

\newcommand{\Mod}[1]{\ (\mathrm{mod}\ #1)}
\newcommand{\lquote}{\text{\textquotedblleft}}
\newcommand{\rquote}{\text{\textquotedblright}}
\DeclareMathOperator{\Li}{\mathrm{Li}}

\title{Sur une généralisation de la conjecture d'Artin parmi les presque-premiers}
\author{Paul Péringuey\footnote{Department of Mathematics, University of British Columbia, Vancouver, British Columbia, Canada.\\ Pacific Institute for the Mathematical Sciences, Vancouver, British Columbia, Canada.\\ peringuey@math.ubc.ca}}
\date{}
\begin{document}
\selectlanguage{french}
\maketitle

\selectlanguage{english}
\begin{abstract}
An integer is a primitive root modulo a prime $p$ if it generates the whole multiplicative group $(\Z/p\Z)^*$. In 1927 Artin conjectured that an integer $a$ which is not $-1$ or a square is a primitive root for infinitely many primes, and that the set of those primes has a positive asymptotic density among all primes. This conjecture was proved, under the generalized Riemann hypothesis (GRH), in 1967 by Hooley.

More generally, an integer is called a generalized primitive root modulo $n$ if it generates a subgroup of $(\Z/n\Z)^*$ of maximal size. Li and Pomerance showed, under GRH, that the set of integers for which a given integer is a generalized primitive root doesn't have an asymptotic density among all integers.
 
 We study here the set of the $\ell$-almost primes, i.e. integers with at most $\ell$ prime factors, for which a given integer $a\in\Z\backslash\{-1\}$, which is not a square, is a generalized primitive root, and we prove, under GRH, that this set has an asymptotic density among all the $\ell$-almost primes.
\end{abstract}
\selectlanguage{french}
\begin{abstract}
 Un entier est une racine primitive modulo un premier $p$ s'il génère le sous-groupe multiplicatif $(\Z/p\Z)^*$. En 1927 Artin conjecture qu'un nombre $a$ qui n'est ni $-1$ ni un carré parfait est racine primitive pour une infinité de nombres premiers, et que l'ensemble de ces premiers a une densité positive parmi tous les premiers. Cette conjecture a été démontrée, sous l'hypothèse de Riemann généralisée (GRH), en 1967 par Hooley.
 
 Plus généralement, on dit  qu'un entier est une racine primitive généralisée modulo $n$ s'il génère un sous-groupe de taille maximale dans $(\Z/n\Z)^*$. Li et Pomerance ont montré, sous GRH, que l'ensemble des entiers pour lesquels un entier est racine primitive généralisée n'admet pas de densité parmi tous les entiers.
 
 On s'intéresse ici à l'ensemble des entiers $\ell$-presque premiers, c'est-à-dire les entiers ayant au plus $\ell$ facteurs premiers, pour lesquels un entier $a\in\Z\backslash\{-1\}$ donné et différent d'un carré est racine primitive généralisée, et nous montrons, sous GRH, que cet ensemble admet une densité parmi tous les $\ell$-presque premiers. 
\end{abstract}
\newpage
\tableofcontents
\section{Introduction}

Soient $a$ un entier et $p$ un nombre premier, on dit que $a$ est une racine primitive modulo $p$ si $a$ engendre le groupe multiplicatif $(\Z/p\Z)^*$. Une question qui émerge alors est de quantifier le nombre de nombres premiers pour lesquels un entier $a$ est racine primitive. Notons $N_a(x)$ le nombre de premiers plus petits que $x$ pour lesquels $a$ est racine primitive, on cherche alors à apprécier $N_a(x)$ lorsque $x$ tend vers l'infini.

En 1927 Emil Artin conjecture que tout entier $a$ différent de $-1$ et d'un carré est racine primitive modulo une infinité d'autres entiers. Il stipule qu'un tel entier $a$ serait générateur pour environ 37\% des premiers, c'est-à-dire $N_a(x)\sim C_A\pi(x)$, où $C_A=\prod\limits_{p}\left(1-\frac{1}{p(p-1)}\right)$ est la constante d'Artin. Une démonstration conditionnelle est fournie par Hooley \cite{Hooley} en 1967, en supposant l'hypothèse de Riemann pour certains corps de nombres. Plus précisément il démontre que $N_a(x)\sim C_A(a)\pi(x)$, la constante dépend donc du nombre $a$ considéré et est la constante conjecturée par Heilbronn (d'après \cite{Hooley}). Concernant des résultats inconditionnels Heath-Brown \cite{Heath}, améliorant un résultat de Gupta et Ram Murty \cite{Gupta}, a démontré qu'au plus deux nombres premiers ne sont pas racines primitives pour une infinité de nombres premiers. Plus précisément si $(q,r,s)$ est un triplet d'entiers multiplicativement indépendants tel que aucun élément de $\{q,r,s,-3qr,-3rs,-3qs,qrs\}$ ne soit un carré, alors l'ensemble des premiers pour lesquels au moins un entier parmi $q,r$ et $s$ est racine primitive est asymptotiquement $\gg \frac{x}{(\log x)^2}$. Pour un état de l'art concernant la conjecture d'Artin et les racines primitives généralisées, le lecteur pourra se référer aux articles de synthèse de Moree \cite{Moree} et de Li et Pomerance\cite{PomeranceLi}.\\

On étend la notion de racine primitive au sous-groupe multiplicatif $(\Z/n\Z)^*$, où $n$ est un entier quelconque, en définissant les racines primitives généralisées modulo $n$ comme étant les éléments de $(\Z/n\Z)^*$ générant des sous-groupes de taille maximale, c'est-à-dire $\{m\in (\Z/n\Z)^*,\ |\langle m\rangle|=\lambda(n)\}$, où $\lambda$ est la fonction lambda de Carmichael \cite{Carmichael}, définie telle que $\lambda(n)$ est la taille maximale atteinte par les sous-groupes cycliques de $(\Z/n\Z)^*$. Il est alors naturel de se demander si des résultats similaires à la conjecture d'Artin existent dans le cadre des racines primitives généralisées. Notons  $N'_a(x)$ le nombre d'entiers plus petits que $x$ pour lesquels $a$ est racine primitive généralisée, on espère alors que $N'_a(x)\sim C(a) x$. Malheureusement ce n'est pas le cas. Soit $\mathtt{E}$ l'ensemble des entiers qui sont soit une puissance d'exposant strictement plus grand que 1, soit un carré multiplié par $-1$ ou $\pm 2$. Li \cite{Li} a démontré que pour tout $a$, $\lim \text{inf} \frac{N'_a(x)}{x}=0$ et Li et Pomerance \cite{PomeranceLi} ont démontré que pour les entiers $a$ n'appartenant pas à $\mathtt{E}$ on a, sous l'hypothèse de Riemann généralisée,  $\lim \text{sup} \frac{N'_a(x)}{x}>0$. Ainsi le nombre d'entiers pour lesquels un certain entier $a$ est racine primitive généralisée n'admet pas de densité parmi tous les entiers. Comme pour le cas classique de la conjecture d'Artin d'autres problèmes surgissent naturellement comme par exemple l'estimation de l'ordre de grandeur de la plus petite racine primitive généralisée pour un entier, dont Martin\cite{Martin} fournit une majoration pour presque tout entier.\\

Soit $\ell$ un entier plus grand que $1$. Dans cet article, on étudie une situation intermédiaire, celle de l'ensemble des nombres $\ell$-presque premiers pour lesquels un entier $a$ est racine primitive généralisée, et nous montrons qu'il admet une densité parmi tous les $\ell$-presque premiers.

Pour le reste de l'article $p$ et $q$ représenteront toujours des nombres premiers, et nous noterons $\ord_n(b)$ l'ordre de $b$ dans $(\Z/n\Z)^*$, $(b,c)$ le pgcd de $b$ et $c$, $[b,c]$ le ppcm de $b$ et $c$, $\nu_q(b)$ la valuation $q$-adique de $b$. On note $\mathscr{P}^*(A)$ l'ensemble des parties non-vides d'un ensemble $A$. On fixe $a\in\Z\backslash\{-1\}$ un entier différent d'un carré, et on écrit $a$ sous la forme $a=a_1a_2^2$ où $a_1$ est sans facteur carré et éventuellement négatif. De plus, on définit $h$ comme le plus grand entier tel que $a$ soit une $h$-ième puissance :
\begin{equation}
 \label{def_h}h:=\max\{\nu,\ \exists b,\ a=b^\nu\}.
\end{equation}

Un nombre $\ell$-presque premier est un nombre ayant au plus $\ell$ diviseurs premiers comptés avec multiplicité. Landau \cite{Landau_fr} a montré que le nombre de $\ell$-presque premiers plus petit que $x$ est asymptotiquement $\frac{x(\log\log x)^{\ell -1}}{(\ell-1)!\log x}$. On s'intéresse ici au comportement asymptotique du nombre de $\ell$-presque premiers pour lesquels $a$ est une racine primitive généralisée. Nous obtenons sous l'hypothèse de Riemann généralisée (GRH) le théorème suivant :

\begin{thm}[GRH]\label{thm_principal}
 Soient $\ell\geq1$, $E=\{1,\ldots,\ell\}$, $a$ un entier qui n'est ni $-1$, ni un carré parfait et $\Na (x)$ le nombre de $\ell$-presque premiers plus petits que $x$ et qui ont $a$ comme racine primitive généralisée. On a :
 \[\Na (x)=\frac{x(\log\log x)^{\ell-1}}{(\ell-1)!\log x}\prod\limits_{\substack{p}}\left(1-W_\ell(p)\right)\left(1+V_\ell(a_1)\right)\left(1+o(1)\right),\]
 avec les notations suivantes :
 \begin{enumerate}
 \item $W_\ell(p):=\sum\limits_{i=1}^{\ell}\binom{\ell}{i}\sum\limits_{j=0}^{\ell-i}\binom{\ell-i}{j}\frac{(-1)^jp^{i+j}(h,p)^i}{p^{2i}(p-1)^j(p^{i+j}-1)}$, est la contribution ne dépendant que de $h$,
 \item $V_\ell(a_1):=\mu(2\tilde{a_1})\frac{H_2(\ell,a_1)}{1-W_\ell(2)}\prod\limits_{\substack{p\\ p|a_1\\ p\geq 3}}\left(1-W_\ell(p)\right)^{-1}$, est la contribution spécifique dépendant de $a$, où $\tilde{a_1}:=\frac{a_1}{(2,a_1)}$,
 \item 
 \begin{equation*}\hspace*{-1cm}
 H_2(\ell,a_1):=\sum\limits_{k=1}^\ell  \binom{\ell}{k} 2^{-\ell-k}\delta_\ell(k)\mu(\tilde{a_1})^{k}\prod\limits_{\substack{p|\tilde{a_1}\\ p\geq 3}}\left(\frac{(p-2)^{\ell -k}}{(p-1)^{\ell }}\right)\sum\limits_{\prod\limits_{\{i,j\}\in \mathcal{D}}a_{i,j}=\tilde{a_1}}\prod\limits_{\{i,j\}\in \mathcal{D}}G_{i,j}^{k}({a_{i,j}})\end{equation*} 
 où
 \[\delta_\ell(k)=\sigma(a_1,k)+2^{\ell-2k}\sum\limits_{m=0}^{\ell-k}\binom{\ell-k}{m}2^{-3m}\sum\limits_{r=0}^{\ell -k-m}\binom{\ell -k-m}{r}\frac{(-1)^r2^{-r}}{2^{k+m+r}-1},\]
  avec $\sigma(a_1,k):=\left\{\begin{array}{cl}
                          2^{-\ell+k}+2^{-2\ell+k}5^{\ell-k}& \text{ si } a_1\equiv 1\Mod{4},\\
                          (-1)^k2^{-\ell+k}+2^{-2\ell+k}5^{\ell-k}& \text{ si } a_1\equiv 3\Mod{4},\\
                          (-1)^k2^{-2\ell+k}5^{\ell-k}& \text{ sinon. }
                         \end{array}\right.$,\\
$\mathcal{D}=[0,k]\times[0,\ell-k]\backslash\{0,0\}$, et pour $(i,j)\in\mathcal{D}$, $G_{i,j}^k$ est la fonction multiplicative définie pour les nombres premiers impairs par :
 \[G_{i,j}^k(p)=\binom{k}{i}\binom{\ell-k}{j}\left(\frac{(p-1)(h,p)}{p^2(p-2)}\right)^{i+j}(2-p)^{i}\left(1+R_p(k+j)\right),\]
 avec $R_p(m):=\sum\limits_{r=0}^{\ell -m}\binom{\ell -m}{r}\frac{(-1)^r(p-1)^{\ell -m-r}}{(p^{m+r}-1)(p-2)^{\ell -m}}$.
 
\end{enumerate}
\end{thm}

La complexité des termes impliqués dans le résultat ci-dessus provient, en grande partie, du fait qu'un entier peut être racine primitive généralisée modulo $p_1\cdots p_\ell$ sans pour autant être racine primitive modulo un des $p_i$ (par exemple $1636 $ est une racine primitive généralisée modulo $ 4054051=1801\times 2251$ mais n'est pas une racine primitive modulo $ 1801$ ou $ 2251$). Nous montrons dans la section suivante que pour être racine primitive généralisée modulo $p_1\cdots p_\ell$ un entier doit vérifier des critères plus faibles que celui d'être racine primitive mais cela simultanément modulo chacun des $p_i$, et donc le résultat ne découle pas immédiatement du résultat de Hooley\cite{Hooley}, dont on retrouve ci-dessus le terme principal en prenant $\ell=1$. En effet ce dernier ramène le problème à compter les idéaux premiers d'un certain corps de nombres, ce qu'il peut alors faire en étudiant, sous GRH, la fonction zeta de Dedekind associée à ce corps. Dans notre cas, nous ramenons le problème à compter simultanément les idéaux premiers de plusieurs corps de nombres, ce qui n'est pas possible en appliquant classiquement la formule de Perron et en déformant le contour. Pour surmonter cette difficulté, nous démontrons le résultat inconditionnel suivant, qui découle d'une application atypique de la méthode de Selberg-Delange :

\begin{thm}\label{thm_finale_Selberg_Delange}
Soient $\ell$ un entier non nul, $a$ un entier non nul qui n'est ni $-1$ ni un carré, $C$ une constante. Soient $\upsilon_1,\ldots,\upsilon_\ell$ des entiers plus petits que $C$ et $\kappa_1,\ldots,\kappa_\ell$ des entiers sans facteur carré tels que pour tout $1\leq i\leq \ell$, $\kappa_i|\upsilon_i$, alors :
 \[\sum\limits_{\substack{p_1\cdots p_\ell\leq x\\ (a,p_1\cdots p_\ell)=1\\ \forall i,\ p_i\equiv 1(\upsilon_i)\\a\in \qroot\left(\kappa_i, p_i\right)}}1=\frac{\ell x}{\log x \prod\limits_i n_i}(\log\log x)^{\ell-1}+\bo_{C}\left(\frac{x}{\log x}(\log\log x)^{\ell-2}\right),\]
 où $n_i:=[\Q\left(\sqrt[\kappa_i]{a},\xi_{\upsilon_i}\right):\Q]$, $\xi_{\upsilon_i}$ est une racine primitive $\upsilon_i$-ième de l'unité et $\qroot\left(\kappa_i, p_i\right)$ est l'ensemble des entiers dont la classe $\Mod{p_i}$ est une puissance $\kappa_i$-ième.
\end{thm}

De plus, inconditionnellement, on a la majoration suivante :

\begin{thm}\label{thm_inconditionnel}
 Avec les mêmes hypothèses que pour le Théorème \ref{thm_principal}, on a :
 \[\Na (x)\leq\frac{x(\log\log x)^{\ell-1}}{(\ell-1)!\log x}\prod\limits_{\substack{p}}\left(1-W_\ell(p)\right)\left(1+V_\ell(a_1)\right)\left(1+o(1)\right).\]
\end{thm}

Posons $C_\ell(a)=\lim\limits_{x\rightarrow\infty}{\mathcal N_{a,\ell}} (x)/\frac{x(\log\log x)^{\ell-1}}{(\ell-1)!\log x}$. Nous avons calculé numériquement plusieurs valeurs particulières, reprises dans le tableau suivant :

\begin{tabular}{|c|c|c|c|c|c|}
\hline
$ \ell $& $\prod\limits_{\substack{p}}\left(1-W_\ell(p)\right),\ h=1$ & $C_\ell(2)$ & $C_\ell(3)$ & $C_\ell(5)$ & $C_\ell(10)$ \\
\hline
1 & $\simeq 0.3739$ & $\simeq 0.3739$ & $\simeq 0.3739$ & $\simeq 0.3936$ & $\simeq 0.3739$\\
2 & $\simeq 0.3759$ & $\simeq 0.3222$ & $\simeq 0.3950$ & $\simeq 0.3878$ & $\simeq 0.3775$\\
5 & $\simeq 0.3261$ & $\simeq 0.1318$ & $\simeq 0.3252$ & $\simeq 0.3278$ & $\simeq 0.3272$\\
10 & $\simeq 0.3051$ & $\simeq 0.0293$ & $\simeq 0.3053$ & $\simeq 0.3046$ & $\simeq 0.3047$\\
20 & $\simeq 0.2919$ & $\simeq 0.0015$ & $\simeq 0.2918$ & $\simeq 0.2920$ & $\simeq 0.2920$\\
50 & $\simeq 0.2807$ & $\simeq 2\times 10^{-7}$ & $\simeq 0.2807$ & $\simeq 0.2807$ & $\simeq 0.2807$\\
\hline
\end{tabular}
\ \\
On retrouve sur la première ligne la constante d'Artin, ainsi que la légère déviation pour $a=5$. Dans les lignes suivantes, toutes les valeurs présentent une déviation, qui semble s'estomper quand $\ell$ grandit, sauf pour $a=2$ auquel cas $C_\ell(2)$ semble tendre vers $0$, ce qui est en accord avec le fait que $2$ est dans l'ensemble $\mathtt{E}$ décrit par Li et Pomerance \cite{PomeranceLi}.
\ \\
\ \\
Nous présentons ensuite une approche heuristique analogue à celle du cas classique. Nous transformons le problème en l'évaluation du nombre de $\ell$-presque premiers ne vérifiant pas une certaine propriété pour tout nombre premier $q$. En utilisant l'Hypothèse de Riemann Généralisée nous ramenons le problème à l'évaluation du nombre de $\ell$-presque premiers ne vérifiant pas ces propriétés pour $q$ petit. Puis nous montrons que les $p_1\cdots p_\ell$ comptés sont tels que chaque $p_i$ se décompose comme produit d'idéaux premiers distincts dans un certain corps de nombres. En utilisant la méthode de Selberg-Delange nous évaluons le nombre de $p_1\cdots p_\ell\leq x$ vérifiant ces propriétés simultanément. Enfin, après avoir contrôlé les termes d'erreurs, nous utilisons des méthodes combinatoires pour obtenir l'expression du terme principal.

\section{Caractérisation des racines primitives généralisées.}\label{section_caract_prim_root}
Nous donnons dans ce paragraphe un critère caractérisant les racines primitives généralisées. Ce critère est énoncé dans le lemme \ref{lemme_caracterisation_racines_primitives_generalisees}.

Nous commençons par introduire des notations qui nous serviront tout au long de l'article.
Pour $l\in \N$ et $n\in \N$ donnés, on notera $\qroot(l,n)$ l'ensemble des entiers dont la classe $\Mod{n}$ est une puissance $l$-ième:
\begin{equation}
 \label{def_Rln}\qroot(l,n)=\{c\in\Z,\exists b\in\Z, c\equiv b^l\Mod{n}\}.
\end{equation}


Pour $q,p_1,\ldots,p_\ell$ des nombres premiers, on notera $M_q(p_1,\ldots,p_\ell)$ l'ensemble des $p\in \{p_1,\ldots,p_\ell\}$ pour lesquels la valuation $q$-adique de $p-1$ est maximale parmi les $p_i$ :
\begin{equation}
 \label{def_Mq} M_q(p_1,\ldots,p_\ell)=\left\{p\in\{p_1,\ldots,p_\ell\}, \nu_q(p-1)=\nu_q(\lambda(p_1\cdots p_\ell)) \right\}.
\end{equation}
Cet ensemble est toujours non vide. En effet la fonction lambda de Carmichael vérifie les propriétés suivantes : 
\begin{enumerate}
  \item $\lambda\big(p^r\big)=\left\{\begin{array}{cl}
                            \frac{1}{2}\varphi\big(p^r\big)&\ \text{si}\ p=2\ \text{et}\ r\geq3 \\
                             \varphi\big(p^r\big)&\ \text{sinon}
                           \end{array}
\right.$,
\item Si $n=\prod\limits_{i=1}^k p_i^{v_i}$ avec $p_i\neq p_j$, alors $\lambda(n)=[\lambda( p_1^{v_1}),\ldots,\lambda( p_k^{v_k})]$.
 \end{enumerate}
\ \\
On commence par donner un résultat classique sur les ordres des éléments dans $(\Z/n\Z)^*$.
\begin{lem}\label{ordreppcm}
Pour tous $a,n_1,n_2$ dans $\N$, $(a,n_1n_2)=1$ :
\[\ord_{[n_1,n_2]}(a)=[\ord_{n_1}(a),\ord_{n_2}(a)]\]
\end{lem}
\begin{proof}
 Commençons par montrer que $\ord_{n_1}(a)|\ord_{[n_1,n_2]}(a)$.\\
 Écrivons $\ord_{[n_1,n_2]}(a)=k\ord_{n_1}(a)+r$ avec $k\in\N$ et $0\leq r<\ord_{n_1}(a)$.\\
 alors 
 \begin{align*}
  a^{\ord_{[n_1,n_2]}(a)}\equiv 1(n_1)&\Leftrightarrow a^{k\ord_{n_1}(a)}a^r\equiv 1(n_1)\\
  &\Leftrightarrow a^r\equiv 1(n_1)
 \end{align*}
Or $r<\ord_{n_1}(a)$ ainsi $r=0$. Ainsi $\ord_{n_1}(a)|\ord_{[n_1,n_2]}(a)$, de même $\ord_{n_2}(a)|\ord_{[n_1,n_2]}(a)$ et donc $[\ord_{n_1}(a),\ord_{n_2}(a)]|\ord_{[n_1,n_2]}(a)$.\\
Reste à montrer que $a^{[\ord_{n_1}(a),\ord_{n_2}(a)]}\equiv 1([n_1,n_2])$, ce qui est immédiat car $n_1|a^{[\ord_{n_1}(a),\ord_{n_2}(a)]}-1$ et $n_2|a^{[\ord_{n_1}(a),\ord_{n_2}(a)]}-1$ donc $[n_1,n_2]|a^{[\ord_{n_1}(a),\ord_{n_2}(a)]}-1$.
\end{proof}

Nous cherchons à compter les nombres $m$ $\ell$-presque premiers c'est-à-dire tels que $\Omega(m)\leq\ell$, où $\Omega(m)$ est le nombre de facteurs premiers de $m$ comptés avec multiplicité, pour lesquels un nombre $a$ donné est racine primitive généralisée. Il est immédiat que si $a$ est un carré ou $a=-1$ alors $a$ ne peut être racine primitive généralisée que dans le groupe trivial $(\Z/2\Z)^*$. 

Nous allons montrer que quand $a$ n'est pas un carré parfait et $a\neq -1$ cet ensemble est infini et admet une densité positive parmi tous les $\ell$-presque premiers. 

De plus on a trivialement que $\sum\limits_{\substack{p_1\cdots p_{\ell-1}\leq x}}1\ll \frac{x}{\log x}(\log\log x)^{\ell-2}$ et $\sum\limits_{\substack{p_1\cdots p_\ell\leq x\\ p_1=p_2}}1\ll \frac{x}{\log x}(\log\log x)^{\ell-2}$. On ne s'intéresse dans la suite qu'aux $a$ qui ne sont pas des carré et aux $m\leq x$ sans facteur carré ayant exactement $\ell$ facteurs premiers.

\begin{lem}\label{lemme_caracterisation_racines_primitives_generalisees}
 Soient $p_1,\ldots,p_\ell$, $\ell$ nombres premiers distincts et $a$ un entier tel que $(a,p_1\cdots p_\ell)=1$.\\
 Alors $a$ est une racine primitive généralisée modulo $p_1\cdots p_\ell$ si et seulement si pour tout $q|\lambda(p_1\cdots p_\ell)$, il existe $p\in M_q(p_1,\ldots,p_\ell)$ tel que  $a$ ne soit pas le résidu d'une $q$-ième puissance modulo $p$.
\end{lem}

\begin{proof}
Supposons qu'il existe $q|\lambda(p_1\cdots p_\ell)$ premier tel que pour tout $p\in M_q(p_1,\ldots,p_\ell)$, $a$ soit le résidu d'une $q$-ième puissance modulo $p$ et donc $\ord_{p}(a)|\frac{p-1}{q}$. 

Ainsi d'après le lemme \ref{ordreppcm}, \[\nu_q({\ord_{p_1\cdots p_\ell}(a)})=\max\limits_{p\in\{p_1,\ldots,p_\ell\}}\nu_q({\ord_{p}(a)})<\max\limits_{p\in\{p_1,\ldots,p_\ell\}}\nu_q({p-1})=\nu_q(\lambda(p_1\cdots p_\ell))\] et donc $a$ n'est pas une racine primitive généralisée modulo $p_1\cdots p_\ell$.\\

Réciproquement, supposons que $a$ ne soit pas une racine primitive généralisée modulo $p_1\cdots p_\ell$. Alors il existe $q|\lambda(p_1\cdots p_\ell)$ tel que $a^{\frac{\lambda(p_1\cdots p_\ell)}{q}}\equiv1(p_1\cdots p_\ell)$.\\
Soit $p\in M_q(p_1,\ldots,p_\ell)$, alors d'après le lemme \ref{ordreppcm}, $\ord_{p}(a)|\frac{p-1}{q}$.

Soit $\gamma$ une racine primitive modulo $p$, alors il existe $\alpha$ tel que $\gamma^\alpha\equiv a(p)$ et donc $\gamma^{\alpha\frac{p-1}{q}}\equiv 1\Mod{p}$.\\
Alors $p-1|\alpha\frac{p-1}{q}$, i.e $q|\alpha$. Ainsi $a$ est bien le résidu d'une $q$-ième puissance modulo $p$.\\

\end{proof}
Nous allons maintenant définir un critère pour caractériser les $\ell$-presque premiers pour lesquels $a$ n'est pas une racine primitive généralisée.
\begin{defi}
 On dit que $a$ vérifie $\RR(q,p_1\cdots p_\ell)$ si $(a,p_1\cdots p_\ell)=1$ et $a$ est le résidu d'une $q$-ième puissance modulo $p$ pour tout $p\in M_q(p_1,\ldots,p_\ell)$.
\end{defi}

Ainsi $a$ est racine primitive généralisée modulo $p_1\cdots p_\ell$ si $(a,p_1\cdots p_\ell)=1$ et s'il ne vérifie pas $\RR(q,p_1\cdots p_\ell)$ pour tout $q$. Pour $\ell=1$ on retrouve la propriété $\RR(q,p)$ introduite par Hooley \cite[\S 3]{Hooley}. La section suivante étend une approche heuristique de la conjecture d'Artin à notre problème.

\section{Approche heuristique expliquant le terme $W_\ell(p)$ des Théorèmes \ref{thm_principal} et \ref{thm_inconditionnel}}

Dans un premier temps nous rappelons l'approche heuristique de Heilbronn \cite{PomeranceLiSurvey} pour la conjecture d'Artin, c'est-à-dire quand $\ell=1$. Dans ce cas $a$ est une racine primitive $\Mod{p}$ si et seulement si $(a,p)=1$ et $a$ ne vérifie pas $\RR(q,p)$ pour tout $q$. Fixons $q$ un nombre premier et examinons la probabilité qu'un nombre premier $p$ premier avec $a$ soit tel que $a$ ne vérifie pas $\RR(q,p)$.
\[\mathbb{P}(\lquote a \text{ ne vérifie pas } \RR(q,p)\rquote)=1-\mathbb{P}(p\equiv 1 \Mod{q})\mathbb{P}(\lquote \exists b \text{ tel que }a\equiv b^q\Mod{p}\rquote|\lquote p\equiv 1 \Mod{q}\rquote),\]
où la dernière probabilité est une probabilité conditionnelle. 
Alors, d'après le Théorème de Dirichlet sur les nombres premiers en progression arithmétique on a $\mathbb{P}(p\equiv 1 \Mod{q})=\frac{1}{q-1}$. Rappelons que $h$ est l'entier défini par \eqref{def_h}, c'est-à-dire le plus grand entier tel que $a$ soit une $h$-ième puissance . Alors trivialement si $q|h$, \[\mathbb{P}(\lquote \exists b \text{ tel que }a\equiv b^q\Mod{p}\rquote|\lquote p\equiv 1 \Mod{q}\rquote)=1.\] Si $q\nmid h$, alors 
\[\mathbb{P}(\lquote\exists b \text{ tel que }a\equiv b^q\Mod{p}\rquote|\lquote p\equiv 1 \Mod{q}\rquote)=\mathbb{P}(\lquote a^{\frac{p-1}{q}}\equiv 1\Mod{p} \rquote|\lquote p\equiv 1 \Mod{q}\rquote).\]
D'après le petit théorème de Fermat, la classe de $a^{\frac{p-1}{q}}$ appartient à $\{\bar{x}\in(\Z/p\Z)^*,\ \bar{x}^q=\bar{1}\}$. Or cet ensemble possède $q$ éléments, on peut donc supposer que la probabilité que $a^{\frac{p-1}{q}}\equiv 1\Mod{p}$ est $\frac{1}{q}$ (voir \cite{Moree}).

Un argument plus rigoureux (\cite{Li} lemme 3.2) consiste à observer que $p\equiv 1\Mod{q}$ et $\lquote a \text{ vérifie }\RR(q,p)\rquote$ si et seulement si $p$ est totalement décomposé dans le corps de Kummer $L:=\Q\left(\xi_q, a^{\frac{1}{q}}\right)$, où $\xi_q$ est une racine primitive $q$-ième de l'unité. Cela revient à dire que les éléments de Frobenius $\sigma_{\mathfrak{p}}$ pour $\mathfrak{p}\in\mathcal{O}_L$ au dessus de $p$, sont bien définis et sont dans la classe de conjugaison de l'identité dans $\Gal(L/\Q)$.

Or $L/\Q$ est un extension abélienne, donc cette classe de conjugaison est réduite à un singleton. D'après le théorème de Tchebotarev la proportion de tels $p$ est $\frac{1}{[L:\Q]}$, qui vaut $\frac{1}{q(q-1)}$ sauf si $a$ est un carré.

Ainsi $\mathbb{P}(\lquote a \text{ ne vérifie pas } \RR(q,p)\rquote)=1-\frac{(h,q)}{q(q-1)}$. Alors en supposant l'indépendance selon les $q$ des événements $\lquote a$ vérifie $\RR(q,p)\rquote$ on obtient
\[\mathbb{P}(\lquote a \text{ est racine primitive } \Mod{p}\rquote)=\prod\limits_{q}\left(1-\frac{(h,q)}{q(q-1)}\right)=:C_a(h),\]
où $C_a(1)$ est la constante d'Artin. Cependant, cette hypothèse d'indépendance est bien trop forte, par exemple $\Q(\sqrt{5})$ est un sous-corps de $\Q(\xi_5,\sqrt[5]{5})$ (car $\sqrt{5}=-2(e^{4i\pi/5}+e^{-4i\pi/5})-1$), et donc la condition $\lquote 5$ vérifie $\RR(5,p)\rquote$ implique $\lquote 5$ vérifie $\RR(2,p)\rquote$. Un terme correctif dépendant de $a$ devra alors être ajouté pour tenir compte de ce genre d'éventualités (voir\cite{Stevenhagen}). \\

Passons maintenant au problème pour les $\ell$-presque premiers. Soient $p_1,\ldots,p_\ell$ des nombres premiers, ne divisant pas $a$. Notons 
\[W_\ell(q):=\mathbb{P}(\lquote a \text{ vérifie } \RR(q,p_1\cdots p_\ell)\rquote ).\]
Découpons $W_\ell(q)$ suivant la valuation $q$-adique de $\lambda(p_1\cdots p_\ell)$, et suivant la taille de $M_q(p_1,\ldots,p_\ell)$.
\begin{align*}
 W_\ell(q)&=\sum\limits_{m=1}^{\infty}\sum\limits_{i=1}^{\ell}\mathbb{P}(\lquote a \text{ vérifie } \RR(q,p_1\cdots p_\ell)\rquote,\ \nu_q(\lambda(p_1\cdots p_\ell))=m,\ |M_q(p_1,\ldots,p_\ell)|=i)\\
 &=\sum\limits_{m=1}^{\infty}\sum\limits_{i=1}^{\ell}\binom{\ell}{i}\mathbb{P}\left(\begin{array}{ll}\nu_q(p_j-1)<m & \forall i<j\leq \ell\\ \nu_q(p_j-1)=m & \forall 1\leq j\leq i\\ \exists b \text{ tel que }a\equiv b^q\Mod{p_j} & \forall 1\leq j\leq i
\end{array}
\right).\end{align*}
Les événements impliquant les $p_j$ sont maintenant indépendants :
\begin{equation*}
W_\ell(q)=\sum\limits_{m=1}^{\infty}\sum\limits_{i=1}^{\ell}\binom{\ell}{i}\mathbb{P}(\nu_q(p-1)<m)^{\ell-i}\mathbb{P}(\nu_q(p-1)=m,\ \lquote\exists b \text{ tel que }a\equiv b^q\Mod{p}\rquote)^{i}.
\end{equation*}

Alors d'après le Théorème de Dirichlet $\mathbb{P}(\nu_q(p-1)<m)=1-\frac{1}{\varphi(q^m)}$. Puis
\begin{align*}
        \mathbb{P}(\nu_q(p-1)=m,\ \lquote\exists b \text{ tel que }a\equiv b^q\Mod{p}\rquote)&=\mathbb{P}(p\equiv 1\Mod{q^m},\ \lquote\exists b \text{ tel que }a\equiv b^q\Mod{p}\rquote)\\
        &-\mathbb{P}(p\equiv 1\Mod{q^{m+1}},\ \lquote\exists b \text{ tel que }a\equiv b^q\Mod{p}\rquote)\\
        &=\frac{(h,q)}{q}\left(\frac{1}{\varphi(q^m)}-\frac{1}{\varphi(q^{m+1})}\right)=\frac{(h,q)(q-1)}{q^2\varphi(q^m)}.
\end{align*}

On reporte dans $W_\ell(q)$ :

\[W_\ell(q)=\sum\limits_{m=1}^{\infty}\sum\limits_{i=1}^{\ell}\binom{\ell}{i}\left(1-\frac{1}{\varphi(q^m)}\right)^{\ell-i}\frac{(h,q)^i(q-1)^i}{q^{2i}\varphi(q^m)^i}.\]
On développe $\left(1-\frac{1}{\varphi(q^m)}\right)^{\ell-i}$ en utilisant le la formule du binôme de Newton, puis on inverse les sommes finies et la somme infinie, on obtient alors,
\[W_\ell(q)=\sum\limits_{i=1}^{\ell}\binom{\ell}{i}\sum\limits_{j=0}^{\ell-i}\binom{\ell-i}{j}\frac{(-1)^jq^{i+j}(h,q)^i}{q^{2i}(q-1)^j(q^{i+j}-1)}.\]
Alors, en supposant l'indépendance des $\RR(q,p_1\cdots p_\ell)$, on a que
\[\mathbb{P}(a \text{ est racine primitive généralisée }\Mod{p_1\cdots p_\ell})=\prod\limits_{q}(1-W_\ell(q)).\]
Cependant on verra par la suite qu'il n'y a pas indépendance, et donc qu'un coefficient correctif dépendant de $a_1$, $h$ et $\ell$ devra être ajouté.
\\

La principale difficulté pour résoudre ce genre de problème est d'exclure les $\ell$-presque premiers tels que $a$ vérifie $\RR(q,p_1\cdots p_\ell)$ pour un $q$ grand, la section suivante est un premier pas dans ce sens.

\section{Un premier découpage : le découpage initial de Hooley} Nous adaptons les notations du paragraphe puis procédons à un découpage de $\Na(x)$ analogue à celui effectué par Hooley \cite{Hooley} pour le cas $\ell=1$.

On a vu dans la section \ref{section_caract_prim_root} qu'un entier $p_1\cdots p_\ell$ était compté dans $\Na(x)$ si $a$ ne vérifie pas $\RR(q,p_1\cdots p_\ell)$ pour tout $q$.

Pour $\eta,\ \eta_1,\ \eta _2\in\R$, $\eta_1<\eta_2$, et $k$ un entier sans facteur carré on considère les cardinaux suivants :
\begin{equation}\label{def_Na_xeta}
 \Na(x,\eta)=\#\{p_1\cdots p_\ell\leq x|\ a\ \text{ne vérifie pas }\RR(q,p_1 \ldots p_\ell),\ \forall q\leq\eta\},
\end{equation}

\begin{equation}\label{def_Pa}
 \Pa(x,k)=\#\{p_1\cdots p_\ell\leq x|\ a\ \text{vérifie }\RR(q,p_1 \ldots p_\ell),\ \forall q|k\}.
\end{equation}

Enfin il nous reste à introduire l'analogue des $M_a(x,\eta_1,\eta_2)$ de Hooley que nous noterons $\Ma(x,\eta_1,\eta_2)$. Pour $\eta_1<\eta_2$, $\Ma(x,\eta_1,\eta_2)$ est le nombre de $p_1\cdots p_\ell\leq x$ pour lesquels $a$ vérifie $ \RR(q,p_1 \ldots p_\ell)$ pour au moins un $q$ dans l'intervalle $]\eta_1,\eta_2]$.

On découpe $]0,x-1[$ en 4 parties suivant les bornes suivantes :
\begin{align}\label{def_bornes_Ci}
 \boldsymbol{\cdot}\ &\ C_1\text{ une constante arbitrairement grande;}\nonumber\\
 \boldsymbol{\cdot}\ & C_2=x^{\frac{1}{2}}\log^{-8}x;\\
 \boldsymbol{\cdot}\ & C_3=x^{\frac{1}{2}}\log x .\nonumber
\end{align}

Tout d'abord,
\begin{equation}\label{majoration_inconditionnelle}
\Na(x)\leq\Na(x,C_1).
\end{equation}
On minore $\Na(x)$ à l'aide des quantités $\Ma(x,\eta_1,\eta_2)$,
\[\Na(x)\geq\Na(x,C_1)-\Ma(x,C_1,x-1).\]
On en déduit
\[\Na(x)=\Na(x,C_1)+\bo(\Ma(x,C_1,x-1)).\]
On obtient alors l'équation fondamentale :
\begin{equation}\label{equationfondav1}
\Na(x)=\Na(x,C_1)+\bo(\Ma(x,C_1,C_2))+\bo(\Ma(x,C_2,C_3))+\bo(\Ma(x,C_3,x-1)).
\end{equation}

Dans la section suivante nous fournissons des majorations de $\Ma(x,C_2,C_3)$ et $\Ma(x,C_3,x-1)$. Le terme $\Ma(x,C_1,C_2)$ est ensuite majoré à l'aide de l'Hypothèse de Riemann Généralisée.


\section{Une deuxième série de découpages et distinction de la majoration nécessitant l'Hypothèse de Riemann Généralisée}

Afin de majorer le terme d'erreur dans \eqref{equationfondav1} nous allons procéder en effectuant des découpages successifs, d'abord suivant les valeurs de $q$ puis suivant les valeurs de $p_2\cdots p_\ell$. Par souci de clarté, ces découpages sont synthétisés dans l'arbre ci-dessous, où les nœuds sont les différents découpages, et les commentaires sur les arrêtes indiquent la méthode utilisée pour parvenir à la majoration.

\newpage
 \scriptsize

\rotatebox{90}{
\begin{minipage}{\textheight}\begin{forest} for tree={l sep = 30mm,
    edge path={\noexpand\path[\forestoption{edge}] (\forestOve{\forestove{@parent}}{name}.parent anchor) -- +(0,-12pt)-| (\forestove{name}.child anchor)\forestoption{edge label};}
}
[{$\sum\limits_{p_1\cdots p_\ell\leq x}\id\left(\exists q\in]C_1,x],\ a \text{ vérifie }\RR(q,p_1\ldots p_\ell)\right)$} [{$C_1<q\leq \frac{\sqrt{x}}{\log^8 x}$}
[{$p_1> \log^{64}x$} [{$p_2\cdots p_\ell\leq \frac{x}{q^2 \log^6 x}$} [{$\ll\frac{x(\log\log x)^{\ell -1}}{C_1\log x}$},edge label={node[align=left,midway,right]{\ \\ \ \\ \ \\ \ \\ \ \\ \ \\ \ \\ GRH,\\ Lemme de Hooley}}]] [{$\frac{x}{q^2 \log^6 x}<p_2\cdots p_\ell\leq \frac{x}{q^2}\log^2 x$} [{$\ll\frac{x(\log\log x)^{\ell -2}}{\log x}$},edge label={node[align=left,midway,right]{\ \\ \ \\ \ \\ \ \\ \ \\ \ \\ \ \\ Théorème de\\Brun-Titchmarsh,\\ Formule de Mertens,\\ Lemme de Landau}}]] [{$p_2\cdots p_\ell> \frac{x}{q^2}\log^2 x$} [{$\ll\frac{x(\log\log x)^{\ell -1}}{\log C_1\log^2 x}$},edge label={node[align=left,midway,right]{\ \\ \ \\ \ \\ \ \\ \ \\ \ \\ \ \\ Combinatoire}}]]] [{$p_1\leq \log^{64}x$} [{$\ll \frac{x(\log\log x)^{\ell-2}}{\log x}\log\log\log x$},edge label={node[align=left,midway,right]{\ \\ \ \\ \ \\ \ \\ \ \\ \ \\ \ \\ Trivial}}]]]
[{$\frac{\sqrt{x}}{\log^8 x}<q\leq\sqrt{x}\log x$} [{$p_2\cdots p_\ell \leq \sqrt{x}^{1-\varepsilon}$} [{$\ll \frac{x}{\log^2 x}(\log\log x)^\ell$},edge label={node[align=left,midway,right]{\ \\ \ \\ \ \\ \ \\ \ \\ \ \\ \ \\ Théorème de\\ Brun-Titchmarsh}}]] [{$\sqrt{x}^{1-\varepsilon}<p_2\cdots p_\ell \leq \sqrt{x}^{1+\varepsilon}$}  [{$\ll \frac{x}{\log x}(\log\log x)^{\ell-2}$},edge label={node[align=left,midway,right]{\ \\ \ \\ \ \\ \ \\ \ \\ \ \\ \ \\ Théorème des\\ nombres premiers,\\ Formule de Mertens}}]] [{$p_2\cdots p_\ell > \sqrt{x}^{1+\varepsilon}$} [{$\ll 1$},edge label={node[align=left,midway,right]{\ \\ \ \\ \ \\ \ \\ \ \\ \ \\ \ \\ $p_1<\sqrt{x}^{1+\varepsilon}$\\ or $q>\frac{\sqrt{x}}{\log^8 x}$}}]]
]
[{$\sqrt{x}\log x<q\leq x$} [{$\ll \frac{x}{\log^2 x}$},edge label={node[align=left,midway,right]{\ \\ \ \\ \ \\ \ \\ \ \\ \ \\ \ \\ Combinatoire}}]
]
]
\end{forest}\end{minipage}}


\normalsize
\subsection{Majoration de $\Ma(x,C_2,C_3)$ et $\Ma(x,C_3,x-1)$}
\begin{prop}\label{prop_maj_C2C3(x-1)}
 En reprenant les définitions de $C_2$ et $C_3$ de \eqref{def_bornes_Ci}, on a :
 \begin{enumerate}[i)]
\item \[\Ma(x,C_2,C_3)=\bo\left(\frac{x}{\log x}(\log\log x)^{\ell-2}\right)\]
\item \[\Ma(x,C_3,x-1)=\bo\left(\frac{x}{\log^2 x}\right)\]
 \end{enumerate}

\end{prop}
\begin{proof}
 \begin{enumerate}[i)]
 \item
 Nous majorons $\Ma(x,C_2,C_3)$ à l'aide des cardinaux $\Pa(x,k)$ définis par \eqref{def_Pa} :
 \[\Ma(x,C_2,C_3)\leq\sum\limits_{C_2<q\leq C_3}{\Pa(x,q)}.\]
 En ne retenant que la condition $q|\lambda(p_1\cdots p_\ell)$ dans $\RR(q,p_1 \ldots p_\ell)$, on obtient l'inégalité
\[\Pa(x,q)\leq \ell!\sum_{\substack{p_1\cdots p_\ell\leq x\\ p_1\equiv1(q)}}1.\]
Soit $\varepsilon>0$, on découpe cette somme selon la taille de $p_2\cdots p_\ell$,
\[\sum\limits_{C_2<q\leq C_3}\sum_{\substack{p_1\cdots p_\ell\leq x\\ p_1\equiv1(q)}}1=S_1+S_2+S_3,\]
avec \begin{align*}
      S_1&=\sum\limits_{C_2<q\leq C_3}\sum_{p_2\cdots p_\ell<\sqrt{x}^{1-\varepsilon}}\sum_{\substack{p_1\leq\frac{x}{p_2\cdots p_\ell}\\ p_1\equiv1(q)}}1,\\
      S_2&=\sum\limits_{C_2<q\leq C_3}\sum_{\sqrt{x}^{1-\varepsilon}\leq p_2\cdots p_\ell\leq\sqrt{x}^{1+\varepsilon}}\sum_{\substack{p_1\leq\frac{x}{p_2\cdots p_\ell}\\ p_1\equiv1(q)}}1,\\
      S_3&=\sum\limits_{C_2<q\leq C_3}\sum_{p_2\cdots p_\ell\geq\sqrt{x}^{1+\varepsilon}}\sum_{\substack{p_1\leq\frac{x}{p_2\cdots p_\ell}\\ p_1\equiv1(q)}}1.
     \end{align*}

Dans $S_1$ la somme sur $p_1$ est longue, on utilise le Théorème de Brun–Titchmarsh (\cite{Tenenbaum_fr}, Th I.4.16),

\begin{align*}
S_1&\ll\sum\limits_{C_2<q\leq C_3}\sum_{p_2\cdots p_\ell<\sqrt{x}^{1-\varepsilon}}\frac{2x}{p_2\cdots p_\ell q \log(\frac{x}{p_2\cdots p_\ell q})}\\
&\ll \frac{x}{\log^2 x}(\log\log x)^{\ell-1}\sum\limits_{C_2<q\leq C_3}\frac{\log q}{q}\\\
&\ll \frac{x(\log\log x)^\ell}{\log^2 x}.
\end{align*}
Pour $S_2$, on ne peut plus profiter de la congruence $p_1\equiv 1\Mod{q}$ car $p_1$ et $q$ sont d'un ordre de grandeur proche :
\begin{align*}
S_2&=\sum_{\sqrt{x}^{1-\varepsilon}\leq p_2\cdots p_\ell\leq\sqrt{x}^{1+\varepsilon}}\sum_{p_1\leq\frac{x}{p_2\cdots p_\ell}}\sum\limits_{\substack{C_2<q\leq C_3\\ q|p_1-1}}1\\
&\ll \sum_{\sqrt{x}^{1-\varepsilon}\leq p_2\cdots p_\ell\leq\sqrt{x}^{1+\varepsilon}}\sum_{p_1\leq\frac{x}{p_2\cdots p_\ell}}1,\\
\end{align*}

puis en appliquant le Théorème des nombres premiers, on obtient :
\begin{align*}
S_2&\ll \sum_{\sqrt{x}^{1-\varepsilon}\leq p_2\cdots p_\ell\leq\sqrt{x}^{1+\varepsilon}}\frac{x}{p_2\cdots p_\ell\log(\frac{x}{p_2\cdots p_\ell})}.
\end{align*}

Comme $p_2\cdots p_\ell\leq \sqrt{x}^{1+\varepsilon}$, $\log(\frac{x}{p_2\cdots p_\ell})\gg \log x $ :

\[S_2\ll\frac{x}{\log x}\sum_{\sqrt{x}^{1-\varepsilon}\leq p_2\cdots p_\ell\leq\sqrt{x}^{1+\varepsilon}}\frac{1}{p_2\cdots p_\ell}.\]

Si $p_2\cdots p_\ell> \sqrt{x}^{1-\varepsilon}$, alors $\max (p_2,\ldots, p_\ell)>\sqrt{x}^{\frac{1-\varepsilon}{\ell-1}}$.
Les rôles de $p_2,\ldots,p_\ell$ étant symétriques, on peut supposer que ce maximum est atteint par $p_2$ :
\[S_2\ll \frac{x}{\log x}\sum\limits_{p_3\cdots p_\ell \leq \sqrt{x}^{\ell+3\ell\varepsilon}}\frac{1}{p_3\cdots p_\ell}\sum\limits_{\sqrt{x}^{\frac{1-\varepsilon}{\ell-1}}< p_2\leq \sqrt{x}^{1+\varepsilon}}\frac{1}{p_2}.\]
La somme sur $p_2$ est un $\bo(1)$ où la constante peut dépendre de $\ell$. On en déduit que $S_2\ll \frac{x}{\log x}(\log\log x)^{\ell -2}$.
Et enfin, pour $S_3$, on a trivialement pour $x$ assez grand,
\[S_3=0,\]
car pour $x$ assez grand, comme $p_1\leq\frac{x}{p_2\cdots p_\ell}$ et $\frac{x}{p_2\cdots p_\ell}<C_2$ et donc $\sum\limits_{\substack{C_2<q\leq C_3\\ q|p_1-1}}1=0$.
Ainsi $\Ma(x,C_2,C_3)=\bo\left(\frac{x}{\log x}(\log \log x)^{\ell-2}\right)$.

\item
Montrons maintenant que $\Ma(x,C_3,x-1)=\bo\left(\frac{x}{\log^2 x}\right)$. Nous procédons ici de la même façon que Hooley \cite[p212]{Hooley}.
Observons que, comme les rôles des $p_i$ sont symétriques dans la condition $\RR(q,p_1 \ldots p_\ell)$ on peut supposer que $a^{\frac{p_1-1}{q}}\equiv 1 \Mod{p_1}$ et donc $a^{2\frac{p_1-1}{q}}\equiv 1\Mod{p_1}$.\\
Comme $q>x^{\frac{1}{2}}\log x$, pour $p_2\cdots p_\ell\leq x$ fixés, les nombres $p_1$ tels que $p_1 \cdots p_\ell$ soit compté dans $\Ma(x,C_3,x-1)$ doivent diviser le produit positif (éventuellement vide) : 
\begin{equation}\label{prod_a2m}
\prod\limits_{m<x^{\frac{1}{2}}\log^{-1} x\prod\limits_{i=2}^\ell p_i^{-1}}(a^{2m}-1)                                                                                                                                                                                                                                                                                                                                                                                                               \end{equation}

Posons $S_a(\eta_1,\eta_2)=\{p:\ \exists q\in]\eta_1,\eta_2]\ \text{tel que}\ q|p-1,\ a\ \text{est une puissance}\ q\text{-ième}\ \Mod{p}\}$. Ainsi 
\[\Ma(x,C_3,x-1)\leq2\sum\limits_{p_2\cdots p_\ell\leq x}\sum\limits_{\substack{p_1\leq\frac{x}{p_2\cdots p_\ell}\\ p_1\in S_a(C_3,x-1)}}1\]
Les éléments de $S_a(C_3,x-1)$ étant des nombres premiers, leur produit divise \eqref{prod_a2m}. En minorant dans  \eqref{prod_a2m} par 2 chaque $p\in S_a(C_3, x-1)$ on obtient :
\[2^{\#\{p_1\leq\frac{x}{p_2\cdots p_\ell}:\ p_1\in S_a(C_3,x-1)\}}<\prod\limits_{m<x^{\frac{1}{2}}\log^{-1} x\prod\limits_{i=2}^\ell p_i^{-1}}a^{2m}\]
Ainsi, en prenant le logarithme :
\[\#\Big\{p_1\leq\frac{x}{p_2\cdots p_\ell}:\ p_1\in S_a(C_3,x-1)\Big\}<\frac{2\log |a|}{\log 2}\sum\limits_{m<x^{\frac{1}{2}}\log^{-1} x\prod\limits_{i=2}^\ell p_i^{-1}}m=\bo\Bigg(\frac{x}{\log^2 x \prod\limits_{i=2}^\ell p_i^{2}}\Bigg)\]
et donc $\Ma(x,C_3,x-1)\ll\sum\limits_{p_2\cdots p_\ell\leq x}\frac{x}{\log^2 x \prod\limits_{i=2}^\ell p_i^{2}}\ll \frac{x}{\log^2 x}$.
\end{enumerate}
\end{proof}

\subsection{Majoration de $\Ma(x,C_1,C_2)$ sous GRH.}
Nous procédons à la majoration de $\Ma(x,C_1,C_2)$. Pour ce faire nous aurons besoin du résultat suivant de Hooley \cite{Hooley} conditionnel à l'hypothèse de Riemann généralisée pour certains corps de nombres.

\begin{lem}[Hooley, GRH]\label{lemme_hooley}
Soit $q$ un nombre premier, on a alors
\begin{equation*}
 \sum\limits_{\substack{p\leq x\\ q|p-1\\ a\in\qroot(q,p)}}1\ll \frac{1}{q(q-1)}\Li(x)+\bo\big(\sqrt{x}\log(qx)\big).
\end{equation*}

\end{lem}
Nous sommes alors en mesure de fournir la majoration suivante pour $\Ma(x,C_1,C_2)$.

\begin{prop}[GRH]
 Sous l'hypothèse de Riemann généralisée, et avec $C_1$ et $C_2$ comme dans \eqref{def_bornes_Ci}, on a :
 \[\Ma(x,C_1,C_2)\ll \frac{x(\log\log x)^{\ell-1}}{C_1\log x}+\frac{x(\log\log x)^{\ell-2}}{\log x}\log\log\log x.\]
\end{prop}
\begin{proof}
Dans $\Ma$ la condition $a$ vérifie $\RR(q,p_1 \ldots p_\ell)$ implique qu'il existe $p\in\{p_1,\ldots,p_\ell\}$ tel que $a$ vérifie $\RR(q,p)$. Le rôle des $p_i$ étant symétriques, on peut supposer que $a$ vérifie $\RR(q,p_1)$.

 Écrivons
 \begin{align*} 
\Ma(x,C_1,C_2)&\leq\#\left\{p_1\cdots p_\ell\leq x\ \exists q\in ]C_1,C_2],\ a \text{ vérifie }\RR(q,p_1)\right\}\\
&\leq S+\#\left\{p_1\cdots p_\ell\leq x,\ p_1\leq \log^{64}x\right\},
 \end{align*}
 où $S:=\#\left\{p_1\cdots p_\ell\leq x,\ p_2\cdots p_\ell\leq\frac{x}{\log^{64}x},\ \exists q\in ]C_1,C_2],\ a \text{ vérifie }\RR(q,p_1)\right\}$.
 
 La contribution de $p_1\leq (\log x)^{64}$ est négligeable :
 \begin{align*}\#\left\{p_1\cdots p_\ell\leq x,\ p_1\leq \log^{64}x\right\}&\ll \sum\limits_{p_1\leq \log^{64}x}\frac{x}{p_1\log x}(\log\log x)^{\ell-2}\\
 &\ll \frac{x}{\log x}(\log\log x)^{\ell-2}\log\log\log x.\end{align*}
 
 Réécrivons maintenant $S$ :
 \[S\ll\sum\limits_{C_1<q\leq C_2}\sum\limits_{\substack{p_2\cdots p_\ell\leq\frac{x}{q}\\ p_2\cdots p_\ell\leq\frac{x}{\log^{64}x}}}\sum\limits_{\substack{q\leq p_1\leq \frac{x}{p_2\cdots p_\ell}\\ q|p_1-1\\ a\in\qroot(q,p_1)}}1.\]
 
 On découpe la somme sur $p_2\cdots p_\ell$ en trois parties comme suit :
 \begin{align*}
  S_1&:=\sum\limits_{C_1<q\leq C_2}\sum\limits_{\substack{p_2\cdots p_\ell\leq\frac{x}{q^2\log^6 x}\\ p_2\cdots p_\ell\leq\frac{x}{\log^{64}x}}}\sum\limits_{\substack{q\leq p_1\leq \frac{x}{p_2\cdots p_\ell}\\ q|p_1-1\\ a\in\qroot(q,p_1)}}1,\\
  S_2&:=\sum\limits_{C_1<q\leq C_2}\sum\limits_{\substack{\frac{x}{q^2\log^6 x}<p_2\cdots p_\ell\leq\frac{x}{q^2}\log^2 x\\ p_2\cdots p_\ell\leq\frac{x}{\log^{64}x}}}\sum\limits_{\substack{q\leq p_1\leq \frac{x}{p_2\cdots p_\ell}\\ q|p_1-1}}1,\\
  S_3&:=\sum\limits_{C_1<q\leq C_2}\sum\limits_{\frac{x}{q^2}\log^2 x<\substack{p_2\cdots p_\ell\leq\frac{x}{q}\\ p_2\cdots p_\ell\leq\frac{x}{\log^{64}x}}}\sum\limits_{\substack{q\leq p_1\leq \frac{x}{p_2\cdots p_\ell}\\ q|p_1-1\\ a\in\RR(q,p_1)}}1,\\
 \end{align*}
et on a donc $S\ll S_1+S_2+S_3$.

Pour majorer $S_1$ on utilise le lemme \ref{lemme_hooley} dû à Hooley.
\begin{align*}
 S_1&\ll \sum\limits_{C_1<q\leq C_2}\sum\limits_{\substack{p_2\cdots p_\ell\leq\frac{x}{q^2\log^6 x}\\ p_2\cdots p_\ell\leq\frac{x}{\log^{64}x}}}\left(\Li\left(\frac{x}{p_2\cdots p_\ell}\right)\frac{1}{q(q-1)}+\sqrt{\frac{x}{p_2\cdots p_\ell}}\log x\right)
\end{align*}
Par sommation d'Abel, et en utilisant le Théorème de Landau sur le nombre de  presque premiers \cite{Landau_fr}, on vérifie que le terme de reste est assez petit :
\begin{align*}
 \sum\limits_{C_1<q\leq C_2}\sum\limits_{\substack{p_2\cdots p_\ell\leq\frac{x}{q^2\log^6 x}\\ p_2\cdots p_\ell\leq\frac{x}{\log^{64}x}}}\sqrt{\frac{x}{p_2\cdots p_\ell}}\log x\ll\sum\limits_{C_1<q\leq C_2} \frac{x}{q\log^2 x}\ll \frac{x}{\log^2 x}\log\log x.
\end{align*}
Puis pour le terme principal, on ne peut que majorer trivialement les contributions de $p_1,\ldots,p_\ell$ :
\begin{align*}
 \sum\limits_{\substack{p_2\cdots p_\ell\leq\frac{x}{q^2\log^6 x}\\ p_2\cdots p_\ell\leq\frac{x}{\log^{64}x}}}\Li\left(\frac{x}{p_2\cdots p_\ell}\right)&\ll \sum\limits_{p_1\cdots p_\ell\leq x}1\ll \frac{x}{\log x}(\log\log x)^{\ell-1}.
\end{align*}
Comme la somme sur $q$ est un reste de série convergente on peut majorer $S_1$:
\begin{align*}
 S_1&\ll\sum\limits_{C_1<q\leq C_2}\frac{x(\log\log x)^{\ell-1}}{q^2 \log x}+ \frac{x}{\log^2 x}\log\log x \ll \frac{x(\log\log x)^{\ell-1}}{C_1 \log x}+ \frac{x}{\log^2 x}\log\log x .
\end{align*}
Passons à la majoration de $S_2$. 

Notons $E(x,n):=\Big\{q,\ \sqrt{\frac{x}{n \log^6 x}}<q\leq \sqrt{\frac{x}{n}\log^2 x}\Big\}$, en appliquant le Théorème de Brun-Titchmarsh on obtient :

\begin{align*}
 S_2&\ll \sum\limits_{p_2\cdots p_\ell\leq \frac{x}{\log^{64}x}}\sum\limits_{q\in E(x,p_2\cdots p_\ell)}\sum\limits_{\substack{q<p_1\leq \frac{x}{p_2\cdots p_\ell}\\ q|p_1-1}}1\\
 &\ll \sum\limits_{p_2\cdots p_\ell\leq \frac{x}{\log^{64}x}}\sum\limits_{q\in E(x,p_2\cdots p_\ell)}\frac{x}{qp_2\cdots p_\ell\log\left(\frac{x}{qp_2\cdots p_\ell}\right)}.
\end{align*}
On a $\sqrt{\frac{x}{p_2\cdots p_\ell}}\log^{-1} x\leq\frac{x}{qp_2\cdots p_\ell}\leq \sqrt{\frac{x}{p_2\cdots p_\ell}}\log^3 x$ et donc $\log\left(\frac{x}{qp_2\cdots p_\ell}\right)\sim \frac{1}{2}\log\left(\frac{x}{p_2\cdots p_\ell}\right)$. Alors en appliquant le Théorème de Mertens et des manipulations standard sur les logarithmes,

\begin{align*}\hspace*{-3.5cm}
 \sum\limits_{q\in E(x,p_2\cdots p_\ell)}\frac{1}{q}&\leq\log\log\left(\sqrt{\frac{x}{p_2\cdots p_\ell}}\log x\right)-\log\log\left(\sqrt{\frac{x}{p_2\cdots p_\ell}}\log^{-3} x\right)+\bo\left(\log^{-1}\left(\frac{x}{p_2\cdots p_\ell}\right)\right)\\
 &\leq 4\frac{\log\log x}{\log\left(\sqrt{\frac{x}{p_2 \cdots p_\ell}}\right)}+\bo\left(\left(\frac{\log\log x}{\log\left(\sqrt{\frac{x}{p_2 \cdots p_\ell}}\right)}\right)^2\right),\\
 &\ll\frac{\log\log x}{\log\left(\sqrt{\frac{x}{p_2\cdots p_\ell}}\right)}
\end{align*}
Ainsi,

\begin{align*}
 S_2&\ll x\log\log x\sum\limits_{p_2\cdots p_\ell\leq \frac{x}{\log^{64}x}}\frac{1}{p_2\cdots p_\ell\log^2\left(\frac{x}{p_2\cdots p_\ell}\right)}.\\
\end{align*}
On évalue alors cette somme par sommation d'Abel et en utilisant le théorème de Landau sur les presque premiers :
\begin{align*}
 \sum\limits_{p_2\cdots p_\ell\leq \frac{x}{\log^{64}x}}\frac{1}{p_2\cdots p_\ell\log^2\left(\frac{x}{p_2\cdots p_\ell}\right)}&\ll\int_{\ell!}^{\frac{x}{\log^{64}x}}\frac{(\log\log t)^{\ell-2}}{t\log t \log^2\frac{x}{t}}dt\ll (\log\log x)^{\ell-2}\int_{\ell!}^{\frac{x}{\log^{64}x}}\frac{1}{t\log t \log^2\frac{x}{t}}dt.
\end{align*}
On utilise le changement de variable $u=\log x$ et on effectue une décomposition en éléments simples, on obtient :
\[S_2\ll\frac{x}{\log x}(\log\log x)^{\ell-2}.\]

Passons à la majoration de $S_3$. On va procéder comme dans le $ii)$ de la proposition \ref{prop_maj_C2C3(x-1)}

\[S_3:=\sum\limits_{C_1<q\leq C_2}\sum\limits_{\frac{x}{q^2}\log^2 x<\substack{p_2\cdots p_\ell\leq\frac{x}{q}\\ p_2\cdots p_\ell\leq\frac{x}{\log^{64}x}}}\sum\limits_{\substack{q\leq p_1\leq \frac{x}{p_2\cdots p_\ell}\\ q|p_1-1\\ a\in\RR(q,p_1)}}1\]

Comme $a^{\frac{p_1-1}{q}}\equiv 1(p_1)$, $a^{2\frac{p_1-1}{q}}\equiv 1(p_1)$. De plus $\frac{p_1-1}{q}\leq \frac{x}{qp_2\cdots p_\ell}\leq \sqrt{\frac{x}{p_2\cdots p_\ell}}\log^{-1}x$. Ainsi il existe $m\leq \sqrt{\frac{x}{p_2\cdots p_\ell}}\log^{-1}x$ tel que $p_1|a^{m}-1$. On minore chaque $p_i$ par $C_1$ (au lieu de $2$ contrairement à la preuve de la proposition \ref{prop_maj_C2C3(x-1)})

\[C_1^{\#\{C_1\leq p_1\leq \frac{x}{p_2\cdots p_\ell}:\ \exists q\geq \sqrt{\frac{x}{p_2}}\log x,\ a\equiv b^q\Mod{p_1}\}}\leq\prod\limits_{m\leq\sqrt{\frac{x}{p_2\cdots p_\ell}}\frac{1}{\log x}}a^{2m}\]
et donc en prenant de nouveau les logarithmes

\begin{align*}
 \#\{C_1\leq p_1\leq \frac{x}{p_2\cdots p_\ell}:\ \exists q\geq \sqrt{\frac{x}{p_2\cdots p_\ell}}\log x,\ a\equiv b^q\Mod{p_1}\}&\leq\frac{2\log a}{\log C_1}\sum\limits_{m\leq\sqrt{\frac{x}{p_2\cdots p_\ell}}\frac{1}{\log x}}m\\
 &\ll \frac{x}{p_2\cdots p_\ell\log C_1 \log^2 x}
\end{align*}

On reporte dans $S_3$

\begin{align*}
 S_3&\ll \sum\limits_{p_2\cdots p_\ell\leq x}\frac{x}{p_2\cdots p_\ell\log C_1 \log^2 x}\\
 &\ll\frac{x(\log\log x)^{\ell-1}}{\log C_1 \log^2 x}.
\end{align*}

Ainsi

\[\Ma(x,C_1,C_2)\ll\frac{x(\log\log x)^{\ell-1}}{C_1\log x}+\frac{x(\log\log x)^{\ell-2}}{\log x}\log\log\log x.\]

\end{proof}

\subsection{Nouvelle équation fondamentale}
Maintenant en reprenant (\ref{equationfondav1}) et la proposition précédente, on obtient :
\begin{equation}\label{equationfondav2}
\Na(x)=\Na(x,C_1)+\bo\left(\frac{x(\log\log x)^{\ell-1}}{C_1\log x}+\frac{x(\log\log x)^{\ell-2}}{\log x}\log\log\log x\right).
\end{equation}

Exprimons alors $\Na(x,C_1)$ en termes de $\Pa(x,k)$. Déjà, par inclusion-exclusion, on a :
\[\Na(x,C_1)=\sum\limits_{P^+(l')\leq C_1}\mu(l')\Pa(x,l').\]
La section suivante a pour objectif de séparer le plus possible les conditions sur les $p_i$ dans $\Pa(x,k)$.

\section{Expression de $\Pa(x,k)$}\label{expressionpa}
Commençons par démontrer le lemme suivant.

\begin{lem}\label{lem_qroot}
 Soient $u$ et $v$ deux entiers premiers entre eux, $a$ un entier et $p$ un nombre premier, alors $a$ appartient à la fois à $\qroot(u,p)$ et à $\qroot(v,p)$ si et seulement si $a$ appartient à $\qroot(uv,p)$.
\end{lem}
\begin{proof}
 On a trivialement que si $a\in \qroot(uv,p)$ alors $a\in\qroot(u,p)\cap\qroot(v,p)$.
 
 Supposons alors $a\in\qroot(u,p)\cap\qroot(v,p)$. Soient $\nu_1$ et $\nu_2$ tels que $\nu_1^{u}\equiv a\Mod{p_i}$ et $\nu_2^{v}\equiv a\Mod{p_i}$.\\ Puis, comme $\left(u,v\right)=1$, il existe d'après le Théorème de Bézout $\lambda_1$ et $\lambda_2$ tels que $\lambda_1 u+\lambda_2 v=1$.\\ Alors on a\[(\nu_1^{\lambda_2}\nu_2^{\lambda_1})^{uv}\equiv(\nu_1^{u})^{\lambda_2 v}(\nu_2^{v})^{\lambda_1u}\equiv a^{\lambda_1 u+\lambda_2 v}\equiv a\Mod{p}.\]
\end{proof}

Notons $E=\{1,\ldots,\ell\}$, $\mathcal{P}^*(E)$ l'ensemble des parties non vides de $E$.

On déduit du lemme précédent et des définitions de $\Pa(x,k)$ et $h$ données par \eqref{def_Pa} et \eqref{def_h} les conditions suivantes pour les nombres $p_1\cdots p_\ell$ contribuant à $\Pa(x,k)$.
\begin{prop}\label{prop_expl_Pa_p1p2}
Soit $(p_1\ldots,p_\ell)$ un $\ell$-uplet de nombres premiers tel que $p_1\cdots p_\ell\leq x$. Alors $p_1\cdots p_\ell$ est compté dans $\Pa(x,k)$ si et seulement si  les conditions suivantes sont vérifiées :
\begin{itemize}
\item[$\bullet$]$(a,p_1\cdots p_\ell)=1$,
\item[$\bullet$] Il existe une unique factorisation de $k$ sous la forme $k=\prod\limits_{L\in\mathcal{P}^*(E)}k_L$ telle que, en notant $k_L':=\frac{k_L}{(h,k_L)}$ :
\begin{itemize}
\item Pour tout $ i\in E$, $a\in\qroot\left(\prod\limits_{\substack{L\in\mathcal{P}^*(E)\\ i\in L}}k'_L,p_i\right)$ et $p_i\equiv 1\Mod{\prod\limits_{\substack{L\in\mathcal{P}^*(E)\\ i\in L}}k_L}$.
\item Pour tout $ j\in E$, et pour tout sous-ensemble $L$ de $E$ ne contenant pas $j$, on a pour chaque diviseur premier $q$ de $k_L$, $\nu_q(p_j-1)<\nu_q(\lambda(p_1\cdots p_\ell))$.
\end{itemize}
\end{itemize}
\end{prop}
\begin{proof}
Tous les $p_1\cdots p_\ell$ comptés dans $\Pa(x,k)$ vérifient les conditions suivantes :
\begin{enumerate}[i)]
 \item $(a,p_1 \cdots p_\ell)=1$,
 \item $k|\lambda(p_1\cdots p_\ell)$,
 \item $\forall q|k,\ \forall p_i\in M_q(p_1,\ldots,p_\ell),\ a\in \qroot(q,p_i)$.
\end{enumerate}
Définissons pour tout diviseur premier $q$ de $k$ l'ensemble des indices des éléments de l'ensemble $M_q$ associé : $I_q:=\{i\in E,\ p_i\in M_q(p_1,\ldots,p_\ell)\}$.

A chaque $L\subset E$ non vide, on peut associer un diviseur $k_L$ de $k$ défini par $k_L:=\prod\limits_{\substack{q|k\\ L=I_q}}q$ et ainsi $\prod\limits_{L\in\mathcal{P}^*(E)}k_L=k$, les $k_L$ pouvant valoir $1$.

La condition $k|\lambda(p_1\cdots p_\ell)$ peut s'écrire sous la forme $\forall q|k,\ \max\limits_{1\leq i\leq\ell}\nu_q(p_i-1)>0$, ce qui revient à $p_i\equiv 1\Mod{\prod\limits_{i\in L}k_L}$.

La condition iii) est elle équivalente au fait que pour tout $i\in E$, $a\in\qroot(q,p_i)$ pour tout $q$ tel que $p_i \in M_q(p_1,\ldots,p_\ell)$. En appliquant plusieurs fois le lemme \ref{lem_qroot}, cette condition sur $p_i$ est équivalente à  \[a\in\qroot\Bigg(\prod\limits_{\substack{q|k\\ i\in I_q}}q,p_i\Bigg)=\qroot\left(\prod\limits_{\substack{i\in L}}k_L,p_i\right).\]

On note $k_L=k_L' (k_L,h)$. Comme $k$ est sans facteur carré, $\big(k_L',(k_L,h)\big)=1$. D'après le lemme \ref{lem_qroot}, $a\in\qroot(k_L,p_i)$ si et seulement si $a\in\qroot(k_L',p_i)$ et $a\in\qroot((k_L,h),p_i)$. 

Or par définition de $h$, on a toujours $a\in\qroot((k_L,h),p_i)$, et donc $a\in\qroot(k_L,p_i)\Leftrightarrow a\in\qroot(k_L',p_i)$.

Appliquons à nouveau le lemme \ref{lem_qroot}
\[
 a\in\qroot\left(\prod\limits_{\substack{i\in L}}k_L,p_i\right)\Leftrightarrow a\in\bigcap_{i\in L}\qroot(k_L,p_i)\Leftrightarrow a\in\bigcap_{i\in L}\qroot(k_L',p_i)\Leftrightarrow a\in\qroot\left(\prod\limits_{\substack{i\in L}}k_L',p_i\right).
\]
Ainsi notre factorisation de $k$ convient. Soit $\prod \tilde{k}_L$ une décomposition de $k$ qui convient, alors pour tout $q|\tilde{k}_L$, $M_q(p_1,\ldots,p_\ell)=L$ et donc $\tilde{k}_L=k_L$.

\end{proof}

Grâce à la proposition \ref{prop_expl_Pa_p1p2}, on peut exprimer $\Pa(x,k)$ comme suit :
\begin{equation}
\label{Pak1k2}
\Pa(x,k)=\sum\limits_{\prod\limits_{L\in\mathcal{P}^*(E)}k_L=k}\sum\limits_{\substack{p_1\cdots p_\ell\leq x\\ (a,p_1\cdots p_\ell)=1\\ \forall i,\ p_i\equiv 1\Mod{\prod\limits_{\substack{L\in\mathcal{P}^*(E)\\ i\in L}}k_L}\\\forall i,\ a\in\qroot\left(\prod\limits_{\substack{L\in\mathcal{P}^*(E)\\ i\in L}}k'_L,p_i\right)}}1
\end{equation}

où les $k_L$ et $k'_L$ dépendent des $p_i$ comme explicité précédemment. Dans la proposition suivante nous obtenons un découpage de $\Pa(x,k)$ dans lequel les conditions sur les $p_1,\ldots,p_\ell$ sont indépendantes (mis à part la contrainte $p_1\cdots p_\ell\leq x$).

\begin{prop}
 On a :
 \[\Pa(x,k)=\frac{1}{\ell!}\sum\limits_{\substack{\left\{k_L,\ L\in\mathcal{P}^*(E)\right\}\\ \prod\limits_{ L\in\mathcal{P}^*(E)}k_L=k}}\sum\limits_{\substack{\{g_L,\ L\in\mathcal{P}^*(E)\}\\ g_L|k_L^\infty\\ k_L|g_L}}\sum\limits_{\substack{\{c_{iL},\ L\in\mathcal{P}^*(E),\ i\notin L\}\\ c_{iL}|\frac{g_L}{k_L}}}\sum\limits_{\substack{\left\{p_i,\ i\in \{1,\ldots ,\ell\}\right\}\\ p_1\cdots p_\ell\leq x\\ (a,p_1\cdots p_\ell)=1\\ \forall i,\ p_i\equiv 1\Mod{u_i}\\ \forall i,\ (\frac{p_i-1}{u_i},k)=1\\ \forall i,\ a\in\qroot\left(k'_i,p_i\right)}}1,\]
 où pour tout $i\in\{1,\ldots ,\ell\}$, $u_i:=\prod\limits_{\substack{L\in\mathcal{P}^*(E)\\ i\in L}}g_L \prod\limits_{\substack{L\in\mathcal{P}^*(E)\\ i\notin L}}c_{iL}$ et $k'_i:=\prod\limits_{\substack{L\in\mathcal{P}^*(E)\\ i\in L}}k'_L$.
\end{prop}

\begin{proof}
À $k$ et $p_1,\ldots,p_\ell$ fixés notons pour $L\in\mathcal{P}^*(E)$ et $i\in L$, $g_L:=\prod\limits_{q|k_L}q^{\nu_q(p_i-1)}=(k^\infty_L,p_i-1)$ et pour $i\notin L$ notons $c_{iL}:=(k^\infty_L,p_i-1)$. Notons que la définition de $g_L$ ne dépend pas du $i\in L$ choisi et que $c_{iL}|\frac{g_L}{k_L}$. Puis posons pour tout $i$ $k_i:=\prod\limits_{\substack{L\in\mathcal{P}^*(E)\\ i\in L}}k_L$ et $k'_i:=\prod\limits_{\substack{L\in\mathcal{P}^*(E)\\ i\in L}}k'_L$.

On va sommer sur toutes les décompositions de $k$ en $\prod\limits_{L\in\mathcal{P}^*(E)}k_L$ et les $g_L$ et $c_{iL}$ possibles. C'est-à-dire sommer sur ces décompositions de $k$ et sur les $g_L|k^\infty_L$, $k_L|g_L$ et $c_{iL}|\frac{g_L}{k_L}$. 

Ainsi à $k$, $\{k_L,\ L\in\mathcal{P}^*(E)\}$ tels que $\prod\limits_{L\in\mathcal{P}^*(E)}k_L=k$, $\{g_L,\ L\in\mathcal{P}^*(E)\}$ et $\{c_{iL},\ L\in\mathcal{P}^*(E),\ i\notin L\}$ fixés les $p_1\cdots p_\ell$ qui contribuent à $\Pa(x,k)$ pour lesquels les $k_L$, $g_L$ et $c_{iL}$ correspondent sont tels que :
\begin{enumerate}
 \item $(a,p_1\cdots p_\ell)=1$;
 \item pour tout $i$, $p_i\equiv 1(k_i)$;
 \item pour tout $i$, $a\in\qroot\left(k'_i,p_i\right)$;
 \item pour tous $L\in\mathcal{P}^*(E)$ et $i\in L$, $g_L|(p_i-1)$ et $(\frac{p_i-1}{g_L},k_L)=1$;
 \item pour tous $L\in\mathcal{P}^*(E)$ et $i\notin L$, $c_{iL}|(p_i-1)$ et $(\frac{p_i-1}{c_{iL}},k_L)=1$.
\end{enumerate}

Comme pour tout $L$, $k_L|g_L$ la condition $\forall L\in\mathcal{P}^*(E)$, $i\in L$, $g_L|(p_i-1)$ et $(\frac{p_i-1}{g_L},k_L)=1$ implique $p_i\equiv 1(k_i)$. Puis, comme $k$ est sans facteur carré, les conditions :
\begin{enumerate}
 \item pour tout $L\in\mathcal{P}^*(E)$, $i\in L$, $g_L|p_i-1$ et $(\frac{p_i-1}{g_L},k_L)=1$;
 \item pour tout $L\in\mathcal{P}^*(E)$, $i\notin L$, $c_{iL}|p_i-1$ et $(\frac{p_i-1}{c_{iL}},k_L)=1$;
\end{enumerate}
sont ensembles équivalentes à :
\[\text{pour tout }i\in\{1,\ldots ,\ell\},\ \prod\limits_{\substack{L\in\mathcal{P}^*(E)\\ i\in L}}g_L \prod\limits_{\substack{L\in\mathcal{P}^*(E)\\ i\notin L}}c_{iL}|(p_i-1)\text{ et }(\frac{p_i-1}{\prod\limits_{\substack{L\in\mathcal{P}^*(E)\\ i\in L}}g_L \prod\limits_{\substack{L\in\mathcal{P}^*(E)\\ i\notin L}}c_{iL}},k)=1 .\]
Notons alors pour tout $i$, $u_i:=\prod\limits_{\substack{L\in\mathcal{P}^*(E)\\ i\in L}}g_L \prod\limits_{\substack{L\in\mathcal{P}^*(E)\\ i\notin L}}c_{iL}$ et disons, pour raccourcir les notations, que la condition $p_i\equiv 1\Mod{u_i}$ est impliquée par $(\frac{p_i-1}{u_i},k)=1$. Les $p_1\cdots p_\ell$ qui contribuent à $\Pa(x,k)$ pour lesquels les $k_L$, $g_L$ et $c_{iL}$ correspondent sont alors tels que :
\begin{enumerate}
 \item $(a,p_1\cdots p_\ell)=1$;
 \item pour tout $i$, $a\in\qroot\left(k'_i,p_i\right)$;
 \item pour tout $i\in\{1,\ldots ,\ell\}$, $(\frac{p_i-1}{u_i},k)=1$.
\end{enumerate}

On peut alors écrire $\Pa(x,k)$ comme suit :

\[\Pa(x,k)=\frac{1}{\ell!}\sum\limits_{\substack{\left\{k_L,\ L\in\mathcal{P}^*(E)\right\}\\ \prod\limits_{ L\in\mathcal{P}^*(E)}k_L=k}}\sum\limits_{\substack{\{g_L,\ L\in\mathcal{P}^*(E)\}\\ g_L|k_L^\infty\\ k_L|g_L}}\sum\limits_{\substack{\{c_{iL},\ L\in\mathcal{P}^*(E),\ i\notin L\}\\ c_{iL}|\frac{g_L}{k_L}}}\sum\limits_{\substack{\left\{p_i,\ i\in \{1,\ldots ,\ell\}\right\}\\ p_1\cdots p_\ell\leq x\\ (a,p_1\cdots p_\ell)=1\\ \forall i,\ (\frac{p_i-1}{u_i},k)=1\\ \forall i,\ a\in\qroot\left(k'_i,p_i\right)}}1.\]
Le facteur $\frac{1}{\ell!}$ provenant du fait que un même nombre $p_1\cdots p_\ell$ est compté $\ell!$ fois suivant l'ordre des facteurs premiers et comme $p_i\neq p_j$ pour $i\neq j$.

\end{proof}

Afin de pouvoir calculer la dernière somme de $\Pa(x,k)$ nous aurons besoin de contrôler la taille des $g_L$, ce que nous faisons dans la proposition suivante.

\begin{prop}\label{prop_controle_gL}
Pour $k\leq C_1$, $C_5$ une constante arbitrairement grande, $0<t<1$, on a :
\[\Pa(x,k)=\Pa'(x,k)+\bo\left(\frac{C_1 x}{C_5^t\log x}(\log\log x)^{\ell-1}\right),\]
où
\[\hspace*{-2cm}\Pa'(x,k):=\frac{1}{\ell!}\sum\limits_{\substack{\left\{k_L,\ L\in\mathcal{P}^*(E)\right\}\\ \prod\limits_{ L\in\mathcal{P}^*(E)}k_L=k}}\sum\limits_{\substack{\{g_L,\ L\in\mathcal{P}^*(E)\}\\ g_L\leq C_5\\ g_L|k_L^\infty\\ k_L|g_L}}\sum\limits_{\substack{\{c_{iL},\ L\in\mathcal{P}^*(E),\ i\notin L\}\\ c_{iL}|\frac{g_L}{k_L}}}\sum\limits_{\substack{\left\{d_i,\ i\in \{1,\ldots ,\ell\}\right\}\\ d_i|k}}\prod\limits_{i\in \{1,\ldots ,\ell\}}\mu(d_i)\sum\limits_{\substack{\left\{p_i,\ i\in \{1,\ldots ,\ell\}\right\}\\ p_1\cdots p_\ell\leq x\\ (a,p_1\cdots p_\ell)=1\\ \forall i,\ p_i\equiv 1(u_i d_i)\\ \forall i,\ a\in\qroot\left(k'_i,p_i\right)}}1.\] 
\end{prop}
\begin{proof}

Il suffit de montrer que la contribution des $g_L$ tels que $\max g_L> C_5$ est un $\bo\left(\frac{x}{C_5^t\log x}(\log\log x)^{\ell-1}\right)$. Quitte à permuter l'ordre des $p_i$, il suffit de vérifier que la contribution d'un $g_{L_0}> C_5$, avec $L_0\in\mathcal{P}^*(E)$ et $1\in L_0$ est un $\bo\left(\frac{x}{C_5^t\log x}(\log\log x)^{\ell-1}\right)$.

Un nombre presque premier donné ne pouvant être compté dans la somme sur tous les presque premiers plus petits que $x$ que pour une valeur donnée des $g_L$ et $c_{i,L}$, on a la majoration suivante :

\begin{align*}
S:&=\sum\limits_{\substack{\left\{k_L,\ L\in\mathcal{P}^*(E)\right\}\\ \prod\limits_{ L\in\mathcal{P}^*(E)}k_L=k}}\sum\limits_{\substack{\{g_L,\ L\in\mathcal{P}^*(E)\backslash L_0\}\\ g_L|k_L^\infty\\ k_L|g_L}}\sum\limits_{\substack{g_{L_0}\geq C_5\\ g_{L_0}|k_{L_0}^\infty\\ k_{L_0}|g_{L_0}}}\sum\limits_{\substack{\{c_{iL},\ L\in\mathcal{P}^*(E),\ i\notin L\}\\ c_{iL}|\frac{g_L}{k_L}}}\sum\limits_{\substack{\left\{p_i,\ i\in \{1,\ldots ,\ell\}\right\}\\ p_1\cdots p_\ell\leq x\\ \forall i,\ p_i\equiv 1(u_i)}}1\\
&\leq \sum\limits_{\substack{\prod\limits_{L\in\mathcal{P}^*(E)}k_L=k}}\sum\limits_{\substack{g_{L_0}|k_{L_0}^\infty\\ g_{L_0}> C_5}}\sum\limits_{\substack{p_2\cdots p_\ell\leq x}}\sum\limits_{\substack{p_1\leq\frac{x}{p_2\cdots p_\ell}\\ p_1\equiv 1 \Mod{g_{L_0}}}}1.
\end{align*}
On veut montrer que $S\ll \frac{x}{C_5^t\log x}(\log\log x)^{\ell-1}$.

On a vu précédemment que la contribution d'un $p_i\leq (\log x)^{64}$ était négligeable, donc
\[S\ll\sum\limits_{\substack{\prod\limits_{L\in\mathcal{P}^*(E)}k_L=k}}\sum\limits_{\substack{g_{L_0}|k_{L_0}^\infty\\ g_{L_0}> C_5}}\sum\limits_{\substack{p_2\cdots p_\ell\leq \frac{x}{(\log x)^{64}}}}\sum\limits_{\substack{p_1\leq\frac{x}{p_2\cdots p_\ell}\\ p_1 \geq (\log x)^{64}\\ p_1\equiv 1 \Mod{g_{L_0}}}}1,\]

$p_1$ étant premier, la condition $p_1\equiv 1\Mod{g_{L_0}}$ entraîne que $p_1\geq g_{L_0}$.

On majore la somme sur $p_1$ en utilisant le théorème de Brun-Titchmarsh quand $g_{L_0}\leq (\log x)^{10}$ et par $\frac{x}{p_2\cdots p_\ell g_{L_0}}$ quand $g_{L_0}> (\log x)^{10}$. Ainsi $S\ll S_1+S_2$, avec
\begin{align*}
 S_1:=&\sum\limits_{\substack{\prod\limits_{L\in\mathcal{P}^*(E)}k_L=k}}\sum\limits_{\substack{g_{L_0}|k_{L_0}^\infty\\ C_5<g_{L_0}\leq (\log x)^{10}}}\sum\limits_{\substack{p_2\cdots p_\ell\leq \frac{x}{(\log x)^{64}}}}\frac{x}{p_2\cdots p_\ell\varphi(g_{L_0})\log\left(\frac{x}{p_2\cdots p_\ell g_{L_0}}\right)},\\
 S_2:=&\sum\limits_{\substack{\prod\limits_{L\in\mathcal{P}^*(E)}k_L=k}}\sum\limits_{\substack{g_{L_0}| k_{L_0}^\infty\\g_{L_0}> (\log x)^{10}}}\sum\limits_{\substack{p_2\cdots p_\ell\leq \frac{x}{(\log x)^{64}}}}\frac{x}{p_2\cdots p_\ell g_ {L_0}}.
\end{align*}

Pour la somme $S_2$ on utilise la méthode de Rankin pour majorer la somme sur $g_{L_0}$ :
\[\sum\limits_{\substack{g_{L_0}| k_{L_0}^\infty\\ g_{L_0}> (\log x)^{10}}}\frac{1}{g_{L_0}}\leq \frac{1}{(\log x)^5}\sum\limits_{\substack{g_{L_0}| k_{L_0}^\infty}}\frac{1}{g_{L_0}^{\frac{1}{2}}}\leq \frac{1}{(\log x)^5}\prod\limits_{p|k_{L_0}}\left(\frac{p^{\frac{1}{2}}}{p^{\frac{1}{2}}-1}\right)\ll \frac{C_1}{(\log x)^5}.\]

En reportant dans la somme $S_2$ on obtient $S_2\ll\frac{C_1x}{(\log x)^5}(\log\log x)^{\ell-1}$.

Dans la somme $S_1$, $g_{L_0}\ll (\log x )^{10}$ tandis que $\frac{x}{p_2\cdots p_\ell }\geq (\log x)^{64}$ donc $\log\left(\frac{x}{p_2\cdots p_\ell g_{L_0}}\right)\gg \log\left(\frac{x}{p_2\cdots p_\ell}\right)$.

De plus $\frac{1}{\varphi(g_{L_0})}\leq \frac{1}{g_{L_0}}\frac{k_{L_0}}{\varphi(k_{L_0})}\ll\frac{\log\log C_1}{\varphi(g_{L_0})}$.

Nous appliquons de nouveau la méthode de Rankin avec un paramètre $t\in]0,1[$. On peut majorer la somme sur $g_{L_0}$ par :
\[\sum\limits_{\substack{g_{L_0}|k_{L_0}^\infty\\ C_5<g_{L_0}\leq (\log x)^{10}}}\frac{1}{\varphi(g_{L_0})}\ll \frac{C_1}{C_5^t}.\]
Ainsi $S_2\ll \frac{C_1}{C_5^t}\frac{x}{\log x}(\log\log x)^{\ell -1}$.
 \end{proof}

Afin de calculer les dernières sommes de $\Pa'(x,k)$ nous aurons besoin de plusieurs résultats de théorie algébrique des nombres, qui seront l'objet de la section suivante.

\section{Correspondance entre les $\ell$-presque premiers recherchés et les idéaux premiers de certains corps de nombres}\label{theoriealgebrique}

Dans toute cette partie $m$ est un entier naturel, ${m'}$ un entier naturel sans facteur carré qui divise $m$, $p$ un nombre premier et $a$ un entier qui n'est ni égal à $-1$ ni un carré et tel que $(a,p)=1$. On suppose de plus que pour tout $q|{m'}$, $q$ premier, $a$ n'est pas une $q$-ième puissance. Commençons par la proposition suivante :

\begin{prop}
L’existence de solutions pour l'équation $\nu^{{m'}}\equiv a\Mod{p}$ et la condition $p\equiv 1\Mod{m'}$ sont ensembles équivalentes au fait que l'équation $\nu^{{m'}}\equiv a\Mod{p}$ a exactement ${m'}$ racines. 
\end{prop}
\begin{proof}
Supposons qu'il existe $\nu$ tel que $\nu^{{m'}}\equiv a\Mod{p}$ et que ${m'}|(p-1)$. Alors $x^{{m'}}\equiv 1\Mod{p}$ a ${m'}$ solutions, que nous pouvons expliciter. Soit $\alpha$ une racine primitive modulo $p$, les racines $m'$-ième de $1$ sont de la forme $x=\alpha^{\left(\frac{p-1}{{m'}}\right)\lambda}$, avec $0\leq\lambda< {m'}$.\\ Alors pour chacune de ces racines, $(\nu x)^{{m'}}\equiv a\Mod{p}$ et ainsi  le polynôme $X^{m'}-a$ a au moins ${m'}$ racines dans $(\Z/p\Z)^*$. Comme c'est un polynôme de degré ${m'}$ on déduit que ce sont les seules. \\ \ \\
Supposons maintenant que $\nu^{{m'}}\equiv a\Mod{p}$ a exactement ${m'}$ racines. Alors le morphisme de groupe $\varphi:\begin{array}{ccc}
(\Z/p\Z)^*&\rightarrow&(\Z/p\Z)^*\\
\nu&\mapsto&\nu^{{m'}}
\end{array}$ a un noyau $\Ker \varphi$ de cardinal $m'$.
 
Or $\Card(\Ker\varphi)|\Card ((\Z/p\Z)^*)$ et on a bien ${m'}|p-1$.

\end{proof}

D'après le Théorème de Kummer (\cite{Neukirch} Prop I.8.3) cette propriété est aussi équivalente au fait que $p\nmid {m'}$ et $p$ se factorise dans $\Q(\sqrt[{m'}]{a})$ comme produit de ${m'}$ idéaux premiers distincts.\\

Montrons maintenant la proposition suivante qui est un résultat classique de théorie algébrique des nombres.

\begin{prop}
Les deux assertions suivantes sont équivalentes :
\begin{itemize}
\item[$\bullet$]$p\equiv 1\Mod{m}$,
\item[$\bullet$]$p\nmid m$ et $p$ se factorise dans $\Q(\xi_m)$ en $\varphi(m)$ idéaux premiers distincts.
\end{itemize}
\end{prop}
\begin{proof}
Soit $\Phi_m$ le $m$-ième polynôme cyclotomique. Les racines de $\Phi_m$ modulo $p$ sont des solutions de l'équation $x^m\equiv 1 \Mod{p}$.\\
Soit $\alpha$ une racine primitive modulo $p$. Alors pour $0\leq r\leq p-1$, $(\alpha^r)^m\equiv 1 (p)$ si et seulement si $rm\equiv 0\Mod{p-1}$. Posons $d=(m,p-1)$. Alors
\begin{align*}
\alpha^{rm}\equiv 1 (p) &\Leftrightarrow r\frac{m}{d}\equiv 0\left(\frac{p-1}{d}\right)\\
&\Leftrightarrow \frac{p-1}{d}|r
\end{align*}
Ainsi les solutions de $x^m\equiv 1 (p)$ sont les $\alpha^{\lambda\frac{p-1}{d}}$ avec $0\leq \lambda< d$. Mais ce sont aussi les solutions de $x^d\equiv 1 (p)$, ainsi les facteurs irréductibles de degré 1 de $\Phi_m\ \Mod{p}$ sont ceux de $\Phi_d\ \Mod{p}$.\\
Les $\alpha^{\lambda\frac{p-1}{d}}$ avec $\lambda<d$, $(\lambda,d)=1$ sont des racines primitives $d$-ièmes de l'unité et donc $\Phi_d\ \Mod{p}=\prod\limits_{\substack{(\lambda,d)=1\\1\leq \lambda<d}}(X-\alpha^{\lambda\frac{p-1}{d}})$. Ainsi $\deg(\Phi_d)=\varphi(d)$.\\
Alors d'après le Théorème de Kummer $p$ se décompose en $\varphi(m)$ idéaux premiers distincts dans $\Q(\xi_m)$ si et seulement si $d=m$, i.e $m|p-1$.
\end{proof}

La proposition suivante, nous permettra de passer de conditions dans $\Q(\xi_m)$ et $\Q(\sqrt[{m'}]{a})$ à des conditions dans $\Q(\sqrt[{m'}]{a},\xi_m)$.

\begin{prop}
 Les trois conditions suivantes $p\nmid m$, $p$ se décompose en $\varphi(m)$ idéaux premiers distincts dans $\Q(\xi_m)$ et $p$ se décompose en ${m'}$ idéaux premiers distincts dans $\Q(\sqrt[{m'}]{a})$ ont lieu si et seulement si $p\nmid m$ et $p$ se factorise dans $K=\Q(\sqrt[{m'}]{a},\xi_m)$ comme produit d'idéaux premiers distincts et de norme $p$.
\end{prop}
\begin{proof}
 Supposons que $p\nmid m$, $p$ se décompose en $\varphi(m)$ idéaux premiers distincts dans $\Q(\xi_m)$ et $p$ se décompose en ${m'}$ idéaux premiers distincts dans $\Q(\sqrt[{m'}]{a})$.\\
 La décomposition de $(p)$ dans $\Q(\xi_m)$ est de la forme $(p)=\mathfrak{p}_1\cdots \mathfrak{p}_{\varphi(m)}$ où pour tout $1\leq i\leq\varphi(m)$, $\mathfrak{p}_i$ est un idéal premier de norme $p$, $\mathfrak{p}_i\neq\mathfrak{p}_j$.\\
 Notons $\overline{\Phi_m}=\Phi_m\Mod{p}$, alors $\overline{\Phi_m}$ se factorise comme suit $\overline{\Phi_m}(X)=(X-u_1)\cdots(X-u_{\varphi(m)})$ et on a, quitte à modifier l'ordre des indices $\mathfrak{p}_i=(p,\xi_m-u_i)$.\\ \ \\
 De la même manière, en notant $B(X)=X^{m'}-a$, $\overline{B}(X)=B(X)\Mod{p}$, on a $\overline{B}(X)=(X-v_1)\cdots(X-v_{{m'}})$.\\
 Comme $K=\Q(\xi_m)(\sqrt[{m'}]{a})$, on peut noter $P$ le polynôme minimal de $\sqrt[{m'}]{a}$ sur $\Q(\xi_m)$ et $\overline{P}=P\Mod{p}$. Ainsi $\overline{P}|\overline{B}$ et donc $\overline{P}$ est scindé. Puis quitte à changer l'ordre des racines, $\overline{P}(X)=(X-v_1)\cdots(X-v_{s})$, où $s\leq {m'}$.\\
 A fortiori pour chaque $1\leq i\leq\varphi(m)$, $P\Mod{\mathfrak{p}_i}$ est scindé et $P\equiv(X-v_1)\cdots(X-v_{s}) \Mod{\mathfrak{p}_i}$.\\
 On applique le théorème de Kummer à l'extension $K$ de $\Q(\xi_m)$. Pour $1\leq i\leq\varphi(m)$, $P\Mod{\mathfrak{p}_i}$ étant scindé, $\mathfrak{p}_i$ se décompose sur $K$ en $\mathfrak{p}_i\mathcal{O}_K=\mathfrak{p}_{i_1}\cdots\mathfrak{p}_{i_{r_i}}$ avec $r_i= s$ et $N_{K/\Q}(\mathfrak{p}_{i_j})=N_{\Q(\xi_m)/\Q}(\mathfrak{p}_i)=p$. Puis $p\mathcal{O}_K=\prod\limits_{i=1}^{\varphi(m)}\mathfrak{p}_i\mathcal{O}_K=\prod\limits_{\substack{1\leq i\leq\varphi(m)\\ 1\leq j\leq r_i}}\mathfrak{p}_{i_j}$ où $N_{K/\Q}(\mathfrak{p}_{i_j})=p$.\\ \ \\\
 Supposons maintenant que $p\nmid m$ et $p$ se factorise dans $K=\Q(\sqrt[{m'}]{a},\xi_m)$ comme produit d'idéaux premiers distincts et de norme $p$.\\
 Alors $p\mathcal{O}_K=\mathfrak{p}_1\cdots\mathfrak{p}_s$ où $s=[K:Q]$ et pour tout $1\leq i\leq s$, $N_{K/\Q}(\mathfrak{p}_i)=p$.\\ 
 Soit $L=\Q(\xi_m)$ et $p\mathcal{O}_L=Q_1\cdots Q_r$ la décomposition de $p$ en produit d'idéaux premiers dans $L$.\\
 Soit $f_i$ pour $1\leq i\leq s$, tel que $N_{K/\Q}(Q_i)=p^{f_i}$. Alors $f_i$ est le degré de l'extension $\mathcal{O}_L/ Q_i \mathcal{O}_L$ par rapport à $\F_p$.\\
 Or pour tout $i\in\{1,\ldots,r\}$ il existe au moins un $j\in\{1,\ldots,s\}$ tel que $\mathfrak{p}_j$ intervienne dans la décomposition de $Q_i$ dans $\mathcal{O}_K$. Ainsi
 \[[\mathcal{O}_K/\mathfrak{p}_j\mathcal{O}_K:\F_p]=[\mathcal{O}_K/\mathfrak{p}_j\mathcal{O}_K:\mathcal{O}_L/Q_i\mathcal{O}_L][\mathcal{O}_L/Q_i\mathcal{O}_L:\F_p]\]
 
 Par hypothèse $[\mathcal{O}_K/\mathfrak{p}_j\mathcal{O}_K:\F_p]=1$ et donc $f_i=[\mathcal{O}_L/Q_i\mathcal{O}_L:\F_p]=1$, ainsi $N_{K/\Q}(Q_i)=p$ pour tout $1\leq i\leq r$ et donc $p$ se décompose en $[\Q(\xi_m):\Q]=\varphi(m)$ idéaux premiers distincts dans $\Q(\xi(m))$.\\
 De la même manière $p$ se décompose en $[\Q(\sqrt[{m'}]{a}):\Q]={m'}$ idéaux premiers distincts dans $\Q(\sqrt[{m'}]{a})$.
\end{proof}

On déduit des trois proposions précédentes que $p\equiv 1(m)$ et $a\in\qroot({m'},p)$ sont ensembles équivalentes au fait que $p\nmid m$ et $p$ se factorise complètement dans $K=\Q(\sqrt[{m'}]{a},\xi_m)$ comme produit d'idéaux premiers distincts de norme $p$.\\

Nous allons maintenant calculer explicitement le degré de $\Q(\sqrt[{m'}]{a},\xi_m)$.

\begin{prop}\label{propdegre}
Soient ${m'},m$ deux entiers naturels tels que ${m'}|m$, ${m'}$ sans facteur carré et $a$ un entier qui ne soit pas un carré. Notons $K=\Q(\sqrt[{m'}]{a},\xi_m)$, où $\xi_m$ est une racine $m$-ième de l'unité. Alors,
\[[K:\Q]=\frac{{m'} \varphi(m)}{\varepsilon({m'},m)}.\]
Où $\varepsilon({m'},m)$ est donné par la formule suivante, en prenant $a_1$ la partie sans facteur carré de $a$ (i.e $a=a_1 a_2^2$, avec $a_1$ sans facteur carré),
\begin{equation}
\label{13}
\varepsilon({m'},m)=\left\{\begin{array}{l}
2\ \text{si}\ 2|{m'},\ 2||m,\ a_1|m\ \text{et}\ a_1\equiv 1\Mod{4}\\
2\ \text{si}\ 2|{m'},\ 4||m,\ a_1|m\ \text{et}\ a_1\equiv 1\Mod{2}\\
2\ \text{si}\ 2|{m'},\ 8|m\ \text{et}\ a_1|m\\
1\ \text{sinon}
\end{array}\right..
\end{equation}

\end{prop}

\begin{proof}
Remarquons que :
\[[K:\Q]=[K:\Q(\xi_m)][\Q(\xi_m):\Q]=\varphi(m)[K:\Q(\xi_m)].\]
Comme $\Q(\xi_m)/\Q$ est une extension galoisienne on a
\[[K:\Q(\xi_m)]|{m'}\]
Posons ${m'}=\delta[K:\Q(\xi_m)]$. Alors, si $q$ est un facteur premier de $\delta$, le degré $[\Q(\xi_m)(\sqrt[q]{a}):\Q(\xi_m)]$ vaut soit $1$ soit $q$. De plus,
\[[\Q(\xi_m)(\sqrt[q]{a}):\Q(\xi_m)]\text{ divise }[K:\Q(\xi_m)]=\frac{{m'}}{\delta}.\]

Or $(\frac{{m'}}{\delta},q)=1$ car ${m'}$ est sans facteur carré, donc $[\Q(\xi_m)(\sqrt[q]{a}):\Q(\xi_m)]=1$ et ainsi $\sqrt[q]{a}\in\Q(\xi_m)$. Montrons que $q<3$.\\
Déjà $\Q(a^{\frac{1}{q}},\xi_q)\subset \Q(\xi_m)$ et $\Q(\xi_m)/\Q$ est une extension abélienne car son groupe de Galois est $(\Z/m\Z)^*$. Soit $L$ une sous extension galoisienne de $\Q(\xi_m)$, alors
\[\Gal(L/\Q)\cong\frac{\Gal(\Q(\xi_m)/\Q)}{\Gal(\Q(\xi_m)/L)}\]
et donc $\Gal(L/\Q)$ est abélien.\\
Supposons  que $q>2$ et montrons que le groupe de galois de $\Q(a^{\frac{1}{q}},\xi_q)/\Q$ n'est pas abélien. Par l'absurde on suppose que ce dernier est abélien.\\
Alors l'extension $\Q(a^{\frac{1}{q}},\xi_q)/\Q(a^{\frac{1}{q}})$ est galoisienne et son groupe de Galois $H$ est un sous-groupe de $G_q:=\Gal\left(\Q(a^{\frac{1}{q}},\xi_q)/\Q\right)$. Comme $G_q$ est abélien $H$ est un sous-groupe invariant de $G_q$. D'après la correspondance de Galois, cela entraîne que $\Q(a^{\frac{1}{q}})/\Q$ est galoisienne de groupe de Galois $G_q/H$.\\
Mais $\Q(a^{\frac{1}{q}})/\Q$ n'est pas galoisienne car $\Q(a^{\frac{1}{q}})$ ne contient pas toutes les racines du polynôme $X^q-a$ si $q\geq 3$.\\
On en déduit que $q$ ne peut être impair et donc $\delta$ vaut $1$ ou $2$, car $m'$ est sans facteur carré.\\
Comme $(q,\frac{{m'}}{q})=1$ il existe deux entiers $u$ et $v$ tels que $uq+v\frac{{m'}}{q}=1$ et donc $a^{\frac{1}{{m'}}}=a^{\frac{uq+v\frac{{m'}}{q}}{{m'}}}=a^{\frac{uq}{{m'}}}a^{\frac{v}{q}}$, ainsi
\[K=\Q(\xi_m)(\sqrt[\frac{{m'}}{q}]{a},\sqrt[q]{a})\]
il vient que si $\sqrt{a}\in\Q(\xi_m)$ alors $\delta=2$ et donc $\delta=2$ si et seulement si $\sqrt{a}\in\Q(\xi_m)$.\\
 Posons
 \[a=a_1 a_2^2\] 
 où $a_1$ est sans facteur carré et éventuellement négatif. \\
 On utilise la caractérisation des sous-corps quadratiques d'un corps cyclotomique (\cite{Weintraub}, Cor 4.5.4).\\
 
 Si $2||m$, comme $a_1$ est sans facteur carré, $\Q(\sqrt{a_1})$ est dans $\Q(\xi_m)$ si et seulement si $a_1$ est de la forme $(-1|D)D$, où $(\cdot|\cdot)$ est le symbole de Jacobi et $D$ est un diviseur positif impair de $m$ différent de $1$.\\
 Ainsi $\delta$ vaut $2$ si et seulement si $a_1|m$ et $a_1$ est un entier impair de même signe que le symbole de Jacobi $(-1\left||a_1|)\right.$. Comme $a$ n'est pas un carré parfait, $a_1\neq1$ et on peut calculer $(-1\left||a_1|)\right.$ :
 \begin{align*}
 (-1\left||a_1|)\right.&=(-1)^{\frac{|a_1|-1}{2}}\\&=\left\{\begin{array}{rcl}1&\ \text{si}\ &|a_1|\equiv 1\Mod{4}\\ -1&\ \text{si}\ &|a_1|\equiv 3\Mod{4} 
 \end{array}\right.\\
 &=\left\{\begin{array}{rcrclcrcl}1&\ \text{si}\ &(a_1\equiv 1\Mod{4}&\ \text{et}\ &a_1>0)&\ \text{ou}\ &(a_1\equiv 3\Mod{4}&\ \text{et}\ &a_1<0)\\ -1&\ \text{si}\ &(a_1\equiv 3\Mod{4}&\ \text{et}\ &a_1>0)&\ \text{ou}\ &(a_1\equiv 1\Mod{4}&\ \text{et}\ &a_1<0)
 \end{array}\right.\\
 \end{align*}
 Ainsi, on a $\delta=2$ avec $2||m$ si et seulement si $2|{m'}$, $a_1|m$ et $a_1\equiv 1\Mod{4}$.\\
 Si $4||m$, alors $\Q(\sqrt{a_1})$ est dans $\Q(\xi_m)$ si et seulement si $a_1$ ou $-a_1$ est comme ci-dessus, et on a donc $\delta=2$ avec $4||m$ si et seulement si $2|{m'}$, $a_1|m$ et $a_1\equiv 1\Mod{2}$.\\
 Si $8|m$, alors $\Q(\sqrt{a_1})$ est dans $\Q(\xi_m)$ si et seulement si $a_1$, $-a_1$, $a_1/2$ ou $-a_1/2$ est comme ci-dessus, et on a donc $\delta=2$ avec $8|m$ si et seulement si $2|{m'}$ et $a_1|m$.\\
\end{proof}

\section{Méthode de Selberg-Delange}

L'objectif de cette section est de démontrer le Théorème \ref{thm_finale_Selberg_Delange}.

D'après la section précédente,  la quantité que l'on souhaite évaluer est :
\[\sum\limits_{\substack{p_1\cdots p_\ell\leq x\\ (a,p_1\cdots p_\ell)=1\\ \forall i,\ p_i\equiv 1(\upsilon_i)\\a\in \qroot\left(\kappa_i, p_i\right)}}1=\sum\limits_{\substack{m_1\cdots m_\ell\leq x\\ m_i\in P_i}}1,\]
où les $\upsilon_1,\ldots,\upsilon_\ell$ sont des entiers plus petits qu'une constante $C$, les $\kappa_1,\ldots,\kappa_\ell$ sont des entiers sans facteur carré tels que pour tout $1\leq i\leq \ell$, $\kappa_i|\upsilon_i$, $P_i$ est l'ensemble des nombres premiers $p$ tels que $p\nmid\ a \upsilon_i$ et $p$ se factorise dans $K_i:=\Q(\sqrt[\kappa_i]{a},\xi_{\upsilon_i})$ comme produits d'idéaux premiers distincts de norme $p$.

Le cas $\ell=1$ a été traité par Hooley \cite{Hooley}, en utilisant un résultat proche du théorème de densité de Tchebotarev et une majoration du discriminant de l'extension considérée. Pour $\ell=2$, il est possible de procéder de la même façon que pour $\ell=1$ en utilisant la méthode de l'hyperbole. Cependant pour $\ell\geq 3$ il devient trop compliqué de contrôler les termes d'erreurs qui émergent de la méthode de l'hyperbole, nous allons donc utiliser la méthode de Selberg-Delange. Nous suivrons la méthode de Selberg-Delange telle que décrite par Tenenbaum \cite{Tenenbaum_fr} chapitre II.5, en comparant une fonction à un produit de puissances de fonctions zêta de Dedekind associées aux corps $K_i$, similairement à un cas traité par Hanrot, Tenenbaum et Wu \cite{HTW}.

Définissons, pour $\boldsymbol{r}:=(r_1,\ldots,r_\ell)\in \N^\ell$, $n\in \N$, 
\[c_{\boldsymbol{r}}(n):=\sum\limits_{m_1\cdots m_\ell=n}\mu^2(n)\prod\limits_{i=1}^\ell\left(\id\left(\Omega(m_i)=r_i\right)\id\left(p|m_i\Rightarrow p\in P_i\right)\right),\]
et donc, en posant $\overline{1}=\left(1,\ldots,1\right)$, $\sum\limits_{\substack{m_1\cdots m_\ell\leq x\\p|m_i\Rightarrow  p\in P_i\\ m_i \text{ premier}}}1=\sum\limits_{n\leq x}c_{\overline{1}}(n)$.\\
 
Notons alors $'\boldsymbol{r}\geq 0'$ lorsque pour tout $1\leq i\leq \ell$, $r_i\geq0$, puis pour $\boldsymbol{z}:=(z_1,\ldots,z_\ell)\in \C^\ell$ :
\[\alpha_{\boldsymbol{z}}(n):=\sum\limits_{\boldsymbol{r}\geq 0}c_{\boldsymbol{r}}(n)\prod\limits_{i=1}^\ell z_i^{r_i}=\sum\limits_{\substack{m_1\cdots m_\ell =n\\
p|m_i\Rightarrow p\in P_i}}\mu^2(n)\prod\limits_{i=1}^\ell z_i^{\omega(m_i)}.\]
Remarquons que $\alpha_{\boldsymbol{z}}(n)$ est multiplicative, vu que c'est un produit de convolutions de fonctions multiplicatives.

Enfin, pour $\Re(s)>1$, la série de Dirichlet associée à $\alpha_{\boldsymbol{z}}$ est :
\[
 F(s,{\boldsymbol{z}}):=\sum\limits_{n\geq 0}\frac{\alpha_{\boldsymbol{z}}(n)}{n^s}=\prod\limits_p\left(1+\sum\limits_{\nu\geq 1}\frac{\alpha_{\boldsymbol{z}}(p^\nu)}{p^{\nu s}}\right)=\prod\limits_{i=1}^\ell\left(\prod\limits_{p\in P_i}\left(1+\frac{z_i}{p^s}\right)\right).
\]
$F$ admet alors un prolongement analytique en une fonction méromorphe sur $\C\backslash\{1\}$ en tant que fonction de $s$. L'étude de $F$ par la méthode de Selberg-Delange nous donnera une expression de $\sum\limits_{n\leq x}\alpha_{\boldsymbol{z}}(n)$, dont nous extrairons la quantité qui nous intéresse par la formule de Cauchy.

Définissons les objets dont nous aurons besoin pour appliquer la méthode de Selberg-Delange. Pour $i\in\{1,\ldots,\ell\}$, on notera $\zeta_{K_i}$ la fonction zêta de Dedekind associée à $K_i$ et $n_i$ le degré de $K_i$ sur $\Q$. On définit $\zeta_{K_i,f}(s)$ comme étant la partie de $\zeta_{K_i}$ portant sur les idéaux premiers de norme une puissance $f$-ième : $\zeta_{K_i,f}(s):= \prod\limits_{\substack{\mathfrak{p}\\ \exists p,\ N(\mathfrak{p})=p^f}}\left(1-\frac{1}{p^{fs}}\right)^{-1}$.

De plus on décompose $F(s,{\boldsymbol{z}})$ en un produit de fonction $F_i$ : $F_i(s,z_i):=\prod\limits_{p\in P_i}\left(1+\frac{z_i}{p^s}\right)$. Nous allons voir que les fonctions $F_i$ se comportent similairement à des puissances des fonctions zêta de Dedekind associées à $K_i$, pour ce faire on pose
\begin{align}
 G_i(s,z_i):&=F_i(s,z_i)\zeta_{K_i}^{-\frac{z_i}{n_i}}(s)\label{def_G}\\
 &=\prod\limits_{p\in P_i}\left(\left(1+\frac{z_i}{p^s}\right)\left(1-\frac{1}{p^s}\right)^{z_i}\right)\prod\limits_{\substack{\mathfrak{p}\\ N(\mathfrak{p})|a\upsilon_i\nonumber\\ N(\mathfrak{p})=p}}\left(1-\frac{1}{p^s}\right)^{\frac{z_i}{n_i}}\prod\limits_{\substack{f|n_i\\ f\geq 2}}\zeta_{K_i,f}(s)^{-\frac{z_i}{n_i}},\nonumber
\end{align}
et $G(s,{\boldsymbol{z}}):=\prod\limits_{i}G_i(s,z_i)$. Nous définissons aussi les fonctions $Z_i(s,z):=\left((s-1)\zeta_{K_i}(s)\right)^{\frac{z_i}{n_i}}$ et la fonction $Z(s,{\boldsymbol{z}}):=\prod\limits_{i}Z_i(s,z_i)$, qui joueront un rôle important dans la suite.

 Les fonction zêta de Dedekind disposent de régions sans zéro semblables à celle de la fonction zêta de Riemann (\cite{Lee21}), comme nous travaillons avec des extensions dont le degré est borné on peut trouver une région sans zéro commune à toutes nos fonctions zêta de Dedekind, ce que l'on résume dans la proposition suivante.
 
 \begin{prop}\label{prop_common_zero_free_region}
  Il existe des constantes $C_6$, $C_7$ et $C_8$ ne dépendant que de $C$ telles que $\zeta_{K_i}$ ne s'annule pas pour
  
  \[\sigma\geq 1-\frac{1}{C_6\log t+C_7},\ \ \ \text{et}\ |t|\geq 1,\]
  \[\sigma\geq 1-C_8,\ \ \ \text{et}\ |t|\leq 1,\]
  et cela pour tout $i\in\{1,\ldots,\ell\}$.
 \end{prop}

Afin d'appliquer la méthode de Selberg-Delange nous devons nous assurer que la fonction $G$ se comporte convenablement dans la région sans zéro décrite ci-dessus. On note $|{\boldsymbol{z}}|:=\max\limits_i\{|z_i|\}$.
\begin{prop}\label{prop_maj_G}
 Soit $A$ un réel strictement positif. Pour $|{\boldsymbol{z}}|\leq A$ et $s$ dans la région sans zéro décrite dans la proposition \ref{prop_common_zero_free_region}, $\mathfrak{Re}\ s\leq2$ on a :
 \[G(s,{\boldsymbol{z}})\ll_{A,\varepsilon} 1.\]
\end{prop}

\begin{proof}
 Déjà pour tout $i$, comme tous les $\upsilon_i$ sont bornés,
 \[\prod\limits_{\substack{\mathfrak{p}\\ N(\mathfrak{p})|a\upsilon_i\\ N(\mathfrak{p})=p}}\left(1-\frac{1}{p^s}\right)^{\frac{z_i}{n_i}}\prod\limits_{\substack{f|n_i\\ f\geq 2}}\zeta_{K_i,f}(s)^{-\frac{z_i}{n_i}}\ll_A 1.\]
Définissons alors pour tout $i$, $\tilde{G}_i(s,z_i):=\prod\limits_{p\in P_i}\left(\left(1+\frac{z_i}{p^s}\right)\left(1-\frac{1}{p^s}\right)^{z_i}\right)=\sum\limits_{n\geq 1}\frac{b_{i,z_i}(n)}{n^s}$, où $b_{i,z_i}$ est la fonction multiplicative dont les valeurs sur les puissances de nombres premiers sont déterminées par l'identité :
\[1+\sum\limits_{\nu\geq 1}b_{i,z_i}(p^\nu)\xi^\nu=\left(1+\xi z_i\right)\left(1-\xi\right)^{z_i},\ |\xi|<1,\ p\in P_i,\]
et $b_{i,z_i}(p^\nu)=0$ pour $p\notin P_i$.\\
Ensuite pour $p\in P_i$, on a
\[b_{i,z_i}(p)=\left.\frac{\partial (1+\xi z_i)(1-\xi)^{z_i}}{\partial \xi}\right\vert_{\xi=0}=0.\]
Puis l'inégalité de Cauchy prise sur le cercle $|\xi|=\frac{1}{\sqrt{2}}$ implique, pour $|z_i|\leq A$, $|b_{i,z_i}(p^\nu)|\leq M 2^{\frac{\nu}{2}}$, avec $M:=\underset{|z_i|\leq A,\ |\xi|\leq \frac{1}{\sqrt{2}}}{\sup}\left[(1+\xi z_i)(1-\xi)^{z_i}\right]$.
 Ainsi, pour $\sigma>\frac{1}{2}$, comme $b_{i,z_i}(p)=0$,
 \[\sum\limits_p \sum\limits_{\nu\geq 1}\frac{|b_{i,z_i}(p^{\nu})|}{p^{\nu\sigma}}\leq 2M\sum\limits_{p}\frac{1}{p^\sigma(p^\sigma-\sqrt{2})}\leq \frac{cM}{\sigma-\frac{1}{2}},\]
 où $c$ est une constante absolue. Ainsi $G_i(s,z_i)$ est absolument convergent pour $\sigma >\frac{1}{2}$, et pour $\sigma\geq \frac{1}{2}+\varepsilon$,
 \[G_i(s,z_i)\ll_{A,\varepsilon} 1.\]
 \end{proof}

 Nous aurons aussi besoin de contrôler $\zeta_{K_i}$ à droite de la droite critique. Pour ce faire on va utiliser un lemme de Wang \cite{Wang} qui découle d'une majoration de Heath-Brown \cite{Heath-Brown88} pour $\zeta_{K_i}$ sur la droite critique.
 
 \begin{lem}[Wang]\label{lem_Wang}
  Soit $K$ une extension algébrique de degré $n$ et $\eta>0$. Alors par le principe de Phragmén-Lindelöf dans la bande $\frac{1}{2}\leq \sigma \leq 1+\varepsilon$ :
  \[\zeta_K(\sigma+it)\ll_\eta (1+|t|)^{\frac{n}{3}(1-\sigma)+\varepsilon},\ \text{pour}\ |t|\geq \eta.\]
 \end{lem}

 Enfin nous avons besoin du développement de Taylor de $Z$ autour de $1$.
 
 \begin{prop}\label{prop_taylor_Z}
 La fonction $Z(s,{\boldsymbol{z}})$ est holomorphe dans le disque $|s-1|<C_8$, et y admet la représentation en série de Taylor suivante :
 
\[Z(s,{\boldsymbol{z}})=\sum\limits_{j\geq 0}\frac{1}{j!}\gamma_j(\boldsymbol{z})(s-1)^j,\]
où les $\gamma_j(\boldsymbol{z})$ sont des fonctions entières de $\boldsymbol{z}$, satisfaisant, pour tout $\frac{1}{\ell}>A>0$, $\frac{1-C_8}{C_8}\geq\varepsilon>0$,
\[\frac{1}{j!}\gamma_j(\boldsymbol{z})\ll_{A,\varepsilon}(1+\varepsilon)^j\quad \quad (|{\boldsymbol{z}}|\leq A).\]
\end{prop}
\begin{proof}
 Le résultat est immédiat comme les $\zeta_{K_i}$ n'ont pas de zéro dans ce cercle, il suffit alors d'appliquer le théorème de Cauchy (Théorème II.5.1, \cite{Tenenbaum_fr}).
\end{proof}
 
 On peut alors appliquer la méthode de Selberg-Delange et obtenir l'estimation escomptée.
 
 \begin{prop}\label{prop_resultat_Selberg_Delange}
  Par la méthode de Selberg-Delange, on a
  \[\sum\limits_{n\leq x}\alpha_{\boldsymbol{z}}(n)=x(\log x)^{\sum\limits_{i}\frac{z_i}{n_i}-1}\left(\frac{G(1,{\boldsymbol{z}})\gamma_0(\boldsymbol{z})}{\Gamma(\sum\limits_{i}\frac{z_i}{n_i})}+\bo_{C}\left(\frac{1}{\log x}\right)\right)+\bo\left(x e^{-c_6\sqrt{\log x}}\right).\]
 \end{prop}
\begin{proof}
 
Posons $A(x,\boldsymbol{z}):=\sum\limits_{n\leq x}\alpha_{\boldsymbol{z}}(n)$. Alors par la formule de Perron on a (Tenenbaum \cite{Tenenbaum_fr}, Théorème II.2.5) :
\[\int_0^x A(t,\boldsymbol{z})dt=\frac{1}{2i\pi}\int_{\kappa -i\infty}^{\kappa +i\infty}F(s,\boldsymbol{z})x^{s+1}\frac{ds}{s(s+1)},\]
avec $\kappa:= 1+\frac{1}{\log x}$.

Définissons alors le domaine $\mathcal{D}$ comme la région sans zéro commune à nos fonctions zêta de Dedekind et $\tilde{\mathcal{D}}:=\mathcal{D}\backslash [\frac{1}{2}+\varepsilon,1]$.

Ainsi, pour $s=\sigma+it\in\mathcal{D}$, $|s-1|\gg 1$, $|{\boldsymbol{z}}|\leq A$, $A<\frac{1}{\ell}$, par la proposition \ref{prop_maj_G} et le lemme \ref{lem_Wang},
\[F(s,{\boldsymbol{z}})\ll\prod\limits_{i=1}^\ell \left((1+|t|)^{\frac{n_i}{3}(1-\sigma)+\frac{1}{\log x}}\right)^{\frac{A}{n_i}}\ll(1+|t|)^{\frac{\ell A}{3}(1-\sigma)+\frac{\ell A}{\log x}}.\]
 Alors, avec $T>1$ un paramètre à définir, les contributions des demi-droites verticales $[\kappa \pm iT,\kappa\pm i\infty[$ sont :
 \begin{align*}
  \int_{\kappa +iT}^{\kappa +i\infty}F(s,{\boldsymbol{z}})x^{s+1}\frac{ds}{s(s+1)}&\ll x^2\int_{T}^{+\infty}t^{\frac{4\ell A}{3\log x}-2}dt
 \end{align*}
et donc,
\[ \int_{\kappa +iT}^{\kappa +i\infty}F(s,{\boldsymbol{z}})x^{s+1}\frac{ds}{s(s+1)}\ll x^2 T^{\frac{4\ell A}{3\log x}-1},\]
et il en va de même pour l'autre demi-droite.

Posons $C_9(T):=\frac{1}{C_6\log T+C_7}$.

Il reste à évaluer  $\int_{\kappa -iT}^{\kappa +iT}F(s,{\boldsymbol{z}})x^{s+1}\frac{ds}{s(s+1)}$, pour ce faire on va déformer le chemin d'intégration comme suit : 
\begin{itemize}
 \item [$\bullet$] $\mathcal{L}_1$ le segment $[\kappa -iT,1-C_9(T) -iT]$;
 \item [$\bullet$] $\mathcal{L}_2$ la courbe $\sigma(t)=1-\frac{1}{C_6\log t+C_7}$, pour $t\in[1-C_9(T) -iT,1-C_8]$;
 \item [$\bullet$] un contour de Hankel tronqué $\Gamma$, entourant $1$ de rayon $\frac{1}{\log x}$ et de partie rectiligne joignant $1-\frac{1}{\log x}$ à $1-C_8$;
 \item [$\bullet$] on complète le contour par symétrie par rapport à l'axe réel.  
\end{itemize}

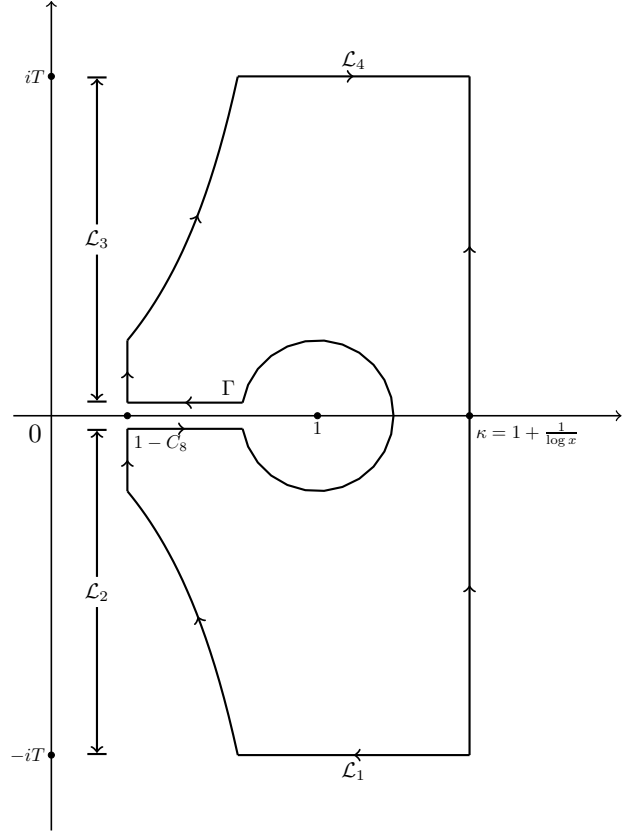
\begin{tikzpicture}
\usetikzlibrary{decorations.markings}
\def\rayonpoint{0.04}
\def\longx{8}
\def\longy{11}
\def\decalflecheshor{0.4}
\def\radun{1}
\def\angleun{170}
\def\decaly{0.5}
\def\poshorun{4}
\def\poshorleftintpath{1.5}
\def\poshorrightintpath{6}
\def\posvertopintpath{10}

\def\posvertdebutzerofree{13/2}
\def\spacel{0.2}

\draw[semithick,->] ({0},{\longy/2}) -- ({\longx},{\longy/2});
\draw[semithick,->] ({\decaly},0) -- ({\decaly},\longy);
\draw [fill] ({\poshorun},{\longy/2}) circle [radius=\rayonpoint] node[scale=0.7,anchor=north] {$1$};
\draw [fill] ({\decaly},{\posvertopintpath}) circle [radius=\rayonpoint] node[scale=0.7,anchor=east] {$iT$};
\draw [fill] ({\decaly},{\longy-\posvertopintpath}) circle [radius=\rayonpoint] node[scale=0.7,anchor=east] {$-iT$};
\draw [thick,domain=\angleun:-\angleun] plot ({\poshorun+\radun*cos(\x)}, {\longy/2+\radun*sin(\x)});
\draw ({\decaly},{\longy/2}) node[anchor=north east] {0};
\draw [fill] ({\poshorleftintpath},{\longy/2}) circle [radius=\rayonpoint];
\draw [fill] ({\poshorrightintpath},{\longy/2}) circle [radius=\rayonpoint] node[scale=0.7,anchor=north west] {$\kappa=1+\frac{1}{\log x}$};
\draw ({\poshorun+\radun*cos(\angleun)},{\longy/2+\radun*sin(\angleun)}) node[scale=0.8,anchor=south east] {$\Gamma$};
\draw ({\poshorleftintpath},{\longy/2-\radun*sin(\angleun)}) node[scale=0.7,anchor=north west] {$1-C_8$};
\draw[thick,->|] ({\poshorleftintpath-\decalflecheshor},{(\posvertopintpath+\longy/2+\radun*sin(\angleun))/2+\spacel})--({\poshorleftintpath-\decalflecheshor},\posvertopintpath);
\draw[thick,|<-] ({\poshorleftintpath-\decalflecheshor},{\longy/2+\radun*sin(\angleun)})--({\poshorleftintpath-\decalflecheshor},{(\posvertopintpath+\longy/2+\radun*sin(\angleun))/2-\spacel});
\draw ({\poshorleftintpath-\decalflecheshor},{(\posvertopintpath+\longy/2+\radun*sin(\angleun))/2}) node[scale=0.8] {$\mathcal{L}_{3}$};
\draw[thick,|<-] ({\poshorleftintpath-\decalflecheshor},{\longy/2-\radun*sin(\angleun)})--({\poshorleftintpath-\decalflecheshor},{(\longy-\posvertopintpath+\longy/2-\radun*sin(\angleun))/2+\spacel});
\draw[thick,|<-] ({\poshorleftintpath-\decalflecheshor},{\longy-\posvertopintpath})--({\poshorleftintpath-\decalflecheshor},{(\longy-\posvertopintpath+\longy/2-\radun*sin(\angleun))/2-\spacel});
\draw ({\poshorleftintpath-\decalflecheshor},{(\longy-\posvertopintpath+\longy/2-\radun*sin(\angleun))/2}) node[scale=0.8] {$\mathcal{L}_{2}$};
\draw ({(4*(1.7*(1-1/(ln(10-\posvertdebutzerofree+exp(1/(1-0.5/1.7))))))-4*0.5+\poshorleftintpath+\poshorrightintpath)/2},{\posvertopintpath}) node[scale=0.8, anchor=south] {$\mathcal{L}_{4}$};
\draw ({(4*(1.7*(1-1/(ln(10-\posvertdebutzerofree+exp(1/(1-0.5/1.7))))))-4*0.5+\poshorleftintpath+\poshorrightintpath)/2},{\longy-\posvertopintpath}) node[scale=0.8, anchor=north] {$\mathcal{L}_{1}$};

\begin{scope}[thick,decoration={
    markings,
    mark=at position 0.5 with {\arrow{>}}}
    ] 
    \draw[postaction={decorate}] ({\poshorun+\radun*cos(\angleun)},{\longy/2+\radun*sin(\angleun)})--({\poshorleftintpath},{\longy/2+\radun*sin(\angleun)});
    
    \draw[postaction={decorate}] ({\poshorleftintpath},{\longy/2+\radun*sin(\angleun)})--({\poshorleftintpath},\posvertdebutzerofree);
    
    \draw[postaction={decorate}] ({4*(1.7*(1-1/(ln(10-\posvertdebutzerofree+exp(1/(1-0.5/1.7))))))-4*0.5+\poshorleftintpath},\posvertopintpath)--({\poshorrightintpath},\posvertopintpath);
    \draw[postaction={decorate}] ({\poshorrightintpath},{\longy/2})--({\poshorrightintpath},\posvertopintpath);
    \draw[postaction={decorate}] ({\poshorrightintpath},{\longy-\posvertopintpath})--({\poshorrightintpath},{\longy/2});
    \draw[postaction={decorate}] ({\poshorrightintpath},{\longy-\posvertopintpath})--({4*(1.7*(1-1/(ln(10-\posvertdebutzerofree+exp(1/(1-0.5/1.7))))))-4*0.5+\poshorleftintpath},{\longy-\posvertopintpath});
    \draw[postaction={decorate}] ({\poshorleftintpath},{\longy-\posvertdebutzerofree})--({\poshorleftintpath},{\longy/2-\radun*sin(\angleun)});
    \draw[postaction={decorate}] ({\poshorleftintpath},{\longy/2-\radun*sin(\angleun)})--({\poshorun+\radun*cos(\angleun)},{\longy/2-\radun*sin(\angleun)});
    \draw[postaction={decorate}] [thick,domain=0.5:{1.7*(1-1/(ln(10-\posvertdebutzerofree+exp(1/(1-0.5/1.7)))))}] plot ({4*\x-4*0.5+\poshorleftintpath}, {\posvertdebutzerofree+exp(1/(1-\x/1.7))-exp(1/(1-0.5/1.7))});
    \draw[postaction={decorate}] [thick,domain={1.7*(1-1/(ln(10-\posvertdebutzerofree+exp(1/(1-0.5/1.7)))))}:0.5] plot ({4*\x-4*0.5+\poshorleftintpath}, {-(\posvertdebutzerofree+exp(1/(1-\x/1.7))-exp(1/(1-0.5/1.7)))+11});
\end{scope}

\end{tikzpicture}
 
Pour le segment $\mathcal{L}_1$ on a la majoration :
 \begin{align*}
  \int_{\kappa -iT}^{1-C_9(T) -iT}F(s,{\boldsymbol{z}})x^{s+1}\frac{ds}{s(s+1)}&\ll  x^2T^{\ell A\left(\frac{C_9(T)}{3}+\frac{1}{\log x}\right)-2},
 \end{align*}
 et il vient la même majoration pour $\mathcal{L}_4$.
 
Puis sur la courbe $\mathcal{L}_2$ :
 \begin{align*}
  \int_{1-C_9(T) -iT}^{1-C_8 }F(s,{\boldsymbol{z}})x^{s+1}\frac{ds}{s(s+1)}&\ll x^{2-C_9(T)}\left(\int_{T}^{+\infty}t^{\ell A\left(\frac{C_9(T)}{3}+\frac{1}{\log x}\right)-2}dt+\bo(1)\right)\ll  x^{2-C_9(T)},
 \end{align*}
car $A<\frac{1}{\ell}$, et il vient la même majoration pour $\mathcal{L}_3$.

Posons alors 
\[\Phi(x):=\frac{1}{2\pi i}\int_\Gamma F(s,{\boldsymbol{z}}) x^{s+1}\frac{ds}{s(s+1)}.\]
On a donc, en prenant $T=x^{c_1\sqrt{\log x}}$,
\begin{equation}\label{int_A_0_x}\int_0^x A(t,\boldsymbol{z})dt=\Phi(x)+\bo\left( x^2 e^{-c_2\sqrt{\log x}}\right).\end{equation}

Pour $s\in\mathcal{D}$, on a
\[F(s,{\boldsymbol{z}})=s G(s,{\boldsymbol{z}})Z(s,{\boldsymbol{z}})(s-1)^{-\sum\limits_i \frac{z_i}{n_i}},\]
et donc par la proposition \ref{prop_taylor_Z}, et comme $|s-1|<1$ pour $s$ dans $\Gamma$,
\begin{equation}\label{maj_F_Hankel}F(s,{\boldsymbol{z}})\ll|s-1|^{-A}, \quad (s\in\Gamma).\end{equation}

Notons, dans le domaine d'holomorphie de $G(s,{\boldsymbol{z}})$,
\[G^{(k)}(s,\boldsymbol{z}):=\frac{\partial^k}{\partial s^k}G(s,{\boldsymbol{z}}),\]
on a alors, pour $s\in\Gamma$,
\[G(s,{\boldsymbol{z}})Z(s,{\boldsymbol{z}})=\sum\limits_{k\geq 0} \mu_k(\boldsymbol{z})(s-1)^k,\]
avec $\mu_k(\boldsymbol{z}):=\sum\limits_{h+j=k}\frac{1}{h!j!}G^{(h)}(1,\boldsymbol{z})\gamma_j(\boldsymbol{z}),$ et donc
\[G(s,{\boldsymbol{z}})Z(s,{\boldsymbol{z}})= \mu_0(\boldsymbol{z})+\bo\left(|s-1|\right).\]

Alors on a, (pour le détail voir \cite{Tenenbaum_fr}, p241-242),
\[\Phi'(x)=x(\log x)^{\sum\limits_{i}\frac{z_i}{n_i}-1}\left(\frac{\mu_0(\boldsymbol{z})}{\Gamma(\sum\limits_{i}\frac{z_i}{n_i})}+\bo\left(\frac{1}{\log x}\right)\right).\]

Remarquons également que l'on a :
\[\Phi'(x)=\frac{1}{2\pi i}\int_\Gamma F(s,\boldsymbol{z}) x^{s}\frac{ds}{s},\quad \quad\quad \quad \Phi''(x)=\frac{1}{2\pi i}\int_\Gamma F(s,\boldsymbol{z}) x^{s-1}ds.\]

Il reste à montrer que $\Phi'(x)$ est une bonne approximation de $A(x,\boldsymbol{z})$. Pour ce faire on procède comme dans \cite{Tenenbaum_fr} p243. Soit $0<h<\frac{x}{2}$, on a alors, par \eqref{int_A_0_x},
\[\int_x^{x+h}A(t,\boldsymbol{z})dt=\Phi(x+h)-\Phi(x)+\bo\left( x^2 e^{-c_3\sqrt{\log x}}\right).\]
Puis, en utilisant la formule de Taylor,
\begin{align*}
 \Phi(x+h)-\Phi(x)&=h\Phi'(x+h)-h^2\int_0^1 t\Phi''(x+th)dt\\
 &=h \Phi'(x)+h^2\int_0^1 (1-t)\Phi''(x+th)dt.
\end{align*}
Puis par \eqref{maj_F_Hankel}, on a $\Phi''(x)\ll (\log x)^A$, ainsi $\Phi(x+h)-\Phi(x)=h \Phi'(x)+\bo\left(h^2(\log x)^A\right)$, et donc
\begin{align*} 
A(x,\boldsymbol{z})&=\frac{1}{h}\int_{x}^{x+h}A(t,\boldsymbol{z})dt+\frac{1}{h}\int_{x}^{x+h}(A(x,\boldsymbol{z})-A(t,\boldsymbol{z}))dt\\
&=\Phi'(x)+\bo\left(x^2 e^{-c_3\sqrt{\log x}}+h(\log x)^A+\frac{1}{h}\int_x^{x+h}(A(t,\boldsymbol{z})-A(x,\boldsymbol{z}))dt\right).
\end{align*}
Or
\begin{align*}
 \int_x^{x+h}(A(t,\boldsymbol{z})-A(x,\boldsymbol{z}))dt\leq \int_x^{x+h}A(t,\boldsymbol{z})dt-\int_{x-h}^{x}A(t,\boldsymbol{z})dt=\bo\left(hx^2 e^{-c_4\sqrt{\log x}}+h^2(\log x)^A\right).
\end{align*}
Ainsi on a finalement, en prenant $h=xe^{-c_5\sqrt{\log x}}$

\begin{equation}\label{A_fin_Selberg_Delange}
 A(x,\boldsymbol{z})=x(\log x)^{\sum\limits_{i}\frac{z_i}{n_i}-1}\left(\frac{\mu_0(\boldsymbol{z})}{\Gamma(\sum\limits_{i}\frac{z_i}{n_i})}+\bo\left(\frac{1}{\log x}\right)\right)+\bo\left(x e^{-c_6\sqrt{\log x}}\right).
\end{equation}
\end{proof}

Il ne reste plus qu'à appliquer le théorème de Cauchy à ce résultat pour obtenir l'estimation attendue, similairement au Théorème II.6.3 de \cite{Tenenbaum_fr}.

\begin{proof}[Démonstration du Théorème \ref{thm_finale_Selberg_Delange}]

On a

\begin{align*}
 G(1,{\boldsymbol{z}})\gamma_0(\boldsymbol{z})=\prod\limits_i \left(G_i(1,z_i)Z_i(1,z_i)\right)\\
\end{align*}
Notons alors $\Lambda_i(z_i):=G_i(1,z_i)Z_i(1,z_i)$, i.e,
\begin{equation}\label{def_Lambda_i}
 \Lambda_i(z_i)=\prod\limits_{p\in P_i}\left(\left(1+\frac{z_i}{p}\right)\left(1-\frac{1}{p}\right)^{z_i}\right)\prod\limits_{\substack{\mathfrak{p}\in O_{K_i}\\ N(\mathfrak{p})|a\upsilon_i\\ N(\mathfrak{p})\ \text{premier}}}\left(1-\frac{1}{N(\mathfrak{p})}\right)^{\frac{z_i}{n_i}}\beta_{K_i}^{\frac{z_i}{n_i}},
\end{equation}
où $\beta_{K_i}:=\lim\limits_{s\rightarrow 1}(s-1)\prod\limits_{\substack{\mathfrak{p}\in O_{K_i}\\ N(\mathfrak{p})\ \text{premier}}}\left(1-\frac{1}{N(\mathfrak{p})}\right)$, qui est bien défini et non nul car $\zeta_{K_i}$ a un pôle simple en $1$.

Soit $\Omega$ le produit de $\ell$ cercles autour de l'origine de rayon $r$, $\overline{1}:=(1,\ldots,1)$. On a donc par la proposition \ref{prop_resultat_Selberg_Delange},
\begin{align*}
\sum\limits_{n\leq x}c_{\overline{1}}(n)&=\sum\limits_{n\leq x}\frac{1}{(2i\pi)^\ell}\oiint_\Omega \frac{\alpha_{\boldsymbol{z}}(n)}{\prod\limits_i z_i^2}\prod\limits_i dz_i\\
&=\frac{1}{(2i\pi)^\ell}\oiint_\Omega A(x,\boldsymbol{z})\prod\limits_i \frac{dz_i}{z_i^2}\\
&=S+\bo(E).
\end{align*}
où $S:=\frac{x}{(2i\pi)^\ell\log x}\oiint_\Omega \frac{(\log x)^{\sum\limits_i\frac{z_i}{n_i}}}{\Gamma\left(\sum\limits_i\frac{z_i}{n_i}\right)}\prod\limits_i \Lambda_i(z_i) \frac{dz_i}{z_i^2}$,\\ et $E:=\frac{1}{(2\pi)^\ell}\oiint_\Omega \left(x(\log x)^{\ell r-2}+x e^{-c_6\sqrt{\log x}}\right)\prod\limits_i \frac{dz_i}{|z_i|^2}$.

On a alors, en prenant $r=\frac{1}{\log\log x}$,
\begin{align*}
 E\ll \frac{x (\log\log x)^\ell}{ \log^2 x}.
\end{align*}

Ensuite, on utilise la formule $\Gamma(u+1)=u\Gamma(u)$ avec $u=\sum\limits_j\frac{z_j}{n_j}$ afin de s'écarter de la singularité de $\Gamma$ en $0$. De plus $\Lambda_i(0)=1$, et ainsi en appliquant la formule de Cauchy :

\begin{align*}
 S&=\frac{x}{(2i\pi)^\ell \log x}\sum\limits_i \frac{1}{n_i}\oiint_\Omega \frac{(\log x)^{\sum\limits_j\frac{z_j}{n_j}}}{\Gamma\left(\sum\limits_j\frac{z_j}{n_j}+1\right)} \Lambda_i(z_i)\frac{dz_i}{z_i}\prod\limits_{j\neq i} \Lambda_j(z_j)\frac{dz_j}{z_j^2}\\
 &=\frac{x}{(2i\pi)^{\ell-1}\log x}\sum\limits_i \frac{1}{n_i}\oiint_{\Omega'} \frac{(\log x)^{\sum\limits_{j\neq i}\frac{z_j}{n_j}}}{\Gamma\left(\sum\limits_{j\neq i}\frac{z_j}{n_j}+1\right)} \prod\limits_{j\neq i} \Lambda_j(z_j)\frac{dz_j}{z_j^2},\\
\end{align*}

où $\Omega'$ est le produit de $\ell-1$ cercles autour de l'origine de rayon $r$.

On utilise la représentation de $\frac{1}{\Gamma}$ avec le contour de Hankel $\mathscr{H}$ (\cite{Tenenbaum_fr}, II.0.17), 
\[\frac{1}{\Gamma(z)}=\frac{1}{2i\pi}\int_\mathscr{H}s^{-z}e^sds,\]
où $\mathscr{H}$ est composé d'un cercle centré en $0$ de rayon $r$ privé du point $-r$ et de la demi-droite $]-\infty,-r]$ parcourue dans les deux sens.

On inverse les intégrales sur les $z_j$ et sur $\xi$ en notant que $\Lambda_j(z_j)$ est borné pour $z_j$ borné, puis en appliquant la formule de Cauchy on a :

\begin{align*}
 S&=\frac{x}{2i\pi\log x}\sum\limits_i \frac{1}{n_i}\int_\mathscr{H}\prod\limits_{j\neq i}\left(\frac{1}{2i\pi}\int_{|z_j|=r} \left(\frac{\log x}{\xi}\right)^{\frac{z_j}{n_j}} \Lambda_j(z_j)\frac{dz_j}{z_j^2}\right)\frac{e^\xi}{\xi}d\xi\\
 &=\frac{x}{2i\pi\log x}\sum\limits_i \frac{1}{n_i}\int_\mathscr{H}\prod\limits_{j\neq i}\frac{\partial}{\partial z_j}\left( \left(\frac{\log x}{\xi}\right)^{\frac{z_j}{n_j}} \Lambda_j(z_j)\right)_{|z_j=0}\frac{e^\xi}{\xi}d\xi.
\end{align*}

Puis, comme $\Lambda_i(0)=1$, 
\[\frac{\partial}{\partial z_j}\left( \left(\frac{\log x}{\xi}\right)^{\frac{z_j}{n_j}} \Lambda_j(z_j)\right)_{|z_j=0}=\frac{\log\log x -\log\xi}{n_j}+\frac{\partial}{\partial z_j}\Lambda_j(0).\]

Pour la dérivée en zéro de $\Lambda_j(z_j)$ on calcule sa dérivée logarithmique :
\begin{align*}
 \frac{\partial}{\partial z_j}\Lambda_j(0)&=\frac{\Lambda'_j(0)}{\Lambda_j(0)}=\sum\limits_{p\in P_j}\left(\frac{1}{p}+\log\left(1-\frac{1}{p}\right)\right)+\sum\limits_{\substack{\mathfrak{p}\in O_{K_i}\\ N(\mathfrak{p})|a\upsilon_i\\ N(\mathfrak{p})\ \text{premier}}}\frac{1}{n_j}\log\left(1-\frac{1}{N(\mathfrak{p})}\right)+\frac{1}{n_j}\log(\beta_{K_j})\\
 &=\bo(\upsilon_i).
\end{align*}

Ainsi

\begin{align*}
S&=\frac{x}{2i\pi\log x}\sum\limits_i \frac{1}{n_i}\int_\mathscr{H}\prod\limits_{j\neq i}\left(\frac{\log\log x-\log\xi}{n_j}+\bo(1)\right)\frac{e^\xi}{\xi}d\xi\\
&=\frac{\ell x}{\log x \prod\limits_i n_i}(\log\log x)^{\ell-1}+\bo\left(\frac{x}{\log x}(\log\log x)^{\ell-2}\int_\mathscr{H}\frac{e^{\Re\xi}}{|\xi|^{1-\varepsilon}}d\xi\right).
\end{align*}
 
 On prend $1$ pour le rayon du cercle autour de l'origine de $\mathscr{H}$, on obtient $\int_\mathscr{H}\frac{e^{\Re\xi}}{|\xi|^{1-\varepsilon}}d\xi\ll 1$, et donc
 
 \begin{equation}\label{resultat_final_Selberg_Delange}
  \sum\limits_{\substack{p_1\cdots p_\ell\leq x\\ (a,p_1\cdots p_\ell)=1\\ \forall i,\ p_i\equiv 1(\upsilon_i)\\a\in \qroot\left(\kappa_i, p_i\right)}}1=\sum\limits_{n\leq x}c_1(n)=\frac{\ell x}{\log x \prod\limits_i n_i}(\log\log x)^{\ell-1}+\bo\left(\frac{x}{\log x}(\log\log x)^{\ell-2}\right).
 \end{equation}

\end{proof}

En appliquant le Théorème \ref{thm_finale_Selberg_Delange} à la dernière somme de l'expression de $\Pa'(x,k)$ de la Proposition \ref{prop_controle_gL} on obtient le résultat suivant (en prenant $\upsilon_i=u_i d_i$ et $\kappa_i=k'_i$):
\begin{equation}\label{equation_finale_Selberg_Delange}
 \sum\limits_{\substack{p_1\cdots p_\ell\leq x\\ (a,p_1\cdots p_\ell)=1\\ \forall i,\ p_i\equiv 1(u_i d_i)\\a\in \qroot\left(k'_i, p_i\right)}}1=\frac{\ell x}{\log x \prod\limits_i n_i}(\log\log x)^{\ell-1}+\bo_{C_1,C_5}\left(\frac{x}{\log x}(\log\log x)^{\ell-2}\right),
\end{equation}
où $n_i:=[\Q(\sqrt[k'_i]{a};\xi_{u_i d_i}):\Q]$.

\section{Contrôle des termes d'erreurs}

Dans cette section nous traitons tous les termes d'erreurs accumulés jusqu'ici, ce que nous regroupons dans la proposition suivante.

\begin{prop}\label{prop_majoration_termes_derreurs}
 On a avec $C_1$ et $C_5$ des constantes arbitrairement grandes, $\varepsilon >0$,
 \begin{align*}
 \Na(x,C_1)=&\frac{x(\log\log x)^{\ell -1}}{(\ell -1)!\log x}\sum\limits_{k}\mu(k)\mathcal{P}'_{a,0}(k)+\bo\left(\frac{e^{(1+\varepsilon)C_1}x}{C_5^{1-\varepsilon}\log x}(\log\log x)^{\ell -1}\right)\\
 &+\bo_{C_1,C_5}\left(\frac{x(\log\log x)^{\ell -2}}{\log x}\left(C_5\log C_5\right)^{2^\ell }\right)+\bo\left(\frac{x(\log\log x)^{\ell -1}}{C_5(\ell -1)!\log x}e^{\frac{C_1}{\log C_1}}\left(\log\log C_1\right)^\ell \right)\\
 &+\bo_\ell\left(\frac{x(\log\log x)^{\ell -1}}{C_1^{1-\varepsilon}\log x}\right),\end{align*}
 où
 \[\mathcal{P}'_{a,0}(k)=\sum\limits_{\substack{\left\{k_L,\ L\in\mathcal{P}^*(E)\right\}\\ \prod\limits_{ L\in\mathcal{P}^*(E)}k_L=k}}\sum\limits_{\substack{\{g_L,\ L\in\mathcal{P}^*(E)\}\\ g_L|k_L^\infty\\ k_L|g_L}}\sum\limits_{\substack{\{c_{iL},\ L\in\mathcal{P}^*(E),\ i\notin L\}\\ c_{iL}|\frac{g_L}{k_L}}}\sum\limits_{\substack{\left\{d_i,\ i\in \{1,\ldots ,\ell \}\right\}\\ d_i|k}}\prod\limits_{i\in \{1,\ldots,\ell \}}\frac{\mu(d_i)}{n_i},\]
 et $n_i:=[\Q(\sqrt[k_i']{a},\xi_{u_i d_i}):\Q]$.
\end{prop}
\begin{proof}
Rappelons que $\Na(x,C_1)=\sum\limits_{P^+(k)\leq C_1}\mu(k)\Pa(x,k)$.

Commençons par majorer le terme d'erreur apporté par la proposition \ref{prop_controle_gL}, il faut alors sommer sur les $k$ sans facteur carré $C_1$-friables. Ainsi,
\[\sum\limits_{P^+(k)\leq C_1}\mu^2(k)\frac{x}{C_5^{1-\varepsilon}\log x}(\log\log x)^{\ell-1}\leq \frac{2^{\pi(C_1)}x}{C_5^{1-\varepsilon}\log x}(\log\log x)^{\ell-1}.\]

$C_1$ et $C_5$ étant des constantes, on a que la contribution des termes d'erreurs issus de l'équation \eqref{equation_finale_Selberg_Delange}, sont $\bo_{C_1,C_5}\left(\frac{x(\log\log x)^{\ell -2}}{\log x}\left(C_5\log C_5\right)^{2^\ell }\right)$. 

Il reste à estimer la somme associée aux termes principaux :
\[T:=\sum\limits_{P^+(k)\leq C_1}\mu(k)\sum\limits_{\substack{\left\{k_L,\ L\in\mathcal{P}^*(E)\right\}\\ \prod\limits_{ L\in\mathcal{P}^*(E)}k_L=k}}\sum\limits_{\substack{\{g_L,\ L\in\mathcal{P}^*(E)\}\\ g_L\leq C_5\\ g_L|k_L^\infty\\ k_L|g_L}}\sum\limits_{\substack{\{c_{iL},\ L\in\mathcal{P}^*(E),\ i\notin L\}\\ c_{iL}|\frac{g_L}{k_L}}}\sum\limits_{\substack{\left\{d_i,\ i\in \{1,\ldots ,\ell \}\right\}\\ d_i|k}}\prod\limits_{i\in \{1,\ldots,\ell \}}\frac{\mu(d_i)}{n_i}.\]
On souhaite enlever la condition $g\leq C_5$. Commençons par majorer 
 \[
  E_{L_0}:=\frac{x(\log\log x)^{\ell-1}}{(\ell-1)!\log x}\sum\limits_{\substack{\left\{k_L,\ L\in\mathcal{P}^*(E)\right\}\\ \prod\limits_{ L\in\mathcal{P}^*(E)}k_L=k}}\sum\limits_{\substack{\{g_L,\ L\in\mathcal{P}^*(E)\backslash L_0\}\\ g_L|k_L^\infty\\ k_L|g_L}}\sum\limits_{\substack{g_{L_0}> C_5\\ g_{L_0}|k_{L_0}^\infty\\ k_{L_0}|g_{L_0}}}\sum\limits_{\substack{\{c_{iL},\ L\in\mathcal{P}^*(E),\ i\notin L\}\\ c_{iL}|\frac{g_L}{k_L}}}\sum\limits_{\substack{\left\{d_i,\ i\in \{1,\ldots ,\ell \}\right\}\\ d_i|k}}\prod\limits_{i\in \{1,\ldots,\ell \}}\frac{1}{n_i}.\]
Remarquons que pour $\varepsilon_1>0$, $\frac{1}{n_i}\ll \frac{1}{k_i \varphi(u_i d_i)}\ll\prod\limits_{\substack{L\in\mathcal{P}^*(E)\\ i\in L}}\frac{1}{{g_L}^{1-\varepsilon_1}}$. En reportant dans $E_{L_0}$ on obtient :
 \[
  E_{L_0}=\frac{x(\log\log x)^{\ell-1}}{(\ell-1)!\log x}\sum\limits_{\substack{\left\{k_L,\ L\in\mathcal{P}^*(E)\right\}\\ \prod\limits_{ L\in\mathcal{P}^*(E)}k_L=k}}\sum\limits_{\substack{\{g_L,\ L\in\mathcal{P}^*(E)\backslash L_0\}\\ g_L|k_L^\infty\\ k_L|g_L}}\sum\limits_{\substack{g_{L_0}> C_5\\ g_{L_0}|k_{L_0}^\infty\\ k_{L_0}|g_{L_0}}}\tau(k)^\ell\prod\limits_{\substack{L\in\mathcal{P}^*(E)}}\Bigg(\tau(g_L)^{|E\backslash L|}\bigg(\frac{1}{ {g_L}^{1-\varepsilon_1}}\bigg)^{|L|}\Bigg),\]
où $\tau$ est la fonction nombre de diviseurs. Soit $\varepsilon_2>0$, on a le résultat classique pour tout entier $\mathbf{n}$, $\tau(\mathbf{n})\ll\mathbf{n}^{\varepsilon_2}$. De plus comme $k$ est sans facteur carré et $C_1$-friable on a la majoration suivante : $k\ll e^{C_1}$. En reportant cela dans l'expression de $E_{L_0}$ on obtient :
\[
  E_{L_0}\ll\frac{x(\log\log x)^{\ell-1}}{(\ell-1)!\log x}e^{\varepsilon_2 \ell C_1}\sum\limits_{\substack{\left\{k_L,\ L\in\mathcal{P}^*(E)\right\}\\ \prod\limits_{ L\in\mathcal{P}^*(E)}k_L=k}}\sum\limits_{\substack{\{g_L,\ L\in\mathcal{P}^*(E)\backslash L_0\}\\ g_L|k_L^\infty\\ k_L|g_L}}\sum\limits_{\substack{g_{L_0}> C_5\\ g_{L_0}|k_{L_0}^\infty\\ k_{L_0}|g_{L_0}}}\prod\limits_{\substack{L\in\mathcal{P}^*(E)}}\Bigg(\frac{1}{ {g_L}^{1-\varepsilon_1-\varepsilon_2}}\Bigg).\]

  Soient $L\neq L_0$, $\varepsilon_3>0$,
\[\sum\limits_{\substack{g_L\\ g_L|k_L^\infty\\ k_L|g_L}}\frac{1}{ {g_L}^{1-\varepsilon_1-\varepsilon_2}}\ll\prod\limits_{p|k_L}\Bigg(1+\frac{1}{p^{1-\varepsilon_1-\varepsilon_2}-1}\Bigg)\ll(1+\varepsilon_3)^{\omega(k_L)}\ll e^{\varepsilon_3 C_1}.\]
  
  Occupons nous de la somme sur $g_{L_0}$, soit $\varepsilon_4>0$ tel que $ \varepsilon_1+\varepsilon_2+\varepsilon_4<1$, on a
 \[\sum\limits_{\substack{g_{L_0}> C_5\\ g_{L_0}|k_{L_0}^\infty\\ k_{L_0}|g_{L_0}}}\frac{1}{ {g_{L_0}}^{1-\varepsilon_1-\varepsilon_2}}\ll\frac{1}{C_5^{\varepsilon_4}}\sum\limits_{g_{L_0}|k_{L_0}^\infty\\ k_{L_0}|g_{L_0}}\frac{1}{ {g_{L_0}}^{1-\varepsilon_1-\varepsilon_2-\varepsilon_4}}\ll\frac{ e^{\varepsilon_3 C_1}}{C_5^{\varepsilon_4}}.\]
 
Soit $\varepsilon_5>\varepsilon_2+\varepsilon_3$, en reportant dans $E_{L_0}$,
\[
  E_{L_0}\ll\frac{x(\log\log x)^{\ell-1}}{(\ell-1)!\log x}\frac{ 2^{C_1}e^{(\varepsilon_2+\varepsilon_3) \ell C_1}}{C_5^{\varepsilon_4}}\ll\frac{x(\log\log x)^{\ell-1}}{(\ell-1)!\log x}\frac{ e^{\varepsilon_5 \ell C_1}}{C_5^{\varepsilon_4}}.\]

  Puis comme le nombre de $C_1$-friable sans facteur carré est majoré par $e^{C_1}$ on a la majoration souhaitée.

Il ne reste qu'à compléter la somme sur les $k$.

La somme que l'on souhaite majorer est :
\[\hspace*{-2.5cm} E_{C_1}:=\frac{x(\log\log x)^{\ell-1}}{(\ell-1)!\log x}\sum\limits_{\substack{P^+(k)>C_1\\ \mu^2(k)=1}}\sum\limits_{\substack{\left\{k_L,\ L\in\mathcal{P}^*(E)\right\}\\ \prod\limits_{ L\in\mathcal{P}^*(E)}k_L=k}}\sum\limits_{\substack{\{g_L,\ L\in\mathcal{P}^*(E)\}\\ g_L|k_L^\infty\\ k_L|g_L}}\sum\limits_{\substack{\{c_{iL},\ L\in\mathcal{P}^*(E),\ i\notin L\}\\ c_{iL}|\frac{g_L}{k_L}}}\sum\limits_{\substack{\left\{d_i,\ i\in \{1,\ldots ,\ell \}\right\}\\ d_i|k}}\prod\limits_{i\in \{1,\ldots,\ell \}}\frac{1}{n_i}.\]

Soit $\varepsilon_1 >0$, on a la majoration suivante pour tout $i$ : $\frac{1}{n_i}\ll\frac{1}{d_i^{1-\varepsilon_1}\prod\limits_{L,\ i\in L}k_L g_L^{1-\varepsilon_1}}$.

On majore simplement les sommes sur le $d_i$ :

\[\sum\limits_{d_i|k}\frac{1}{d_i^{1-\varepsilon_1}}=\prod\limits_{p|k}\left(1+\frac{1}{p^{1-\varepsilon_1}}\right)\ll\left(\frac{k}{\varphi(k)}\right)^{1-\varepsilon_1}\ll (\log\log k)^{1-\varepsilon_1}.\]

Soit $\varepsilon_2>0$, pour tous $i$ et $L$ la somme sur $c_{i,L}$ est petite devant $\left(\frac{g_L}{k_L}\right)^{\varepsilon_2}$.

Notons $\varepsilon_3=\ell(\varepsilon_1+\varepsilon_2)$, et choisissons $\varepsilon_1$ et $\varepsilon_2$ de sorte que $0<\varepsilon_3<1$. Les sommes sur les $g_L$ sont alors :
\[
 \sum\limits_{\substack{g_L|k_L^\infty\\ k_L|g_L}}\frac{1}{g_L^{|L|(1-\varepsilon_1)-|E\backslash L|\varepsilon_2}}\ll\frac{1}{k_L^{|L|-\varepsilon_3}} \sum\limits_{\substack{g_L|k_L^\infty}}\frac{1}{g_L^{|L|-\varepsilon_3}}\ll\prod\limits_{p|k_L}\frac{1}{p^{|L|-\varepsilon_3}-1}.
\]
Soit $\varepsilon_4>0$ tel quel $\varepsilon_4>\varepsilon_3$, en reportant dans $E_{C_1}$ on obtient :
\begin{align*} 
E_{C_1}&\ll\frac{x(\log\log x)^{\ell-1}}{(\ell-1)!\log x}\sum\limits_{\substack{P^+(k)>C_1\\ \mu^2(k)=1}}(\log\log k)^{\ell(1-\varepsilon_1)}\sum\limits_{\substack{\left\{k_L,\ L\in\mathcal{P}^*(E)\right\}\\ \prod\limits_{ L\in\mathcal{P}^*(E)}k_L=k}}\prod\limits_{p|k_L}\frac{1}{p^{|L|}(p^{|L|-\varepsilon_3}-1)}\\
&\ll\frac{x(\log\log x)^{\ell-1}}{(\ell-1)!\log x}\sum\limits_{\substack{P^+(k)>C_1\\ \mu^2(k)=1}}(\log\log k)^{\ell(1-\varepsilon_1)}\sum\limits_{\substack{\left\{k_L,\ L\in\mathcal{P}^*(E)\right\}\\ \prod\limits_{ L\in\mathcal{P}^*(E)}k_L=k}}\prod\limits_{L\in\mathscr{P}^*(E)}\left(\frac{1}{k_L^{2|L|-\varepsilon_4}}\right).
\end{align*}

Posons $f(k):=\mu^2(k)\sum\limits_{\substack{\left\{k_L,\ L\in\mathcal{P}^*(E)\right\}\\ \prod\limits_{ L\in\mathcal{P}^*(E)}k_L=k}}\prod\limits_{L\in\mathscr{P}^*(E)}\left(\frac{1}{k_L^{2|L|-\varepsilon_4}}\right)$, alors $f$ est multiplicative et pour $p$ premier,
\[f(p)=\sum\limits_{m=1}^\ell \binom{\ell}{m} \frac{1}{p^{2m-\varepsilon_4}}=p^{\varepsilon_4}\sum\limits_{m=0}^{\ell-1} \binom{\ell}{m}p^{2(m-\ell)}\ll_\ell p^{-2+\varepsilon_4}.\]
Ainsi $f(k)\ll \frac{1}{k^{2-\varepsilon_4}}$.

Soit $\varepsilon_5>\varepsilon_4$, en reportant le résultat précédent dans $E_{C_1}$ on a : 
\[E_{C_1}\ll_\ell \frac{x(\log\log x)^{\ell-1}}{C_1^{1-\varepsilon_5}\log x}.\]
Or on peut prendre $\varepsilon_1$ et $\varepsilon_2$ arbitrairement petits et donc par extension $\varepsilon_5$ aussi, ce qui permet de conclure.

\end{proof}

Il ne nous reste alors plus qu'à évaluer le terme principal.

\section{Calcul du terme principal}

Dans cette section on notera pour tout nombre premier $p\geq 3$ et tout entier $i\geq1$, $R_p(i):=\sum\limits_{j=0}^{\ell -i}\binom{\ell -i}{j}\frac{(-1)^j(p-1)^{\ell -i-j}}{(p^{i+j}-1)(p-2)^{\ell -i}}$ et pour tout entier $i\geq1$, $R_2(i):=\sum\limits_{j=0}^{\ell -i}\binom{\ell -i}{j}\frac{(-1)^j2^{-j}}{2^{i+j}-1}$.

On veut calculer : $\sum\limits_{\substack{k}}\mu(k)\mathcal{P}_{a,0}'(k)$, avec
\[\mathcal{P}_{a,0}'(k)=\sum\limits_{\substack{\prod\limits_{\substack{L\in \mathscr{P}^*(E)}}k_L=k}}\sum\limits_{\substack{\{g_L\}_{L\in \mathscr{P}^*(E)}\\ g_L|k_L^\infty\\ k_L|g_L}}\sum\limits_{\substack{\{c_{bL}\}_{\substack{L\in \mathscr{P}^*(E)\\ b\notin L}}\\c_{bL}|\frac{g_L}{k_L}}}\sum\limits_{\substack{\{d_i\}_{i\in E}\\ d_i|k}}\prod\limits_{i\in E}\frac{\mu(d_i)}{n_i}\]
où $\mathscr{P}^*(E):=\{L\subset E,\ L\neq \emptyset\}$ et $n_i:=n(\prod\limits_{\substack{L\in \mathscr{P}^*(E)\\i\in L}}k_L',u_i d_i)$ le degré de l'extension $K_i$, où $k_L'=\frac{k_L}{(h,k_L)}$, avec $h$ impair, dépendant de $a$.\\
D'après la proposition \ref{propdegre}, $n(\ell,m):=[\Q(\sqrt[\ell]{a},\xi_m):\Q]=\frac{\ell\varphi(m)}{\varepsilon(\ell,m)}$ où $\ell$ est sans facteur carré, $\ell|m$ et
\begin{equation*}
\varepsilon(\ell,m)=\left\{\begin{array}{l}
2\ \text{si}\ 2|\ell ,\ 2|m,\ a_1|m\ \text{et}\ a_1\equiv 1\mod 4\\
2\ \text{si}\ 2|\ell ,\ 4|m,\ a_1|m\ \text{et}\ a_1\equiv 3\mod 4\\
2\ \text{si}\ 2|\ell ,\ 8|m,\ a_1|m\ \text{et}\ a_1\equiv 0\mod 2\\
1\ \text{sinon}
\end{array}\right..
\end{equation*}

Notons alors $u_i:=\prod\limits_{\substack{L\in \mathscr{P}^*(E)\\ i\in L}}g_L \prod\limits_{\substack{L\in \mathscr{P}^*(E)\\ i\notin L}}c_{iL}$.\\

Le résultat principal sera la proposition suivante, en sommant les $\mathcal{P}_{a,0}'(k)$ sur les $k$. 

\begin{prop}\label{prop_terme_principal}
On obtient :
 \[\sum\limits_{k}\mu(k)\mathcal{P}_{a,0}'(k)=M(E)\left(1+V_\ell(a_1)\right)
\]
où
\begin{itemize}
 \item [$\bullet$]$M(E):=\left(1-H_1(E)\right)\prod\limits_{\substack{p\\ p\geq 3}}\left(1-\left(\frac{p-2}{p-1}\right)^{\ell }\sum\limits_{i=1}^{\ell }\binom{\ell }{i}F_i(p)\right)$, est la contribution ne dépendant que de $h$,
 \item [$\bullet$]$V_\ell(a_1):=\mu(2\tilde{a_1})\frac{H_2(\ell,a_1)}{1-H_1(E)}\prod\limits_{\substack{p\\ p|a_1\\ p\geq 3}}\left(1-\left(\frac{p-2}{p-1}\right)^{\ell }\sum\limits_{i=1}^{\ell }\binom{\ell }{i}F_i(p)\right)^{-1}$, est la contribution spécifique dépendant de $a$,
 \item [$\bullet$]$\tilde{a_1}:=\frac{a_1}{(2,a_1)}$,
 \item [$\bullet$]Les $F_i$ sont des fonctions multiplicatives définies pour les nombres premiers impairs par :
 \[F_i(p):=\left(\frac{(h,p)(p-1)}{p^2(p-2)}\right)^i\left(1+R_p(i)\right)\]
 \item [$\bullet$]$H_1(E):=2^{-\ell }\left(\sum\limits_{i=1}^{\ell }\binom{\ell }{i}2^{-2i}\left(1+2^{\ell -i}R_2(i)\right)+2^{-\ell }\right)$, la contribution à $M(E)$ de $p=2$,
 \item [$\bullet$]
 \begin{equation*}\hspace*{-2cm}
 H_2(\ell,a_1):=\sum\limits_{L_0\in\mathscr{P}^*(E)}2^{-\ell -|L_0|}\sum\limits_{L_1\in \mathscr{P}^*(L_0)}\delta_5(L_0,L_1)\mu(\tilde{a_1})^{|L_1|}\prod\limits_{\substack{p|\tilde{a_1}\\ p\geq 3}}\left(\frac{(p-2)^{\ell -|L_1|}}{(p-1)^{\ell }}\right)\underset{L\in \mathscr{P}^*(E)}{\mathlarger{\mathlarger{\mathlarger{\mathlarger{\ast}}}}}\left(G_L^{L_1}\right)(\tilde{a_1})\end{equation*} la contribution de $a$ à $V_\ell(a_1)$, avec $\delta_5(L_0,L_1)=\widehat{\delta}(L_0,L_1)+\tilde{\delta}(L_0,L_1)$,
 
\begin{equation*}
 \widehat{\delta}(L_0,L_1)=\left\{\begin{array}{l}
                        2^{-|L_0|}\left((-1)^{|L_1|}+2^{\ell -|L_0|}R_2(|L_0|)\right) \text{ si }a_1\equiv 0\mod 2,\\
                        2^{-|L_0|}\left(1+2^{\ell -|L_0|}R_2(|L_0|)\right)  \text{ sinon, }
                      \end{array}
\right.
\end{equation*}

\begin{equation*}\hspace*{-3cm}
 \tilde{\delta}(L_0,L_1)=\left\{\begin{array}{l}
                        1\text{ si } L_0=E \text{ et } a_1\equiv 1\mod 4,\\
                        (-1)^{|L_1|}\text{ si } L_0=E \text{ et }a_1\equiv 3\mod 4,\\
                        0\text{ sinon, }\\
                      \end{array}
\right.
\end{equation*}
 
et les $G_L^{L_1}$ sont des fonctions multiplicatives définies pour les nombres premiers impairs par :
 \[\hspace{-2cm}G_L^{L_1}(p):=\left(\frac{(p-1)(h,p)}{p^2(p-2)}\right)^{|L|}(2-p)^{|L_1\cap L|}\left(1+R_p(|L \cup L_1|)\right).\]
\end{itemize}

\end{prop}

Afin de démontrer cette proposition nous allons expliciter $\mathcal{P}_{a,0}(k)$ en fonction de la parité de $k$.

\subsection{$\mathcal{P}_{a,0}(k)$ pour $k$ impair}
Nous commençons par démontrer un lemme sur la somme sur les diviseurs d'un entier sans facteur carré d'une certaine fonction.
\begin{lem}\label{lemme_somme_multiplicative}
 Soit $\alpha$ un entier, $\beta$ un entier sans facteur carré, alors
 \[\sum\limits_{d|\beta}\frac{\mu(d)}{\varphi(\alpha d)}=\frac{1}{\varphi(\alpha)}\prod\limits_{\substack{p|\beta\\ p|\alpha}}\left(1-\frac{1}{p}\right)\prod\limits_{\substack{p|\beta\\ p\nmid \alpha}}\left(1-\frac{1}{p-1}\right).\]
\end{lem}
\begin{proof}
 On a trivialement,
 \[\sum\limits_{d|\beta}\frac{\mu(d)}{\varphi(\alpha d)}=\frac{1}{\varphi(\alpha)}\sum\limits_{d|\beta}\frac{\mu(d)\varphi(\alpha)}{\varphi(\alpha d)}.\]
 La fonction $d\rightarrow \frac{\mu(d)\varphi(\alpha)}{\varphi(\alpha d)}$ est multiplicative et on a pour $p|\alpha$, $\frac{\mu(p)\varphi(\alpha)}{\varphi(\alpha p)}=-\frac{1}{p}$ et pour $p\nmid\alpha$, $\frac{\mu(p)\varphi(\alpha)}{\varphi(\alpha p)}=-\frac{1}{p-1}$. On obtient alors la formule en exprimant la somme sur les diviseurs de $\beta$ sous forme de produit eulérien.
\end{proof}

\begin{prop}\label{prop_k_impair}
 Pour $k$ impair, sans facteur carré, on a :
 \begin{align*}
\mathcal{P}_{a,0}'(k)&=\prod\limits_{\substack{p|k\\p\geq 3}}\left(1-\frac{1}{p-1}\right)^{\ell } \overset{\ell }{\underset{i=1}{\mathlarger{\mathlarger{\mathlarger{\mathlarger{\ast}}}}}}\left(F_i^{\ast \binom{\ell }{i}}\right)(k),
\end{align*}

où $F^{*n}:=\underbrace{F\ast\cdots\ast F}_{n\ fois}$.
\end{prop}
\begin{proof}
 Comme $k$ est impair, les $k_L$ le sont aussi et donc $\forall i$, le terme correctif de $n_i$ vérifie : \[\varepsilon(\prod\limits_{\substack{L\in \mathscr{P}^*(E)\\i\in L}}k_L',u_i d_i)=1.\] 
 
 Nous allons calculer successivement les sommes sur les $d_i$, $c_{i,L}$ et $g_L$. En utilisant le lemme \ref{lemme_somme_multiplicative}, celles sur les $d_i$ sont de la forme :
\[\sum\limits_{d|k}\frac{\mu(d)}{\varphi(u_i d)}=\frac{1}{\varphi(u_i)}\prod\limits_{\substack{p|k\\p|u_i\\p\geq 3}}\left(1-\frac{1}{p}\right)\prod\limits_{\substack{p|k\\p\nmid u_i\\ p\geq 3}}\left(1-\frac{1}{p-1}\right)=\frac{1}{\varphi(u_i)}\prod\limits_{\substack{p|k\\p|u_i\\ p\geq 3}}\left(\frac{(p-1)^2}{p(p-2)}\right)\prod\limits_{\substack{p|k\\p\geq 3}}\left(1-\frac{1}{p-1}\right).\]
On décompose $u_i$ en $\prod\limits_{\substack{L\in \mathscr{P}^*(E)\\ i\in L}}g_L \prod\limits_{\substack{L\in \mathscr{P}^*(E)\\ i\notin L}}c_{iL}$.

Rappelons que $g_L$ et $k_L$ ont les mêmes facteurs premiers et que ceux de $c_{i,L}$ divisent $k_L$.

La somme sur les $d_i$ devient pour $1\leq i \leq \ell$ :

\[\sum\limits_{d_i|k}\frac{\mu(d)}{\varphi(u_i d_i)}=\prod\limits_{\substack{p|k\\p\geq 3}}\left(1-\frac{1}{p-1}\right)\prod\limits_{\substack{L\in\mathcal{P}^*(E)\\ i\in L}}\left(\frac{1}{\varphi(g_L)}\prod\limits_{\substack{p|k_L\\ p\geq 3}}\left(\frac{(p-1)^2}{p(p-2)}\right)\right)\prod\limits_{\substack{L\in\mathcal{P}^*(E)\\ i\notin L}}\left(\frac{1}{\varphi(c_{i,L})}\prod\limits_{\substack{p|c_{i,L}\\ p\geq 3}}\left(\frac{(p-1)^2}{p(p-2)}\right)\right).\]
On calcule ensuite le produit pour $i=1,\ldots,\ell$ de cette expression.

\[\hspace*{-2.5cm}
\prod\limits_{i=1}^\ell \left(\sum\limits_{d_i|k}\frac{\mu(d)}{\varphi(u_i d_i)}\right)=\prod\limits_{\substack{p|k\\p\geq 3}}\left(1-\frac{1}{p-1}\right)^\ell \prod\limits_{\substack{L\in\mathcal{P}^*(E)}}\left(\frac{1}{\varphi(g_L)}\prod\limits_{\substack{p|k_L\\ p\geq 3}}\left(\frac{(p-1)^2}{p(p-2)}\right)\right)^{|L|}\prod\limits_{\substack{L\in\mathcal{P}^*(E)}}\prod\limits_{\substack{i=1\\ i\notin L}}^\ell \left(\frac{1}{\varphi(c_{i,L})}\prod\limits_{\substack{p|c_{i,L}\\ p\geq 3}}\left(\frac{(p-1)^2}{p(p-2)}\right)\right).\]

En reportant cela dans $\mathcal{P}_{a,0}'(k)$ on obtient :

\begin{align*}
\mathcal{P}_{a,0}'(k)&=\prod\limits_{\substack{p|k\\p\geq 3}}\left(1-\frac{1}{p-1}\right)^{\ell }\sum\limits_{\substack{\prod\limits_{\substack{L\in \mathscr{P}^*(E)}}k_L=k}}\prod\limits_{L\in \mathscr{P}^*(E)}\left(\frac{1}{{k'_L}^{|L|}}\prod\limits_{\substack{p|k_L\\ p\geq 3}}\left(\frac{(p-1)^2}{p(p-2)}\right)^{|L|}\right)\\
&\sum\limits_{\substack{\{g_L\}_{L\in \mathscr{P}^*(E)}\\ g_L|k_L^\infty\\ k_L|g_L}}\prod\limits_{L\in \mathscr{P}^*(E)}\left(\frac{1}{\varphi(g_L)^{|L|}}\prod\limits_{\substack{i\in E\\ i\notin L}}\left(\sum\limits_{c|\frac{g_L}{k_L}}\frac{1}{\varphi(c)}\prod\limits_{\substack{p|c\\ p\geq 3}}\left(\frac{(p-1)^2}{p(p-2)}\right)\right)\right).
\end{align*}

Calculons la somme sur $c$ à part en l'exprimant sous forme de produit eulérien, et en posant pour $p$ un nombre premier $\geq 2$ et $n$ un entier $f_p(n):=1+\frac{1-p^{-n}}{p-2}=\frac{p-1-p^{-n}}{p-2}$,

\begin{align*}
\sum\limits_{c|\frac{g_L}{k_L}}\frac{1}{\varphi(c)}\prod\limits_{\substack{p|c\\ p\geq 3}}\left(\frac{(p-1)^2}{p(p-2)}\right)&=\prod\limits_{\substack{p|\frac{g_L}{k_L}\\p\geq 3}}\left(1+\frac{p-1}{p-2}\sum\limits_{n=1}^{\nu_p(g_L)-1}\frac{1}{p^n}\right)\\
&=\prod\limits_{\substack{p|k_L\\p\geq 3}}f_p(\nu_p(g_L)-1).\\
\end{align*}

Alors en reprenant $\mathcal{P}_{a,0}'(k)$ :

\begin{align}
\mathcal{P}_{a,0}'(k)&=\prod\limits_{\substack{p|k\\p\geq 3}}\left(1-\frac{1}{p-1}\right)^{\ell }\sum\limits_{\substack{\prod\limits_{\substack{L\in \mathscr{P}^*(E)}}k_L=k}}\prod\limits_{L\in \mathscr{P}^*(E)}\left(\frac{1}{{k'_L}^{|L|}}\prod\limits_{\substack{p|k_L\\ p\geq 3}}\left(\frac{(p-1)^2}{p(p-2)}\right)^{|L|}\right)\nonumber\\
&\prod\limits_{L\in \mathscr{P}^*(E)}\left(\sum\limits_{\substack{g|k_L^\infty\\ k_L|g}}\frac{1}{\varphi(g)^{|L|}}\prod\limits_{\substack{p|k_L\\p\geq 3}}\left(f_p(\nu_p(g)-1)\right)^{\ell -|L|}\right).\label{Pa_kimpair_jusqua_g}
\end{align}

Puis, comme pour $c$, calculons la somme sur $g$ en commençant par se débarrasser de la condition $k_L|g$ :

\begin{align*}
 \sum\limits_{\substack{g|k_L^\infty\\ k_L|g}}\frac{1}{\varphi(g)^{|L|}}\prod\limits_{\substack{p|k_L\\p\geq 3}}\left(f_p(\nu_p(g)-1)\right)^{\ell -|L|}&=\frac{1}{\varphi(k_L)^{|L|}}\sum\limits_{\substack{g|k_L^\infty}}\frac{1}{g^{|L|}}\prod\limits_{\substack{p|k_L\\p\geq 3}}\left(f_p(\nu_p(g))\right)^{\ell -|L|}\\
 &=\frac{1}{\varphi(k_L)^{|L|}}\prod\limits_{\substack{p|k_L\\ p\geq 3}}\left(1+\sum\limits_{n=1}^\infty\left(\frac{1}{p^{n|L|}}\left(f_p(n)\right)^{\ell -|L|}\right)\right).
\end{align*}

Pour exprimer la somme à l'intérieur du produit on applique la formule du binôme puis celle des sommes géométriques, et en posant pour $n$ un entier strictement positif $R_p(n):=\sum\limits_{j=0}^{\ell -n}\binom{\ell -n}{j}\frac{(-1)^j(p-1)^{\ell -n-j}}{(p^{j+n}-1)(p-2)^{\ell -n}}$ on obtient :
\begin{align}
 \sum\limits_{\substack{g|k_L^\infty\\ k_L|g}}\frac{1}{\varphi(g)^{|L|}}&\prod\limits_{\substack{p|k_L\\p\geq 3}}\left(1+\frac{1-p^{-\nu_p(g)+1}}{p-2}\right)^{\ell -|L|}\nonumber\\
 &=\frac{1}{\varphi(k_L)^{|L|}}\prod\limits_{\substack{p|k_L\\ p\geq 3}}\left(1+R_p(|L|)\right).\label{somme_g_kimpair}
\end{align}  
\goodbreak
Alors en reportant \eqref{somme_g_kimpair} dans \eqref{Pa_kimpair_jusqua_g} :

\begin{equation}\hspace*{-1.5cm}
\mathcal{P}_{a,0}'(k)=\prod\limits_{\substack{p|k\\p\geq 3}}\left(1-\frac{1}{p-1}\right)^{\ell }\sum\limits_{\substack{\prod\limits_{\substack{L\in \mathscr{P}^*(E)}}k_L=k}}\prod\limits_{L\in \mathscr{P}^*(E)}\left(\frac{1}{{k'_L}^{|L|}\varphi(k_L)^{|L|}}\prod\limits_{\substack{p|k_L\\ p\geq 3}}\left(\frac{(p-1)^2}{p(p-2)}\right)^{|L|}\left(1+R_p(|L|)\right)\right).\label{Pa_kimpair_jusqua_kL}
\end{equation}

Définissons pour tout $1\leq i\leq \ell $ les fonctions multiplicatives $F_i$ par :

\[F_i(p):=\left(\frac{(h,p)(p-1)}{p^2(p-2)}\right)^i\left(1+R_p(i)\right),\]
pour $p\geq 3$.

Enfin en exprimant la somme sur les $k_L$ de \eqref{Pa_kimpair_jusqua_kL} sous forme de produit eulérien, on obtient bien :
\begin{equation}
\mathcal{P}_{a,0}'(k)=\prod\limits_{\substack{p|k\\p\geq 3}}\left(1-\frac{1}{p-1}\right)^{\ell }\overset{\ell }{\underset{i=1}{\mathlarger{\mathlarger{\mathlarger{\mathlarger{\ast}}}}}}\left(F_i^{\ast \binom{\ell }{i}}\right)\left(k\right).\label{fin_prop_impair}
\end{equation}

\end{proof}

\subsection{Découpage selon le diviseur pair de $k$ parmi les $k_L$.}

Pour le cas pair nous allons introduire une somme sur les sous-ensembles non-vides de $E$ afin de distinguer le $L_0\subset E$ tel que $2|k_{L_0}$. Par souci de lisibilité nous noterons $k'_i:=\prod\limits_{\substack{L\in \mathscr{P}^*(E)\\i\in L}}k_L'$. Ainsi,
\begin{align}
&\mathcal{P}_{a,0}'(k)=\sum\limits_{L_0\in\mathscr{P}^*(E)}\sum\limits_{\substack{\prod\limits_{\substack{L\in \mathscr{P}^*(E)}}k_L=k}}\id(2|k_{L_0})\sum\limits_{\substack{\{g_L\}_{L\in \mathscr{P}^*(E)}\\ g_L|k_L^\infty\\ k_L|g_L}}\sum\limits_{\substack{\{c_{bL}\}_{\substack{L\in \mathscr{P}^*(E)\\ b\notin L}}\\c_{bL}|\frac{g_L}{k_L}}}\prod\limits_{i\in E}\left(\sum\limits_{\substack{d|k}}\frac{\mu(d)\varepsilon\left(k'_i,u_i d\right)}{\varphi(u_i d)k'_i}\right)\nonumber\\
&=\sum\limits_{L_0\in\mathscr{P}^*(E)}\sum\limits_{\substack{\prod\limits_{\substack{L\in \mathscr{P}^*(E)}}k_L=k}}\id(2|k_{L_0})\prod\limits_{L\in \mathscr{P}^*(E)}\left(\frac{1}{{k_L'}^{|L|}}\right)\sum\limits_{\substack{\{g_L\}_{L\in \mathscr{P}^*(E)}\\ g_L|k_L^\infty\\ k_L|g_L}}\sum\limits_{\substack{\{c_{bL}\}_{\substack{L\in \mathscr{P}^*(E)\\ b\notin L}}\\c_{bL}|\frac{g_L}{k_L}}}\prod\limits_{i\in E}\left(\sum\limits_{\substack{d|k}}\frac{\mu(d)\varepsilon\left(k'_i,u_i d\right)}{\varphi(u_i d)}\right).\label{Pa_pair_jusqua_d}
\end{align}
Afin de calculer les dernières sommes sur les diviseurs de $k$ nous utilisons le lemme suivant.
\begin{lem}\label{lemme_contribution_degre}
 On a :
 \[\sum\limits_{\substack{d|k}}\frac{\mu(d)\varepsilon\left(k'_i,u_i d\right)}{\varphi(u_i d)}=\frac{1}{2\varphi(u_i)}\prod\limits_{\substack{p|k\\p|u_i\\ p\geq 3}}\left(\frac{(p-1)^2}{p(p-2)}\right)\prod\limits_{\substack{p|k\\p\geq 3}}\left(1-\frac{1}{p-1}\right)B(i,L_0)
\]

où 
\[B(i,L_0):=\left\{\begin{array}{l l}
  \id(2|c_{i,L_0}) & \text{ si } i\in E\backslash L_0,\\
  1+\id(2\tilde{a_1}|k)\delta_2(g_{L_0})\mu(\tilde{a_1})\mu\left((\tilde{a_1},u_i)\right)\prod\limits_{\substack{p|\tilde{a_1}\\p\nmid u_i\\ p\geq 3}}\left(\frac{1}{p-2}\right) & \text{ si } i\in  L_0.\end{array}\right.,\]
  \[\delta_2(g_{L_0})=\left\{\begin{array}{l}
1\ \text{si}\ \nu_2(g_{L_0})\geq s(a_1)\\
-1\ \text{si}\ \nu_2(g_{L_0})=s(a_1)-1\\
0\ \text{sinon}
\end{array}\right.,\]
avec $\tilde{a_1}:=a_1/(2,a_1)$ et $s(a_1):=\left\{\begin{array}{l}
1\ \text{si}\ a_1\equiv 1\mod 4,\\
2\ \text{si}\ a_1\equiv 3\mod 4,\\
3\ \text{si}\ a_1\equiv 0\mod 2.
\end{array}\right.$.
\end{lem}
\begin{proof}
Notons $b:=2^{s(a_1)}\tilde{a_1}$, on a
\begin{equation*}
\varepsilon(\ell,m)=\left\{\begin{array}{l}
2\ \text{si}\ 2|\ell , \text{et}\ b|m\\
1\ \text{sinon}.
\end{array}\right.
\end{equation*}

Alors si $i\in E\backslash L_0$, $k'_i$ est impair et donc $\varepsilon(k'_i,u_i d_i)=1$. Ce qui nous donne :
\begin{align*}
\sum\limits_{\substack{d|k}}\frac{\mu(d)\varepsilon\left(k'_i,u_i d\right)}{\varphi(u_i d)}=\sum\limits_{\substack{d|k}}\frac{\mu(d)}{\varphi(u_i d)}&=\frac{1}{\varphi(u_i)}\prod\limits_{\substack{p|k\\p|u_i}}\left(1-\frac{1}{p}\right)\prod\limits_{\substack{p|k\\p\nmid u_i}}\left(1-\frac{1}{p-1}\right).\\
\end{align*}
Puis en traitant à part la contribution donnée par $p=2$ :
\begin{equation}
\sum\limits_{\substack{d|k}}\frac{\mu(d)\varepsilon\left(k'_i,u_i d\right)}{\varphi(u_i d)}=\id(2|c_{i,L_0})\frac{1}{2\varphi(u_i)}\prod\limits_{\substack{p|k\\p|u_i\\ p\geq 3}}\left(\frac{(p-1)^2}{p(p-2)}\right)\prod\limits_{\substack{p|k\\p\geq 3}}\left(1-\frac{1}{p-1}\right),\label{somme_d_pas_L0}
\end{equation}
car $\prod\limits_{\substack{p|k\\ p\nmid u_i}}\left(1-\frac{1}{p-1}\right)=0$ si $2\nmid u_i$, c'est-à-dire si $2\nmid c_{i,L_0}$.\\ 

Puis, lorsque $i\in L_0$, $\varepsilon(k'_i,u_i d)=1+\id(b|u_i d)$, ainsi, en utilisant le lemme \ref{lemme_somme_multiplicative} :
\begin{align}
\sum\limits_{\substack{d|k}}\frac{\mu(d)\varepsilon\left(k'_i,u_i d\right)}{\varphi(u_i d)}&=\sum\limits_{\substack{d|k}}\frac{\mu(d)}{\varphi(u_i d)}+\sum\limits_{\substack{d|k\\ b|u_i d}}\frac{\mu(d)}{\varphi(u_i d)}\nonumber\\
&=\frac{1}{2\varphi(u_i)}\prod\limits_{\substack{p|k\\p|u_i\\ p\geq 3}}\left(\frac{(p-1)^2}{p(p-2)}\right)\prod\limits_{\substack{p|k\\p\geq 3}}\left(1-\frac{1}{p-1}\right) +\sum\limits_{\substack{d|k\\ b|u_i d}}\frac{\mu(d)}{\varphi(u_i d)}.\label{somme_d_L0}
\end{align}

Calculons cette dernière somme, en notant que si $b$ divise $u_i d$ alors $\frac{b}{(b,u_i)}$ divise $d$ et a fortiori $k$.
\begin{align*}
 \sum\limits_{\substack{d|k\\ b|u_i d}}\frac{\mu(d)}{\varphi(u_i d)}&=\id\left(\frac{b}{(u_i,b)}|k\right)\sum\limits_{\substack{d'|k/\frac{b}{(u_i,b)}}}\frac{\mu(d'\frac{b}{(u_i,b)})}{\varphi(u_i d'\frac{b}{(u_i,b)})}.\\
 \end{align*}
 Puis comme $k$ est sans facteur carré, les diviseurs $d'$ de la somme précédente vérifient $\left(\frac{b}{(b,u_i)},d'\right)=1$, ainsi, en appliquant le lemme \ref{lemme_somme_multiplicative} :
 \begin{align}
 \sum\limits_{\substack{d|k\\ b|u_i d}}\frac{\mu(d)}{\varphi(u_i d)}&=\id\left(\frac{b}{(u_i,b)}|k\right)\frac{\mu\left(\frac{b}{(u_i,b)}\right)}{\varphi\left(u_i \frac{b}{(u_i,b)}\right)}\sum\limits_{\substack{d'|k/\frac{b}{(u_i,b)}}}\frac{\mu(d')\varphi\left(u_i \frac{b}{(u_i,b)}\right)}{\varphi(u_i d'\frac{b}{(u_i,b)})}\nonumber\\
 &=\id\left(\frac{b}{(u_i,b)}|k\right)\frac{\mu\left(\frac{b}{(u_i,b)}\right)}{\varphi\left(u_i \frac{b}{(u_i,b)}\right)}\prod\limits_{\substack{p|\frac{k(u_i,b)}{b}\\ p|u_i }}\left(1-\frac{1}{p}\right)\prod\limits_{\substack{p|\frac{k(u_i,b)}{b}\\ p\nmid u_i }}\left(1-\frac{1}{p-1}\right).\label{somme_dpair}
\end{align}

Décomposons chaque élément de cette dernière égalité :

\begin{equation}\id\left(\frac{b}{(u_i,b)}|k\right)=\id(2\tilde{a_1}|k)\id(\nu_2(g_{L_0})\geq s(a_1)-1),\label{prodid}\end{equation}
car $\nu_2\left(\frac{b}{(b,u_i)}\right)=s(a_1)-\min \left(\nu_2(g_{L_0}),s(a_1)\right)$ et $\nu_2(k)=1$.\\

Alors, comme on a $\nu_2(g_{L_0})\geq s(a_1)-1$ cela implique que :
\begin{equation}\mu\left(\frac{b}{(u_i,b)}\right)=\mu(\tilde{a_1})\mu\left((\tilde{a_1},u_i)\right)\mu\left(2^{\id\left(\nu_2(g_{L_0})=s(a_1)-1\right)}\right).\label{prodmu}\end{equation}

Puis, comme $2||\frac{b}{(u_i,b)}$ ou $2\nmid\frac{b}{(u_i,b)}$,
\begin{align*}
 \frac{1}{\varphi\left(u_i \frac{b}{(u_i,b)}\right)}&=\frac{1}{\varphi(u_i)\varphi\left(\frac{b}{(u_i,b)}\right)}\left(1-\frac{1}{2}\id\left(\nu_2(g_{L_0})=s(a_1)-1\right)\right)\\
 &=\frac{1}{\varphi(u_i)\varphi\left(\frac{\tilde{a_1}}{(u_i,\tilde{a_1})}\right)}\left(1-\frac{1}{2}\id\left(\nu_2(g_{L_0})=s(a_1)-1\right)\right),\\
 \end{align*}
 et en rappelant que $\tilde{a_1}$ est sans facteur carré on obtient :
 \begin{align}
 \frac{1}{\varphi\left(u_i \frac{b}{(u_i,b)}\right)}&=\frac{1}{\varphi(u_i)}\prod\limits_{\substack{p|\tilde{a_1}\\p\nmid u_i}}\left(\frac{1}{p-1}\right)\left(1-\frac{1}{2}\id\left(\nu_2(g_{L_0})=s(a_1)-1\right)\right).\label{prodphi}
\end{align}

Décomposons ensuite les deux produits de \eqref{somme_dpair} :

\begin{align*}
 \prod\limits_{\substack{p|\frac{k(u_i,b)}{b}\\ p|u_i }}\left(1-\frac{1}{p}\right)&=\prod\limits_{\substack{p|k\\ p|u_i }}\left(1-\frac{1}{p}\right)\prod\limits_{\substack{p|\frac{b}{(u_i,b)}\\ p|u_i }}\left(1-\frac{1}{p}\right)^{-1},\\
 \end{align*}
 or pour $p\geq 3$, $\nu_p(b)\leq 1$ et ainsi $\nu_p\left(u_i,\frac{b}{(u_i,b)}\right)=0$, donc
\begin{align}
 \prod\limits_{\substack{p|\frac{k(u_i,b)}{b}\\ p|u_i }}\left(1-\frac{1}{p}\right)&=\frac{1}{2}\prod\limits_{\substack{p|k\\ p|u_i \\p\geq 3}}\left(1-\frac{1}{p}\right)\left(1+\id\left(\nu_2(g_{L_0})=s(a_1)-1\right)\right).\label{proddiv}
\end{align}

et, comme $2|u_i$,

\begin{align}
 \prod\limits_{\substack{p|\frac{k(u_i,b)}{b}\\ p\nmid u_i }}\left(1-\frac{1}{p-1}\right)&=\prod\limits_{\substack{p|k\\ p\nmid u_i }}\left(1-\frac{1}{p-1}\right)\prod\limits_{\substack{p|\frac{b}{(u_i,b)}\\ p\nmid u_i }}\left(1-\frac{1}{p-1}\right)^{-1}\nonumber\\
 &=\prod\limits_{\substack{p|k\\ p\nmid u_i }}\left(1-\frac{1}{p-1}\right)\prod\limits_{\substack{p|\tilde{a_1}\\ p\nmid u_i }}\left(1-\frac{1}{p-1}\right)^{-1}.\label{prodnodiv}
\end{align}

Notons :

\begin{equation*}
\delta_2(g_{L_0})=\left\{\begin{array}{l}
1\ \text{si}\ \nu_2(g_{L_0})\geq s(a_1)\\
-1\ \text{si}\ \nu_2(g_{L_0})=s(a_1)-1\\
0\ \text{sinon}
\end{array}\right.
\end{equation*}

Ainsi en reprenant \eqref{prodid}, \eqref{prodmu}, \eqref{prodphi}, \eqref{proddiv} et \eqref{prodnodiv} dans \eqref{somme_dpair} :

\begin{equation}\sum\limits_{\substack{d|k\\ b|u_i d}}\frac{\mu(d)}{\varphi(u_i d)}=\id(2\tilde{a_1}|k)\delta_2(g_{L_0})\frac{\mu(\tilde{a_1})\mu\left((\tilde{a_1},u_i)\right)}{2\varphi(u_i)}\prod\limits_{\substack{p|\tilde{a_1}\\p\nmid u_i\\ p\geq 3}}\left(\frac{1}{p-2}\right)\prod\limits_{\substack{p|k\\p|u_i\\ p\geq 3}}\left(\frac{(p-1)^2}{p(p-2)}\right)\prod\limits_{\substack{p|k\\p\geq 3}}\left(1-\frac{1}{p-1}\right).\label{somme_d_buid}\end{equation}

Ce qui, en reprenant \eqref{somme_d_pas_L0} et \eqref{somme_d_L0}, permet de conclure.

\end{proof}

On a donc dans l'expression \eqref{Pa_pair_jusqua_d} de $\mathcal{P}_{a,0}'(k)$ un terme $\prod\limits_{i\in E}B(i,L_0)$, intéressons-nous dans ce dernier au produit sur les $i\in L_0$. Nous développons le produit $\prod\limits_{i\in E}B(i,L_0)$ ce qui nous amène à introduire une somme sur les sous-ensembles $L_1\subset L_0$ non vides.
\goodbreak
\begin{align}
 \prod\limits_{i\in L_0}&\left(1+\id(2\tilde{a_1}|k)\delta_2(g_{L_0})\mu(\tilde{a_1})\mu\left((\tilde{a_1},u_i)\right)\prod\limits_{\substack{p|\tilde{a_1}\\p\nmid u_i\\ p\geq 3}}\left(\frac{1}{p-2}\right)\right)\nonumber\\
 &=1+\sum\limits_{L_1\in \mathscr{P}^*(L_0)}\id(2\tilde{a_1}|k)\delta_2(g_{L_0})^{|L_1|}\mu(\tilde{a_1})^{|L_1|}\prod\limits_{\substack{p|\tilde{a_1}\\ p\geq 3}}\left(\frac{1}{p-2}\right)^{|L_1|}\prod\limits_{i\in L_1}\left(\prod\limits_{\substack{p|\tilde{a_1}\\ p|u_i\\ p\geq 3}}\left(2-p\right)\right),\label{separation_alpha_beta}
\end{align}

car $\mu\left((\tilde{a_1},u_i)\right)=\prod\limits_{\substack{p|\tilde{a_1}\\ p|u_i}}(-1)$. Notons $\mathcal{P}_\alpha'(k)$ la contribution à $\mathcal{P}_{a,0}'(k)$ du terme $1$ dans l'expression qui précède, et  $\mathcal{P}_\beta'(k)$ la contribution à $\mathcal{P}_{a,0}'(k)$ du terme $\sum\limits_{L_1\in \mathscr{P}^*(L_0)}\id(2\tilde{a_1}|k)\delta_2(g_{L_0})^{|L_1|}\mu(\tilde{a_1})^{|L_1|}\prod\limits_{\substack{p|\tilde{a_1}\\ p\geq 3}}\left(\frac{1}{p-2}\right)^{|L_1|}\prod\limits_{i\in L_1}\left(\prod\limits_{\substack{p|\tilde{a_1}\\ p|u_i\\ p\geq 3}}\left(2-p\right)\right)$. Les deux prochains paragraphes ont pour objectif d'évaluer ces deux quantités.

\subsection{Évaluation de $\mathcal{P}_\alpha'(k)$.}
L'évaluation de $\mathcal{P}_\alpha'(k)$ est semblable à celle de $\mathcal{P}_{a,0}'(k)$ dans le cas $k$ impair, nous allons exprimer successivement les sommes sur les $c_{i,L}$, $g_L$ et $k_L$ en produits eulériens.
\begin{prop}\label{prop_Palpha}
On a, en reprenant les définitions de $H_1$ et $F_i$ de la proposition \ref{prop_terme_principal},
\begin{align*}
 \mathcal{P}_\alpha'(k)&=H_1(E)\prod\limits_{\substack{p|k\\p\geq 3}}\left(1-\frac{1}{p-1}\right)^{\ell }\overset{\ell }{\underset{i=1}{\mathlarger{\mathlarger{\mathlarger{\mathlarger{\ast}}}}}}\left(F_i^{\ast \binom{\ell }{i}}\right)\left(\frac{k}{2}\right).
\end{align*}

\end{prop}
\begin{proof}
Commençons par expliciter $\mathcal{P}_\alpha'(k)$ en utilisant le lemme \ref{lemme_contribution_degre}, on obtient

 \begin{equation}\hspace*{-2cm}
 \mathcal{P}_\alpha'(k):=\alpha_1(k)\sum\limits_{L_0\in\mathscr{P}^*(E)}\sum\limits_{\substack{\prod\limits_{\substack{L\in \mathscr{P}^*(E)}}k_L=k}}\sum\limits_{\substack{\{g_L\}_{L\in \mathscr{P}^*(E)}\\ g_L|k_L^\infty\\ k_L|g_L}}\prod\limits_{L\in \mathscr{P}^*(E)}\bigg(\alpha_2(k_L)\alpha_3(g_L)\bigg)\sum\limits_{\substack{\{c_{bL}\}_{\substack{L\in \mathscr{P}^*(E)\\ b\notin L}}\\c_{bL}|\frac{g_L}{k_L}}}\prod\limits_{\substack{\{i,L\}\\ L\in \mathscr{P}^*(E)\\ i\notin L}}\alpha_4(c_{i,L})\label{def_Palpha},
\end{equation}
où 

\begin{itemize}
 \item [$\bullet$]$\alpha_1(k):=2^{-\ell }\prod\limits_{\substack{p|k\\p\geq 3}}\left(1-\frac{1}{p-1}\right)^{\ell }$,
 \item [$\bullet$]$\alpha_2(k_L):=\left\{\begin{array}{ll}\frac{\id(2|k_{L_0})}{{k_{L_0}'}^{|{L_0}|}}\prod\limits_{\substack{p|k_{L_0}\\ p\geq 3}}\left(\frac{(p-1)^2}{p(p-2)}\right)^{|{L_0}|} & \text{ si }L=L_0,\\ \frac{1}{{k_L'}^{|L|}}\prod\limits_{\substack{p|k_L\\ p\geq 3}}\left(\frac{(p-1)^2}{p(p-2)}\right)^{|L|} & \text{ sinon}, \end{array}\right.$,

 \item [$\bullet$]$\alpha_3(g_L):=\frac{1}{\varphi(g_L)^{|L|}}$,
 \item [$\bullet$]$\alpha_4(c_{i,L}):=\left\{\begin{array}{ll}\frac{\id(2|c_{i,L_0})}{\varphi(c_{i,L})}\prod\limits_{\substack{p|k\\p|c_{i,L}\\ p\geq 3}}\left(\frac{(p-1)^2}{p(p-2)}\right) & \text{ si }L=L_0,\\\frac{1}{\varphi(c_{i,L})}\prod\limits_{\substack{p|k\\p|c_{i,L}\\ p\geq 3}}\left(\frac{(p-1)^2}{p(p-2)}\right)  & \text{ sinon}, \end{array}\right.$.
 \end{itemize}
 
Occupons-nous d'abord des somme sur les $c$, pour ce faire on introduit une autre fonction indicatrice pour tenir compte des $\id(2|c_{i,L_0})$ :
\begin{equation*}
\delta_3(g_{L_0})=\left\{\begin{array}{l}
\left(2\left(1-2^{-\nu_2(g_{L_0})+1}\right)\right)^{\ell -|L_0|}\ \text{si}\ 4|g_{L_0}\ \text{et} L_0\neq E,\\
1\ \text{si}\ 2|g_{L_0}\ \text{et}\ L_0=E,\\
0\ \text{sinon}
\end{array}\right.
\end{equation*}

En effet, si $L_0\neq E$ alors il existe $i\in E\backslash L_0$ avec $2|c_{i,L_0}|\frac{g_{L_0}}{k_{L_0}}$ et donc $4|g_{L_0}$et dans ce cas

\[\sum\limits_{\substack{c_{i,L_0}|\frac{g_{L_0}}{k_{L_0}}\\ 2|c_{i,L_0}}}\frac{1}{\varphi(c_{i,L_0})}\prod\limits_{\substack{p|k\\p|c_{i,L_0}\\ p\geq 3}}\left(\frac{(p-1)^2}{p(p-2)}\right)=\left(2\left(1-2^{-\nu_2(g_{L_0})+1                                                                                                                                                                                                                                                                                                                                                                                                                                                                                                                                                                                                                                                                                                                                                                                                                                                                                                                                                                                                                                                                                                                                                                                                                                                                                                                                                                                                                                                                                                                                                                                                                     }\right)\right)\sum\limits_{\substack{c_{i,L_0}|\frac{g_{L_0}}{k_{L_0}}\\ 2\nmid c_{i,L_0}}}\frac{1}{\varphi(c_{i,L_0})}\prod\limits_{\substack{p|k\\p|c_{i,L_0}\\ p\geq 3}}\left(\frac{(p-1)^2}{p(p-2)}\right).\]
Ainsi
\begin{equation*}
 \sum\limits_{\substack{\{c_{bL}\}_{\substack{L\in \mathscr{P}^*(E)\\ b\notin L}}\\c_{bL}|\frac{g_L}{k_L}}}\prod\limits_{\substack{\{i,L\}\\ L\in \mathscr{P}^*(E)\\ i\notin L}}\alpha_4(c_{i,L})=\delta_3(g_{L_0})\prod\limits_{i\in E}\left(\prod\limits_{\substack{L\in \mathscr{P}^*(E)\\ i\notin L}}\left(\sum\limits_{\substack{c|\frac{g_L}{k_L}\\ 2\nmid c}}\frac{1}{\varphi(c)}\prod\limits_{\substack{p|k\\p|c\\ p\geq 3}}\left(\frac{(p-1)^2}{p(p-2)}\right)\right)\right)
 \end{equation*}
 Puis en exprimant la somme sur $c$ comme un produit eulérien :
 \begin{align}
 \sum\limits_{\substack{\{c_{bL}\}_{\substack{L\in \mathscr{P}^*(E)\\ b\notin L}}\\c_{bL}|\frac{g_L}{k_L}}}\prod\limits_{\substack{\{i,L\}\\ L\in \mathscr{P}^*(E)\\ i\notin L}}\alpha_4(c_{i,L}&)=\delta_3(g_{L_0})\prod\limits_{\substack{L\in \mathscr{P}^*(E)}}\left(\prod\limits_{\substack{p|k_L\\p\geq 3}}\left(1+\frac{p-1}{p-2}\sum\limits_{n=1}^{\nu_p(g_L)-1}\frac{1}{p^n}\right)^{\ell -|L|}\right)\nonumber\\
  &=\delta_3(g_{L_0})\prod\limits_{\substack{L\in \mathscr{P}^*(E)}}\left(\prod\limits_{\substack{p|k_L\\p\geq 3}}\left(1+\frac{1-p^{-\nu_p(g_L)+1}}{p-2}\right)^{\ell -|L|}\right)\label{somme_cil_alpha}
\end{align}

Puis exprimons séparément la somme sur $g_{L_0}$.
\begin{align}
 \sum\limits_{\substack{\{g_{L_0}\}_{{L_0}\in \mathscr{P}^*(E)}\\ g_{L_0}|k_{L_0}^\infty\\ k_{L_0}|g_{L_0}}}&\delta_3(g_{L_0})\frac{1}{\varphi(g_{L_0})^{|L_0|}}\prod\limits_{\substack{p|k_{L_0}\\p\geq 3}}\left(1+\frac{1-p^{-\nu_p(g_L)+1}}{p-2}\right)^{\ell -|L_0|}\nonumber\\
 &=\frac{1}{\varphi(k_{L_0})^{|L_0|}}\sum\limits_{\substack{\{g_{L_0}\}_{{L_0}\in \mathscr{P}^*(E)}\\ g_{L_0}|k_{L_0}^\infty}}\delta_3(2g_{L_0})\frac{1}{g_{L_0}^{|L_0|}}\prod\limits_{\substack{p|k_{L_0}\\p\geq 3}}\left(1+\frac{1-p^{-\nu_p(g_L)}}{p-2}\right)^{\ell -|L_0|}.\label{somme_gl0_Palpha}
\end{align}
Alors en posant :
\begin{equation*}
\delta_4(L_0)=\left\{\begin{array}{l}
2^{-|L_0|}\left(1+2^{\ell -|L_0|}R_2(|L_0|)\right)\ \text{si}\ L_0\neq E,\\
\frac{2^{\ell }}{2^{\ell }-1}\ \text{si}\ L_0=E,
\end{array}\right.
\end{equation*}

on obtient en reprenant \eqref{somme_gl0_Palpha} dans \eqref{def_Palpha} et en en procédant de la même manière que pour \eqref{fin_prop_impair} :

\begin{align*}
 \mathcal{P}_\alpha'(k)&=2^{-\ell }\sum\limits_{L_0\in\mathscr{P}^*(E)}\frac{\delta_4(L_0)}{2^{|L_0|}}\prod\limits_{\substack{p|k\\p\geq 3}}\left(1-\frac{1}{p-1}\right)^{\ell }\overset{\ell }{\underset{i=1}{\mathlarger{\mathlarger{\mathlarger{\mathlarger{\ast}}}}}}\left(F_i^{\ast \binom{\ell }{i}}\right)\left(\frac{k}{2}\right).
\end{align*}
Reste à évaluer $2^{-\ell }\sum\limits_{L_0\in\mathscr{P}^*(E)}\frac{\delta_4(L_0)}{2^{|L_0|}}$, on a
\begin{align*}
 2^{-\ell }\sum\limits_{L_0\in\mathscr{P}^*(E)}\frac{\delta_4(L_0)}{2^{|L_0|}}&=2^{-\ell }\left(\sum\limits_{i=1}^{\ell -1}\binom{\ell }{i}2^{-2i}\left(1+2^{\ell -i}R_2(i)\right)+\frac{1}{2^{\ell }-1}\right)\\
 &=2^{-\ell }\left(\sum\limits_{i=1}^{\ell }\binom{\ell }{i}2^{-2i}\left(1+2^{\ell -i}R_2(i)\right)+2^{-\ell }\right).\\
\end{align*}

\end{proof}

\subsection{Évaluation de $\mathcal{P}_\beta'(k)$.}
L'évaluation de $\mathcal{P}_\beta'(k)$ va se dérouler dans le même esprit que celle de $\mathcal{P}_\alpha'(k)$, mais avec quelques subtilités importantes. En effet la présence de termes dépendants de $\tilde{a_1}$, et notamment $\id(2\tilde{a_1}|k)$, va nous imposer d'introduire une somme sur toutes les décompositions possibles de $\tilde{a_1}$ en produits de $|\mathscr{P}^*(E)|$ facteurs.

\begin{prop}\label{prop_Pbeta}
On a, en reprenant les définitions de $H_2$ et $F_i$ de la proposition \ref{prop_terme_principal},

 \[\mathcal{P}_\beta'(k)=H_2(\ell,a_1)\id(2\tilde{a_1}|k)\prod\limits_{\substack{p|\frac{k}{2\tilde{a_1}}\\p\geq 3}}\left(1-\frac{1}{p-1}\right)^{\ell }\overset{\ell }{\underset{i=1}{\mathlarger{\mathlarger{\mathlarger{\mathlarger{\ast}}}}}}\left(F_i^{\ast \binom{\ell }{i}}\right)(\frac{k}{2\tilde{a_1}}).\]
\end{prop}
\begin{proof}

Commençons par expliciter $\mathcal{P}_\alpha'(k)$, en remarquant que
\[\prod\limits_{\substack{p|\tilde{a_1}\\ p| u_i\\ p\geq 3}}(2-p)=\prod\limits_{\substack{L\in\mathscr(P)^*(E)\\ i\in L}}\left(\prod\limits_{\substack{p|\tilde{a_1}\\p|g_L\\p\geq 3}}(2-p)\right)\prod\limits_{\substack{L\in\mathscr(P)^*(E)\\ i\notin L}}\left(\prod\limits_{\substack{p|\tilde{a_1}\\ p|c_{i,L}\\ p\geq 3}}(2-p)\right)\]

et en utilisant le lemme \ref{lemme_contribution_degre}, on obtient

\begin{align}\label{def_Pbeta}
 \mathcal{P}_\beta'(k):=&\beta_1(k)\sum\limits_{L_0\in\mathscr{P}^*(E)}\sum\limits_{L_1\in \mathscr{P}^*(L_0)}\beta_2(L_1)\sum\limits_{\substack{\prod\limits_{\substack{L\in \mathscr{P}^*(E)}}k_L=k}}\prod\limits_{L\in \mathscr{P}^*(E)}\beta_3(k_L)\\
 &\sum\limits_{\substack{\{g_L\}_{L\in \mathscr{P}^*(E)}\\ g_L|k_L^\infty\\ k_L|g_L}}\prod\limits_{L\in \mathscr{P}^*(E)}\beta_4(g_L)\sum\limits_{\substack{\{c_{bL}\}_{\substack{L\in \mathscr{P}^*(E)\\ b\notin L}}\\c_{bL}|\frac{g_L}{k_L}}}\prod\limits_{\substack{\{i,L\}\\ L\in \mathscr{P}^*(E)\\ i\notin L}}\beta_5(c_{i,L})\nonumber
\end{align}
où
\begin{itemize}
\item[$\bullet$] $\beta_1(k):=2^{-\ell }\id(2\tilde{a_1}|k)\prod\limits_{\substack{p|k\\p\geq 3}}\left(1-\frac{1}{p-1}\right)^{\ell }$,
\item[$\bullet$]$\beta_2(L_1):=\mu(\tilde{a_1})^{|L_1|}\prod\limits_{\substack{p|\tilde{a_1}\\ p\geq 3}}\left(\frac{1}{p-2}\right)^{|L_1|}$,
\item[$\bullet$]$\beta_3(k_L):=\left\{\begin{array}{ll}\frac{\id(2|k_{L_0})}{{k_{L_0}'}^{|{L_0}|}}\prod\limits_{\substack{p|k_{L_0}\\ p\geq 3}}\left(\frac{(p-1)^2}{p(p-2)}\right)^{|{L_0}|}\prod\limits_{\substack{p|\tilde{a_1}\\ p|k_{L_0}\\ p\geq 3}}\left(2-p\right)^{|{L_0}\cap L_1|} & \text{ si }L=L_0,\\\frac{1}{{k_L'}^{|L|}}\prod\limits_{\substack{p|k_L\\ p\geq 3}}\left(\frac{(p-1)^2}{p(p-2)}\right)^{|L|}\prod\limits_{\substack{p|\tilde{a_1}\\ p|k_L\\ p\geq 3}}\left(2-p\right)^{|L\cap L_1|} & \text{ sinon,}\end{array}\right.$,
\item[$\bullet$]$\beta_4(g_L):=\left\{\begin{array}{ll}\frac{\delta_2(g_{L_0})^{|L_1|}}{\varphi(g_{L_0})^{|{L_0}|}}  & \text{ si }L=L_0,\\\frac{1}{\varphi(g_L)^{|L|}} & \text{ sinon,}\end{array}\right.$, avec $\delta_2(g_{L_0})=\left\{\begin{array}{l}
1\ \text{si}\ \nu_2(g_{L_0})\geq s(a_1)\\
-1\ \text{si}\ \nu_2(g_{L_0})=s(a_1)-1\\
0\ \text{sinon}
\end{array}\right.
$,
\item[$\bullet$]$\beta_5(c_{i,L}):=\left\{\begin{array}{ll}\frac{1}{\varphi(c_{i,L})}\prod\limits_{\substack{p|k\\p|c_{i,L}\\ p\geq 3}}\left(\frac{(p-1)^2}{p(p-2)}\right)\prod\limits_{\substack{L\in \mathscr{P}^*(E)\\ i\notin L}}\left(\prod\limits_{\substack{p|\tilde{a_1}\\ p|c_{i,L}\\ p\geq 3}}\left(2-p\right)\right) & \text{ si }i\in L_1,\\
\frac{\id(2|c_{i,L_0})}{\varphi(c_{i,L})}\prod\limits_{\substack{p|k\\p|c_{i,L}\\ p\geq 3}}\left(\frac{(p-1)^2}{p(p-2)}\right)& \text{ si }i\in E\backslash L_0,\\
\frac{1}{\varphi(c_{i,L})}\prod\limits_{\substack{p|k\\p|c_{i,L}\\ p\geq 3}}\left(\frac{(p-1)^2}{p(p-2)}\right)& \text{ sinon,}\end{array}\right.$.
\end{itemize}

Les sommes sur les $c_{i,L}$ pour $i\notin L_1$ sont exactement les mêmes que pour \eqref{somme_cil_alpha}, intéressons-nous au cas $i\in L_1$ et $i\notin L$ (et donc $2\nmid k_L$) : 
\begin{align*}
 \sum\limits_{c_{i,L}| \frac{g_L}{k_L}}\frac{1}{\varphi(c_{i,L})}\prod\limits_{\substack{p|k\\p|c_{i,L}\\ p\geq 3}}\left(\frac{(p-1)^2}{p(p-2)}\right)\prod\limits_{\substack{p|\tilde{a_1}\\ p|c_{i,L}\\ p\geq 3}}\left(2-p\right)&=\prod\limits_{\substack{p|k_L\\ p\nmid \tilde{a_1}\\p\geq 3}}\left(f_p(\nu_p(g_L)-1)\right)\prod\limits_{\substack{p|k_L\\ p| \tilde{a_1}\\p\geq 3}}p^{-\nu_p(g_L)+1},
\end{align*}
où on a posé pour $n\geq 1$, et $p\geq3$, $f_p(n):=1+\frac{1-p^{-n}}{p-2}$.

et ainsi en reprenant le calcul de \eqref{somme_cil_alpha},

\begin{align*}
 &\sum\limits_{\substack{\{c_{bL}\}_{\substack{L\in \mathscr{P}^*(E)\\ b\notin L}}\\c_{bL}|\frac{g_L}{k_L}}}\prod\limits_{\substack{\{i,L\}\\ L\in \mathscr{P}^*(E)\\ i\notin L}}\left(\frac{1}{\varphi(c_{i,L})}\prod\limits_{\substack{p|k\\p|c_{i,L}\\ p\geq 3}}\left(\frac{(p-1)^2}{p(p-2)}\right)\right)\prod\limits_{i\in L_1}\left(\prod\limits_{\substack{L\in \mathscr{P}^*(E)\\ i\notin L}}\left(\prod\limits_{\substack{p|\tilde{a_1}\\ p|c_{i,L}\\ p\geq 3}}\left(2-p\right)\right)\right)\prod\limits_{i\in E\backslash L_0}\left(\id(2|c_{i,L_0})\right)\\
 &=\delta_3(g_{L_0})\prod\limits_{\substack{L\in \mathscr{P}^*(E)}}\left(\prod\limits_{\substack{p|k_L\\p\geq 3}}\left(f_p(\nu_p(g_L)-1)\right)^{\ell -|L\cup L_1|}\prod\limits_{\substack{p|k_L\\ p\nmid \tilde{a_1}\\p\geq 3}}\left(f_p(\nu_p(g_L)-1)\right)^{|L_1\backslash L|}\prod\limits_{\substack{p|k_L\\ p| \tilde{a_1}\\p\geq 3}}\left(p^{-\nu_p(g_L)+1}\right)^{|L_1 \backslash L|}\right)
\end{align*}

Passons aux sommes sur les $g$, en commençant par $g_{L_0}$ :

\begin{align*}
\sum\limits_{\substack{g_{L_0}|k_{L_0}^\infty\\ k_{L_0}|g_{L_0}}}&\delta_2(g_{L_0})^{|L_1|}\delta_3(g_{L_0})\frac{1}{\varphi(g_{L_0})^{|L|}}\prod\limits_{\substack{p|k_{L_0}\\p\geq 3}}\left(f_p(\nu_p(g_{L_0})-1)\right)^{\ell -|L_0|}\\
&= \frac{1}{\varphi(k_{L_0})^{|L_0|}}\sum\limits_{\substack{g_{L_0}|k_{L_0}^\infty}}\delta_2(2g_{L_0})^{|L_1|}\delta_3(2g_{L_0})\frac{1}{g_{L_0}^{|L_0|}}\prod\limits_{\substack{p|k_{L_0}\\p\geq 3}}\left(f_p(\nu_p(g_{L_0}))\right)^{\ell -|L_0|}.
\end{align*}
Exprimons les valeurs prises par $\delta_2(2g_{L_0})$ suivant les valeurs de $s(a_1)$ et celles prises par $\delta_3(2g_{L_0})$ suivant si $E=L_0$ ou $E\neq 0$.
\[
\begin{array}{|c|c|c|c|c|}\hline
 s(a_1) & \multicolumn{2}{|c|}{\delta_2(2g_{L_0})} & \multicolumn{2}{|c|}{\delta_3(2g_{L_0})}\\ \hline
 &\multicolumn{2}{c|}{}&L_0=E & L_0\neq E\\ \hline
 1 &  \multicolumn{2}{c|}{1} & 1 & \left\{\begin{array}{l}
\left(2\left(1-2^{-\nu_2(g_{L_0})}\right)\right)^{\ell -|L_0|}\ \text{si}\ 2|g_{L_0},\\
0\ \text{sinon}
\end{array}\right. \\ \hline
 2 &  \multicolumn{2}{c|}{\left\{\begin{array}{l}
1\ \text{si}\ \nu_2(g_{L_0})\geq 1\\
-1\ \text{si}\ \nu_2(g_{L_0})=0
\end{array}\right.} & 1 & \left\{\begin{array}{l}
\left(2\left(1-2^{-\nu_2(g_{L_0})}\right)\right)^{\ell -|L_0|}\ \text{si}\ 2|g_{L_0},\\
0\ \text{sinon}
\end{array}\right. \\ \hline
 3 &  \multicolumn{2}{c|}{\left\{\begin{array}{l}
1\ \text{si}\ \nu_2(g_{L_0})\geq 2\\
-1\ \text{si}\ \nu_2(g_{L_0})=1\\
0\ \text{sinon}
\end{array}\right.} & 1 & \left\{\begin{array}{l}
\left(2\left(1-2^{-\nu_2(g_{L_0})}\right)\right)^{\ell -|L_0|}\ \text{si}\ 2|g_{L_0},\\
0\ \text{sinon}
\end{array}\right. \\ \hline
\end{array}
\]

Définissons alors une fonction $\delta_5(L_0,L_1)$ comme suit :

\[
\begin{array}{|c|c|c|}\hline
s(a_1) & \multicolumn{2}{|c|}{\delta_5(L_0,L_1)}\\ \hline
& E=L_0 & E\neq L_0 \\ \hline
1 & \frac{2^{\ell }}{2^{\ell }-1} & 2^{-|L_0|}\left(1+2^{\ell -|L_0|}R_2(|L_0|)\right) \\ \hline
2 & (-1)^{|L_1|}\left(\frac{2^{\ell }-1+(-1)^{|L_1|}}{2^{\ell }-1}\right) & 2^{-|L_0|}\left(1+2^{\ell -|L_0|}R_2(|L_0|)\right) \\ \hline
3 & \frac{(-1)^{|L_1|}}{2^{\ell }}\left(\frac{2^{\ell }-1+(-1)^{|L_1|}}{2^{\ell }-1}\right) & (-1)^{|L_1|}2^{-|L_0|}\left(1+(-1)^{|L_1|}2^{\ell -|L_0|}R_2(|L_0|)\right) \\ \hline
\end{array}
\]
%
On obtient, en exprimant la somme sur $g_{L_0}$ sous forme de produit eulérien :

\begin{align}
\sum\limits_{\substack{g_{L_0}|k_{L_0}^\infty\\ k_{L_0}|g_{L_0}}}&\delta_2(g_{L_0})^{|L_1|}\delta_3(g_{L_0})\frac{1}{\varphi(g_{L_0})^{|L|}}\prod\limits_{\substack{p|k_{L_0}\\p\geq 3}}\left(f_p(\nu_p(g_{L_0})-1)\right)^{\ell -|L_0|}\nonumber\\
&= \delta_5(L_0,L_1)\frac{1}{\varphi(k_{L_0})^{|L_0|}}\prod\limits_{\substack{p|k_{L_0}\\p\geq 3}}\left(1+R_p(|{L_0}|)\right).\label{somme_gL0_beta}
\end{align}

Définissons
\begin{equation*}
 \widehat{\delta}(L_0,L_1)=\left\{\begin{array}{l}
                        (-1)^{|L_1|}2^{-|L_0|}\left(1+(-1)^{|L_1|}2^{\ell -|L_0|}R_2(|L_0|)\right) \text{ si }s(a_1)=3,\\
                        2^{-|L_0|}\left(1+2^{\ell -|L_0|}R_2(|L_0|)\right)  \text{ sinon. }
                      \end{array}
\right.
\end{equation*}
Et une fonction $\tilde{\delta}(L_0,L_1)$,
\begin{equation*}\hspace*{-3cm}
 \tilde{\delta}(L_0,L_1)=\left\{\begin{array}{l}
                        1\text{ si } L_0=E \text{ et } s(a_1)=1,\\
                        (-1)^{|L_1|}\text{ si } L_0=E \text{ et }s(a_1)=2,\\
                        0\text{ sinon. }\\
                      \end{array}
\right.
\end{equation*}

On a alors $\delta_5(L_0,L_1)=\widehat{\delta}(L_0,L_1)+\tilde{\delta}(L_0,L_1)$.

Puis pour $L\neq L_0$ :
\begin{align}
 \sum\limits_{\substack{g|k_{L}^\infty\\ k_{L}|g}}&\frac{1}{\varphi(g)^{|L|}}
 \prod\limits_{\substack{p|k_L\\p\geq 3}}\left(f_p(\nu_p(g)-1)\right)^{\ell -|L\cup L_1|}\prod\limits_{\substack{p|k_L\\ p\nmid \tilde{a_1}\\p\geq 3}}\left(f_p(\nu_p(g)-1)\right)^{|L_1\backslash L|}\prod\limits_{\substack{p|k_L\\ p| \tilde{a_1}\\p\geq 3}}\left(p^{-\nu_p(g)+1}\right)^{|L_1 \backslash L|}\nonumber\\
 &=\frac{1}{\varphi(k_L)^{|L|}}\prod\limits_{\substack{p|k_L\\ p\nmid \tilde{a_1}\\p\geq 3}}\left(1+R_p(|L|)\right)\prod\limits_{\substack{p|k_L\\ p| \tilde{a_1}\\p\geq 3}}\left(1+R_p(|L\cup L_1|)\right).\label{somme_g_pasL0_beta}
\end{align}
On va ensuite exprimer $\id(2\tilde{a_1}|k)$ en décomposant $\tilde{a_1}$ en $\prod\limits_{L\in \mathscr{P}^*(E)}a_L'$ où $a_L'|\frac{k_L}{(2,k_L)}$, et en reprenant \eqref{somme_gL0_beta} et \eqref{somme_g_pasL0_beta} dans \eqref{def_Pbeta} cela donne :

\begin{equation*}\hspace*{-2cm}
 \mathcal{P}_\beta'(k)=\beta_1(k)\sum\limits_{L_0\in\mathscr{P}^*(E)}2^{-|L_0|}\sum\limits_{L_1\in \mathscr{P}^*(L_0)}\beta_6(L_1)\sum\limits_{\prod\limits_{L\in \mathscr{P}^*(E)}a_L'=\tilde{a_1}}\prod\limits_{L\in \mathscr{P}^*(E)}\beta_7(a_L')\sum\limits_{\substack{\prod\limits_{\substack{L\in \mathscr{P}^*(E)}}k_L=\frac{k}{2\tilde{a_1}}}}\prod\limits_{L\in \mathscr{P}^*(E)}\beta_8(k_L),
\end{equation*}
où

\begin{itemize}
 \item [$\bullet$]$\beta_6(L_1):=\delta_5(L_0,L_1)\mu(\tilde{a_1})^{|L_1|}\prod\limits_{\substack{p|\tilde{a_1}\\ p\geq 3}}\left(\frac{1}{p-2}\right)^{|L_1|}$,
 \item [$\bullet$]$\beta_7(a_L'):=\left(\frac{(h,a_L')}{{a_L'}{\varphi(a_L')}}\right)^{|L|}\prod\limits_{\substack{p|a_L'}}\left(\left(\frac{(p-1)^2}{p(p-2)}\right)^{|L|}\left(2-p\right)^{|L\cap L_1|}\right)\prod\limits_{\substack{p| a_L'}}\left(1+R_p(|L\cup L_1|)\right)$,
 \item [$\bullet$]$\beta_8(k_L):=\frac{1}{{k_L'}^{|L|}{\varphi(k_L)}^{|L|}}\prod\limits_{\substack{p|k_L\\ p\geq 3}}\left(\frac{(p-1)^2}{p(p-2)}\right)^{|L|}\prod\limits_{\substack{p|k_L\\p\geq 3}}\left(1+R_p(|L|)\right)$.
\end{itemize}

Définissons pour tout $L\in \mathscr{P}^*(E)$ les fonctions multiplicatives $G_L^{L_1}$ par :

\[\hspace{-2cm}G_L^{L_1}(p):=\left(\frac{(p-1)(h,p)}{p^2(p-2)}\right)^{|L|}(2-p)^{|L_1\cap L|}\left(1+R_p(|L \cup L_1|)\right),\]

et la fonction 

\begin{equation*}\hspace*{-2cm}
H_2(\ell,a_1):=2^{-\ell }\sum\limits_{L_0\in\mathscr{P}^*(E)}2^{-|L_0|}\sum\limits_{L_1\in \mathscr{P}^*(L_0)}\left(\widehat{\delta}(L_0,L_1)+\tilde{\delta}(L_0,L_1)\right)\mu(\tilde{a_1})^{|L_1|}\prod\limits_{\substack{p|\tilde{a_1}\\ p\geq 3}}\left(\frac{(p-2)^{\ell -|L_1|}}{(p-1)^{\ell }}\right)\underset{L\in \mathscr{P}^*(E)}{\mathlarger{\mathlarger{\mathlarger{\mathlarger{\ast}}}}}\left(G_L^{L_1}\right)(\tilde{a_1}).\end{equation*}

On a alors
\[\mathcal{P}_\beta'(k)=H_2(\ell,a_1)\id(2\tilde{a_1}|k)\prod\limits_{\substack{p|\frac{k}{2\tilde{a_1}}\\p\geq 3}}\left(1-\frac{1}{p-1}\right)^{\ell }\overset{\ell }{\underset{i=1}{\mathlarger{\mathlarger{\mathlarger{\mathlarger{\ast}}}}}}\left(F_i^{\ast \binom{\ell }{i}}\right)(\frac{k}{2\tilde{a_1}}).\]
\end{proof}

\subsection{Preuve de la proposition \ref{prop_terme_principal}.} Nous sommes maintenant en mesure de démontrer la proposition \ref{prop_terme_principal}.

\begin{proof}[Preuve de la proposition \ref{prop_terme_principal}]
 En utilisant les propositions \ref{prop_k_impair}, \ref{prop_Palpha}  et \ref{prop_Pbeta} on a,
 \begin{align*}
  \sum\limits_{k}\mu(k)\mathcal{P}_{a,0}'(k)&=\sum\limits_{\substack{k\\ 2\nmid k}}\mu(k)\mathcal{P}_{a,0}'(k)+\sum\limits_{\substack{k\\ 2| k}}\mu(k)\left(\mathcal{P}_{\alpha}'(k)+\mathcal{P}_{\beta}'(k)\right)\\
  &=\sum\limits_{\substack{k\\ 2\nmid k}}\mu(k)\left(\mathcal{P}_{a,0}'(k)-\mathcal{P}_{\alpha}'(x,2k)\right)+\mu(2\tilde{a_1})\sum\limits_{\substack{k\\ 2\tilde{a_1}\nmid k}}\mu(k)\mathcal{P}_{\beta}'(x,2\tilde{a_1}k)\\
 &= \prod\limits_{\substack{p\\ p\geq 3}}\left(1-\left(\frac{p-2}{p-1}\right)^{\ell }\sum\limits_{i=1}^{\ell }\binom{\ell }{i}F_i(p)\right)\\
 &\left(1-H_1(E)+\mu(2\tilde{a_1})H_2(\ell,a_1)\prod\limits_{\substack{p\\ p|a_1\\ p\geq 3}}\left(1-\left(\frac{p-2}{p-1}\right)^{\ell }\sum\limits_{i=1}^{\ell }\binom{\ell }{i}F_i(p)\right)^{-1}\right),
 \end{align*}
 ce qui donne la proposition \ref{prop_terme_principal}.
 
\end{proof}

\section{Lien entre le terme principal de la proposition \ref{prop_terme_principal} et les théorèmes \ref{thm_principal} et \ref{thm_inconditionnel}}
\subsection{Équivalence entre le terme principal de la proposition \ref{prop_terme_principal} et le résultat heuristique}
Nous allons commencer par montrer que le coefficient correspondant à un nombre premier impair $p$ dans le produit $M(E)$ de la proposition \ref{prop_terme_principal} coïncide avec celui obtenu par l'approche heuristique.
\begin{prop}\label{prop_correspondance_terme_principal}
 En reprenant les notations de la proposition \ref{prop_terme_principal}, on a
 \[\left(\frac{p-2}{p-1}\right)^\ell\sum\limits_{i=1}^\ell \binom{\ell}{i}F_i(p)=\sum\limits_{i=1}^\ell \binom{\ell}{i}\sum\limits_{j=0}^{\ell-i} \binom{\ell-i}{j}\frac{(-1)^jp^{i+j}(h,p)^i}{p^{2i}(p-1)^j(p^{i+j}-1)}=W_\ell(p),\]
 et $H_1(E)=W_\ell(2)$.
\end{prop}

\begin{proof}
On a trivialement,
\begin{align*}
 \left(\frac{p-2}{p-1}\right)^\ell\sum\limits_{i=1}^\ell \binom{\ell}{i}F_i(p)&=\sum\limits_{i=1}^\ell \binom{\ell}{i}\left(\frac{p-2}{p-1}\right)^{\ell-i} \frac{(h,p)^i}{p^{2i}}+\sum\limits_{i=1}^\ell \binom{\ell}{i}\sum\limits_{j=0}^{\ell-i} \binom{\ell-i}{j}\frac{(-1)^j(p-1)^{-j}(h,p)^i}{p^{2i}(p^{i+j}-1)}
\end{align*}

On écrit dans le premier terme, $\left(\frac{p-2}{p-1}\right)^{\ell-i}=\left(1-\frac{1}{p-1}\right)^{\ell-i}=\sum\limits_{j=0}^{\ell-i}\binom{\ell-i}{j}\frac{(-1)^j}{(p-1)^j}$. Ainsi 

\begin{align*}
 \left(\frac{p-2}{p-1}\right)^\ell\sum\limits_{i=1}^\ell \binom{\ell}{i}F_i(p)&=\sum\limits_{i=1}^\ell \binom{\ell}{i}\sum\limits_{j=0}^{\ell-i} \binom{\ell-i}{j}\frac{(-1)^j(h,p)^i}{p^{2i}(p-1)^j}\left(1+\frac{1}{p^{i+j}-1}\right)\\
 &=\sum\limits_{i=1}^\ell \binom{\ell}{i}\sum\limits_{j=0}^{\ell-i} \binom{\ell-i}{j}\frac{(-1)^jp^{i+j}(h,p)^i}{p^{2i}(p-1)^j(p^{i+j}-1)}=W_\ell(p).
\end{align*}

Montrons maintenant que $H_1(E)=W_\ell(2)$. Pour ce faire on va montrer que $Q(2):=2^\ell W_\ell(2)-\sum\limits_{i=1}^\ell \binom{\ell}{i}2^{-2i}\left(1+2^{\ell-i}\sum\limits_{j=0}^{\ell-i} \binom{\ell-i}{j}\frac{(-1)^j2^{-j}}{2^{i+j}-1}\right)=2^{-\ell}$. On a
\begin{align*}
 Q(2)&=\sum\limits_{i=1}^\ell \binom{\ell}{i}\left(\sum\limits_{j=0}^{\ell-i} \binom{\ell-i}{j}\left(\frac{(-1)^j2^\ell\left(2^{i+j}-2^{-i-j}\right)}{2^{2i}(2^{i+j}-1)}\right)-2^{-2i}\right)\\
 &=\sum\limits_{i=1}^\ell \binom{\ell}{i}\left(\sum\limits_{j=0}^{\ell-i} \binom{\ell-i}{j}\left((-1)^j\left(2^{\ell-2i}+2^{\ell-3i-j}\right)\right)-2^{-2i}\right)\\
\end{align*}
Or $\sum\limits_{j=0}^{\ell-i} \binom{\ell-i}{j}(-1)^j2^{\ell-2i}=\left\{\begin{array}{rl}
                                                                   2^{-\ell} & \text{ si } i=\ell,\\
                                                                   0&\text{ sinon.}
                                                                  \end{array}\right.
$ et $\sum\limits_{j=0}^{\ell-i} \binom{\ell-i}{j}(-1)^j2^{\ell-3i-j}=2^{-2i}$, et ainsi $Q(2)=2^{-\ell}$.
\end{proof}

\subsection{Égalité entre le terme $H_2(\ell,a_1)$ de la proposition \ref{prop_terme_principal} et de celui des théorèmes \ref{thm_principal} et \ref{thm_inconditionnel}}

Les contributions dans $H_2(\ell,a_1)$ des sous-ensembles $L_0$ et $L_1$ ne dépendant que de leurs tailles il est naturel de préférer exprimer $H_2(\ell,a_1)$ sans faire intervenir des sous-ensembles de $\mathcal{P}^*(E)$. Ce que nous faisons dans la proposition suivante.

\begin{prop}\label{prop_plus_densembles}
 Soit $H_2(\ell,a_1)$ tel que définit dans la proposition \ref{prop_terme_principal}, 
 \[H_2(\ell,a_1)=\sum\limits_{k=1}^\ell  \binom{\ell}{k} 2^{-\ell-k}\delta_\ell(k)\mu(\tilde{a_1})^{k}\prod\limits_{\substack{p|\tilde{a_1}\\ p\geq 3}}\left(\frac{(p-2)^{\ell -k}}{(p-1)^{\ell }}\right)\sum\limits_{\prod\limits_{\{i,j\}\in \mathcal{D}}a_{i,j}=\tilde{a_1}}\prod\limits_{\{i,j\}\in \mathcal{D}}G_{i,j}^{k}({a_{i,j}}),\]
 où $\mathcal{D}=[0,k]\times[0,\ell-k]\backslash\{0,0\}$, pour $(i,j)\in\mathcal{D}$, $G_{i,j}^k$ est la fonction multiplicative définie pour les nombres premiers impairs par :
 \[G_{i,j}^k(p)=\binom{k}{i}\binom{\ell-k}{j}\left(\frac{(p-1)(h,p)}{p^2(p-2)}\right)^{i+j}(2-p)^{i}\left(1+R_p(k+j)\right),\]
 et
 \[\delta_\ell(k)=\sigma(a_1,k)+2^{\ell-2k}\sum\limits_{m=0}^{\ell-k}\binom{\ell-k}{m}2^{-3m}\sum\limits_{r=0}^{\ell -k-m}\binom{\ell -k-m}{r}\frac{(-1)^r2^{-r}}{2^{k+m+r}-1},\]
 avec $\sigma(a_1,k):=\left\{\begin{array}{cl}
                          2^{-\ell+k}+2^{-2\ell+k}5^{\ell-k}& \text{ si } a_1\equiv 1\Mod{4},\\
                          (-1)^k2^{-\ell+k}+2^{-2\ell+k}5^{\ell-k}& \text{ si } a_1\equiv 3\Mod{4},\\
                          (-1)^k2^{-2\ell+k}5^{\ell-k}& \text{ sinon. }
                         \end{array}\right.
$.
\end{prop}
\begin{proof}
 Fixons $L_1$. Soient $L_\alpha$ et $L_\beta$ deux sous-ensembles non-vides de $E$. Alors $G_{L_\alpha}^{L_1}=G_{L_\beta}^{L_1}$ si et seulement si $|L_1\cap L_\alpha|=|L_1\cap L_\beta|$ et $|L_\alpha\backslash L_1|=|L_\beta\backslash L_1|$. En effet si $|L_1\cap L_\alpha|=|L_1\cap L_\beta|$ et $|L_\alpha\backslash L_1|=|L_\beta\backslash L_1|$ alors $|L_\alpha|=|L_1\cap L_\alpha|+|L_\alpha\backslash L_1|=|L_\beta|$ et $|L_1\cup L_\alpha|=|L_1|+|L_\alpha\backslash L_1|=|L_1\cup L_\beta|$.
 
 Soit $\mathcal{D}=[0,|L_1|]\times[0,\ell-|L_1|]\backslash\{0,0\}$,
 
 \[\underset{L\in \mathscr{P}^*(E)}{\mathlarger{\mathlarger{\mathlarger{\mathlarger{\ast}}}}}\left(G_L^{L_1}\right)(\tilde{a_1})=\sum\limits_{\prod\limits_{L\in \mathcal{P}^*(E)}a_{L}=\tilde{a_1}}\prod\limits_{\{i,j\}\in \mathcal{D}}\prod\limits_{\substack{L\in\mathcal{P}^*(E)\\|L_1\cap L|=i\\ |L_\backslash L_1|=j }}G_L^{L_1}(a_L).\]
 
 Soit $(i,j)\in\mathcal{D}$, notons $a_{i,j}=\prod\limits_{\substack{L\in\mathcal{P}^*(E)\\ |L\cap L_1|=i\\ |L\backslash L_1|=j}}a_L$ et $\tilde{G}_{i,j}^{L_1}={G}_{L}^{L_1}$ où $L$ est tel que $|L\cap L_1|=i$ et $ |L\backslash L_1|=j$.
 
  \begin{equation}\label{eqn_premiere_etape_convolutions_G}
\underset{L\in \mathscr{P}^*(E)}{\mathlarger{\mathlarger{\mathlarger{\mathlarger{\ast}}}}}\left(G_L^{L_1}\right)(\tilde{a_1})=\sum\limits_{\prod\limits_{L\in \mathcal{P}^*(E)}a_{L}=\tilde{a_1}}\prod\limits_{\{i,j\}\in \mathcal{D}}\tilde{G}_{i,j}^{L_1}(a_{i,j}).\end{equation}
 
 Nous allons remplacer la somme sur les décompositions de $a_1$ en $\prod\limits_{L\in \mathcal{P}^*(E)}a_{L}$ en une somme sur les décompositions de $a_1$ en $\prod\limits_{\{i,j\}\in \mathcal{D}}a_{i,j}$. Il faut donc calculer le nombre de décompositions en $\prod\limits_{L\in \mathcal{P}^*(E)}a_{L}$ qui aboutissent à une même décomposition en $\prod\limits_{\{i,j\}\in \mathcal{D}}a_{i,j}$.
 
 Soit $(i,j)\in\mathcal{D}$, $\mathcal{E}=\{L\in \mathcal{P}^*(E), |L\cap L_1|=i,\ |L\backslash L_1|=j\}$ et $\gamma=|\mathcal{E}|$. Soit $p|a_{i,j}$ alors il existe $L\in\mathcal{E}$ tel que $p|a_L$, et il y a $\gamma$ choix possible pour $L$. Ainsi le nombre décomposition en $\prod\limits_{L\in \mathcal{P}^*(E)}a_{L}$ qui correspondent à une même décomposition $\prod\limits_{\{i,j\}\in \mathcal{D}}a_{i,j}$ est $\gamma^{\omega(a_{i,j})}$.
De plus $L_1$ a $\binom{|L_1|}{i}$ sous-ensembles de taille $i$ et $E\backslash L_1$ a $\binom{|E\backslash L_1|}{j}$ sous-ensembles de taille $j$ et ainsi
\[\gamma=\binom{|L_1|}{i}\binom{|E\backslash L_1|}{j}.\]
En reportant cela dans \eqref{eqn_premiere_etape_convolutions_G}, on a :
 \[\underset{L\in \mathscr{P}^*(E)}{\mathlarger{\mathlarger{\mathlarger{\mathlarger{\ast}}}}}\left(G_L^{L_1}\right)(\tilde{a_1})=\sum\limits_{\prod\limits_{\{i,j\}\in \mathcal{D}}a_{i,j}=\tilde{a_1}}\prod\limits_{\{i,j\}\in \mathcal{D}}G_{i,j}^{|L_1|}({a_{i,j}}),\]
 avec pour $k\leq \ell$, $i\leq k$ et $j\leq \ell-k$, $G_{i,j}^k$ est la fonction multiplicative telle que pour $p$ un nombre premier impair :
 \[G_{i,j}^k(p)=\binom{k}{i}\binom{\ell-k}{j}\left(\frac{(p-1)(h,p)}{p^2(p-2)}\right)^{i+j}(2-p)^{i}\left(1+R_p(k+j)\right).\]
 
 Nous pouvons maintenant remplacer les sommes sur $L_0$ et $L_1$ dans $H_2(\ell,a_1)$ par des sommes sur $1\leq m\leq \ell$ et $1\leq k\leq m$. $L_0$ a $\binom{|L_0|}{k}$ sous-ensembles de taille $k$ et $E$ a $\binom{\ell}{m}$ sous-ensembles de taille $m$. Ainsi en notant de manière transparente $\widehat{\delta}(m,k):=\widehat{\delta}(L_0,L_1)$ et $\tilde{\delta}(m,k):=\tilde{\delta}(L_0,L_1)$ avec $|L_0|=m$ et $|L_1|=k$, on obtient :
 
 \begin{equation}\label{eqn_H2_sans_ensembles}\hspace*{-2.5cm}
  H_2(\ell,a_1):=\sum\limits_{m=1}^\ell \binom{l}{m} 2^{-\ell-m}\sum\limits_{k=1}^m \binom{m}{k}\left(\widehat{\delta}(m,k)+\tilde{\delta}(m,k)\right)\mu(\tilde{a_1})^{k}\prod\limits_{\substack{p|\tilde{a_1}\\ p\geq 3}}\left(\frac{(p-2)^{\ell -k}}{(p-1)^{\ell }}\right)\sum\limits_{\prod\limits_{\{i,j\}\in \mathcal{D}}a_{i,j}=\tilde{a_1}}\prod\limits_{\{i,j\}\in \mathcal{D}}G_{i,j}^{k}({a_{i,j}}).
 \end{equation}
Inversons alors les sommes sur $m$ et $k$,
\begin{equation}\label{eqn_H2_sans_m}
 H_2(\ell,a_1):= \sum\limits_{k=1}^\ell  \binom{\ell}{k} 2^{-\ell-k}\delta_\ell(k)\mu(\tilde{a_1})^{k}\prod\limits_{\substack{p|\tilde{a_1}\\ p\geq 3}}\left(\frac{(p-2)^{\ell -k}}{(p-1)^{\ell }}\right)\sum\limits_{\prod\limits_{\{i,j\}\in \mathcal{D}}a_{i,j}=\tilde{a_1}}\prod\limits_{\{i,j\}\in \mathcal{D}}G_{i,j}^{k}({a_{i,j}}).
\end{equation}
avec $\delta_\ell(k):=\sum\limits_{m=0}^{\ell-k}\binom{\ell-k}{m} 2^{-m}\left(\widehat{\delta}(m+k,k)+\tilde{\delta}(m+k,k)\right)$.

Simplifions maintenant l'écriture de $\delta_\ell(k)$. Notons $c:=\left\{\begin{array}{cl}
                                                                   (-1)^k & \text{ si } a_1 \equiv 0\Mod{2},\\
                                                                   1 & \text{ sinon.}
                                                                  \end{array}\right.$.
\begin{align*}\hspace*{-2cm}
 \sum\limits_{m=0}^{\ell-k}\binom{\ell-k}{m} 2^{-m}\widehat{\delta}(m+k,k)&= \sum\limits_{m=0}^{\ell-k}\binom{\ell-k}{m}2^{-2m-k}\left(c+2^{\ell -k-m}\sum\limits_{r=0}^{\ell -k-m}\binom{\ell -k-m}{r}\frac{(-1)^r2^{-r}}{2^{k+m+r}-1}\right)\\
 &=c2^{-2\ell+k}5^{\ell-k}+2^{\ell-2k}\sum\limits_{m=0}^{\ell-k}\binom{\ell-k}{m}2^{-3m}\sum\limits_{r=0}^{\ell -k-m}\binom{\ell -k-m}{r}\frac{(-1)^r2^{-r}}{2^{k+m+r}-1}.
\end{align*}
De plus
\[\sum\limits_{m=0}^{\ell-k}\binom{\ell-k}{m} 2^{-m}\tilde{\delta}(m+k,k)=\left\{\begin{array}{cl}
                                                                           2^{-\ell+k}& \text{ si } a_1\equiv 1\Mod{4},\\
                                                                           (-1)^k2^{-\ell+k}& \text{ si } a_1\equiv 3\Mod{4},\\
                                                                           0& \text{ sinon. }
                                                                          \end{array}\right.
\]

Ainsi, en notant $\sigma(a_1,k):=\left\{\begin{array}{cl}
                          2^{-\ell+k}+2^{-2\ell+k}5^{\ell-k}& \text{ si } a_1\equiv 1\Mod{4},\\
                          (-1)^k2^{\ell-k}+2^{-2\ell+k}5^{\ell-k}& \text{ si } a_1\equiv 3\Mod{4},\\
                          (-1)^k2^{-2\ell+k}5^{\ell-k}& \text{ sinon. }
                         \end{array}\right.
$,
\[\delta_\ell(k)=\sigma(a_1,k)+2^{\ell-2k}\sum\limits_{m=0}^{\ell-k}\binom{\ell-k}{m}2^{-3m}\sum\limits_{r=0}^{\ell -k-m}\binom{\ell -k-m}{r}\frac{(-1)^r2^{-r}}{2^{k+m+r}-1}.\]
Ce qui permet de conclure.
\end{proof}

\subsection{Preuve des théorèmes \ref{thm_principal} et \ref{thm_inconditionnel}}
En regroupant l'équation \eqref{equationfondav2} et les propositions \ref{prop_majoration_termes_derreurs}, \ref{prop_terme_principal}, \ref{prop_correspondance_terme_principal}, \ref{prop_plus_densembles} et en prenant, par exemple, $C_5=e^{2C_1}$, on obtient le théorème \ref{thm_principal}.

De plus en reportant les résultats des propositions \ref{prop_majoration_termes_derreurs}, \ref{prop_terme_principal}, \ref{prop_correspondance_terme_principal}, \ref{prop_plus_densembles} dans \eqref{majoration_inconditionnelle} on obtient la majoration inconditionnelle du théorème \ref{thm_inconditionnel}.

\bibliographystyle{abbrv}
\bibliography{biblio}

\end{document}